\documentclass{amsart}

\usepackage[a4paper, margin=2.5cm]{geometry}
\usepackage[foot]{amsaddr}

\usepackage[T1]{fontenc}
\usepackage{graphicx}%
\usepackage{subcaption}
\usepackage{multirow}%
\usepackage{amsmath,amssymb,amsfonts}%
\usepackage{amsthm}%
\usepackage{mathrsfs}%
\usepackage{mathtools}%
\usepackage[title]{appendix}%
\usepackage{xcolor}%
\usepackage{textcomp}%
\usepackage{manyfoot}%
\usepackage{booktabs}%
\usepackage{listings}%
\usepackage{nicematrix}
\usepackage{microtype}
\usepackage[scr=boondox,  
            cal=esstix]   
           {mathalpha}

\usepackage{tikz}
\usetikzlibrary{tikzmark}
\usetikzlibrary{patterns, calc}
\usetikzlibrary{positioning,matrix,fit}
\usetikzlibrary{shapes,arrows}
\usetikzlibrary{decorations.pathreplacing,angles,quotes}
\usetikzlibrary{shapes.geometric, arrows}
 \usetikzlibrary{decorations.pathmorphing}
\usepackage{tikz-cd}

\tikzset{
  commutative diagrams/.cd, 
  arrow style=tikz, 
  diagrams={>=stealth},
  ampersand replacement=\&
}

\usepackage{hyperref}    
\usepackage[nameinlink]{cleveref}    

\def\R{\mathbb{R}}
\newcommand{\longto}{\longrightarrow}
\DeclareMathOperator{\grad}{grad}
\DeclareMathOperator{\curl}{curl}

\def\divgn{\operatorname{div}}
\DeclareMathOperator{\rot}{rot}
\newcommand{\ainnerproduct}[2]{\langle #1, #2 \rangle}
\newcommand{\aInnerproduct}[2]{\bigl\langle #1, #2 \bigr\rangle}
\DeclarePairedDelimiterX{\norm}[1]{\lVert}{\rVert}{#1}

\makeatletter
\@namedef{subjclassname@2020}{AMS Subject Classification}
\makeatother

\allowdisplaybreaks

\newtheorem{theorem}{Theorem}
\newtheorem{lemma}[theorem]{Lemma}

\newtheorem{assumption}{Assumption}

\numberwithin{equation}{section}

\usepackage[backend=biber, style=alphabetic, backref, url=false, eprint=false]{biblatex}
\begin{document}


\title[Two Frameworks and Fourth-Order Implicit Schemes for Maxwell's Equations]{Two Frameworks and their Fourth Order Implicit Schemes for Time Discretization of Maxwell's Equations}






\author{Archana Arya and Kaushik Kalyanaraman}
\address{Department of Mathematics, Indraprastha Institute of Information Technology, Delhi, New Delhi, 110020, India}
\email{archanaa@iiitd.ac.in, kaushik@iiitd.ac.in}

\begin{abstract}

  Our work is about energy conserving fourth-order time discretizations of a three-field formulation of Maxwell’s equations in conjunction with a spatial discretization using higher-order and compatible de Rham finite element spaces. Toward this end, we delineate two broad classes of strategies for general higher-order time discretizations which we term spatial and temporal strategies. We provide a description of these two strategies and develop fourth-order time accurate schemes in the context of our Maxwell's system. However, our description can be used to prescribe similar fourth- or even higher-order time-integration methods for any linear (or quasi-linear) system of time-dependent partial differential equations. Our organizing principle in our proposed two strategies is to Taylor expand the unknown solution in time by assuming sufficient regularity. Then, in the spatial strategy, we use Maxwell's equations themselves to replace the fourth-order time derivatives in an appropriately truncated Taylor expansion with corresponding higher-order spatial derivatives. On the other hand, in the temporal strategy, we simply use higher-order finite difference schemes for the various higher-order time derivative terms in the truncated Taylor approximation. In both cases, we then defer to a standard finite element exterior calculus manner of compatible discretization for the spatial component of the Maxwell's solution. For our proposed schemes corresponding to the two strategies, we show that they are both stable and convergent and provide some validating numerical examples in $\mathbb{R}^2$. Our main contributions are in the development of the fourth-order time discretization methods that are energy conserving using our two outlined strategies and proofs of their convergence for semi- and full-discretizations of our three-field system of Maxwell's equations.

\end{abstract}

\keywords{Error analysis; finite element exterior calculus; higher-order methods; implicit schemes; Maxwell's equations; structure preservation}

\subjclass{35Q61, 65M06, 65M12, 65M15, 65M22, 65M60, 65Z05, 78M10}

\maketitle

\section{Introduction and Preliminaries} \label{sec:introduction}

In this work, we provide fourth-order implicit time discretizations of Maxwell's equations in a three-field formulation. Our original intent was to develop an arbitrary-order time discretization for this Maxwell's system in conjunction with a compatible spatial discretization of the field variables using finite element exterior calculus (FEEC)~\cite{Arnold2018}. However, our attempt at extending an earlier work~\cite{ArKa2025} from its second-order accurate time discretization to merely a fourth-order accurate one led us to through not simply a complicated set of analysis which we believe warrants its own description that we provide hereto, but also that it helped us categorize into two classes a recipe for development of such methods. Our prescriptions therein then helped us methodically concoct these two energy conserving and implicit fourth-order in time discretization of the three-field Maxwell's system. We term these two methodologies as being a \textit{spatial} or a \textit{temporal} strategy. In our work here, we describe these two frameworks, provide the corresponding fourth-order time discretization schemes they yield, carry out their respective error analyses for both time semi-discretization and for the full error alongside a FEEC discretization of the spatial solution.

The inspirations for our two frameworks derive from~\cite{Fahs2009} for the spatial strategy and~\cite{BrTuTs2018} for the temporal strategy. Our main idea in writing out a fourth-order time discretization is rather straightforward: express the time component of the solution for each of the field variables in the Maxwell's system using a Taylor expansion (with the assumption that it is sufficiently regular) around a generic discretized time, and then use the Taylor series for deriving a higher-order and suitably truncated discretization of the first-order time derivatives in the Maxwell's equations. This requires a meaningful way to approximate each of the higher order time derivatives in the Taylor expansion and our alternatives are to: 
\begin{enumerate} 
\item either use the Maxwell's equations themselves repeatedly to switch all these time derivatives to higher order spatial derivatives, or
\item directly discretize the higher-order time derivative terms,
\end{enumerate}
in the Taylor approximation of the solution. This can be enjoined with a spatial discretization by way of, as we carry out here, finite element methods to arrive at a numerical solution for the Maxwell's equations. An earlier work~\cite{ArKa2025} already provides two flavors of second-order time implicit discretization that is also energy conserving and we seek their corresponding fourth-order generalizations here. \cite{Fahs2009} set out to provide a strategy for such a solution scheme in the context of discontinuous Galerkin methods by transforming the higher-order time derivatives to higher-order spatial derivatives. This idea is superficially similar to arbitrary-order non-oscillatory solutions for the advection schemes~\cite{MiTiTo2001}. Several of Toro's works stem from this viewpoint, however, the similarity ends merely at this idea of recasting time derivative into spatial derivatives. On the other hand, we merely carry forward Fahs's particularly simple ansatz, in itself originally credited to~\cite{Young2001} by Fahs, to its logical conclusion. This paves the way for our spatial strategy discretization that we so name because the temporal higher-order derivatives are replaced by spatial ones. On the other hand, Britt and coauthors in~\cite{BrTuTs2018} provide us with our idea for the temporal strategy, namely, that we just replace the higher-order time derivatives in a truncated Taylor approximation with suitably appropriate order finite differencing schemes. Again, using this principle, we can massage out a fourth-order time discretization that is implicit and with an energy conservation. Our underlying idea also somewhat superficially resembles or reminds one of the Cauchy-Kowalewskaya procedure which is a higher-order numerical technique in the context of generalized Riemann problems; see, for example, \cite{MoBa2020}. The similarity here is that we also transform a single time derivative in a higher-order time derivative term first by a spatial derivative from Maxwell's equation. We then diverge in our formalism by factoring the spatial derivatives through the time differentiations (given their mutual independence), and in temporal strategy finally approximate these using suitably higher-order finite difference schemes.

Our main contribution in this work is outlining and analyzing these two different energy conserving and fourth-order in time discretizations of the three-field formulation of the Maxwell's equations, showing that our proposed schemes are stable and converging in time at the correct rate of four as well as analyzing its convergence in conjunction with a compatible FEEC discretization of this Maxwell's system.

The organization of the rest of our article is as follows. In the remainder of this introductory section, we provide a brief overview of the three-field formulation of the Maxwell's equations, the two proposed schemes, a compressed and just necessary background on some of the analysis tools we use, and a formal description of the two frameworks which we have categorized for deriving such higher-order schemes. Sections~\labelcref{sec:implicit_lf4,sec:ts4} provide the detailed stability and error analysis of our two proposed schemes. Finally, in Section~\labelcref{sec:numerics}, we provide some validating numerical computations using our two schemes along with a discussion and future outlook.

\subsection{Three-field Formulation of Maxwell's Equations}

A three-field $(p, E, H)$-formulation of the Maxwell's equations consists of the governing electromagnetic equations as is usual describing the time evolution of the electric and magnetic fields $E$ and $H$ as well as an additional time-varying electrical pressure field $p$. In this formulation, $p$ may be regarded as being physically fictitious but provides a set up for a compatible (and exact) discretization of $\divgn E$. This three-field formulation is as follows: \begin{subequations}
  \begin{equation}
    \left. \begin{aligned}
      \dfrac{\partial p}{\partial t} + \nabla \cdot \varepsilon E & = 0, \\
      \nabla p + \varepsilon \dfrac{\partial E}{\partial t} - \nabla \times H &= 0, \\
      \mu \dfrac{\partial H}{\partial t} + \nabla \times E &= 0,
    \end{aligned} \right\} \text{ in } \Omega \times (0, T], \label{eqn:maxwells_eqns}
  \end{equation}
  where $\Omega \subset \R^2/\R^3$ is a domain with Lipschitz boundary $\partial \Omega$, and $T > 0$. We use the following homogeneous boundary conditions:
  \begin{equation}
    p = 0,  E \times n = 0, H \cdot n = 0 \text{ on } \partial\Omega \times (0, T], \label{eqn:BCs}
  \end{equation}
  where $n$ is the unit outward normal to $\partial \Omega$, and with the following initial conditions:
  \begin{equation}
    p(x, 0) = p_0(x), E(x, 0) = E_0(x), \text{ and } H(x, 0) = H_0(x) \text{~for~} x \in \Omega, \label{eqn:ICs}
  \end{equation}
\end{subequations}
Furthermore, the initial conditions are required to satisfy $\nabla \cdot (\varepsilon E_0) = p_0 $ and $\nabla \cdot (\mu H_0) = 0$ in $\Omega$ for consistency with the time evolution. This formulation has already been studied in earlier works such as~\cite{AdPeZi2013,AdHuZi2017,AdCaHuZi2021}, and in particular we are continuing with the development of, and by suitably generalizing, the schemes for it and their analysis as in~\cite{ArKa2025}. Consequently, all of our background in this work is already well stated in~\cite{ArKa2025} and we only provide here some essential details. As an example, we wish to restate here again that the homogeneous boundary conditions is merely for simplicity of our analysis and all our results can be extended to the non homogeneous case in a standard manner. We also wish to add that it is possible to extend our method and analysis to nonzero right-hand side functions in Equation~\eqref{eqn:maxwells_eqns} but doing so will further complicate the analysis without providing any meaningful additional insight beyond what we state in this work and prove for the zero right-hand side functions. A nonzero right-hand side for the Maxwell's system can be discretized using our principles in a straightforward manner.

The variational formulation for Equations~\eqref{eqn:maxwells_eqns} to~\eqref{eqn:ICs} is as following: for $t \in (0, T]$, find $(p(t), E(t), H(t)) \in \mathring{H}^1_{\varepsilon^{-1}}(\Omega) \times \mathring{H}_{\varepsilon}(\curl; \Omega) \times \mathring{H}_{\mu}(\divgn; \Omega)$ such that: \begin{subequations}
  \begin{alignat}{2}
    \aInnerproduct{\dfrac{\partial p}{\partial t}}{\widetilde{p}} - \aInnerproduct{ \varepsilon E}{\nabla \widetilde{p}} &=0, &&\quad \widetilde{p} \in \mathring{H}^1_{\varepsilon^{-1}}(\Omega), \label{eqn:maxwell_p_wf} \\
    \aInnerproduct{\nabla p}{\widetilde{E}} + \aInnerproduct{\varepsilon \dfrac{\partial E}{\partial t}}{\widetilde{E}} - \aInnerproduct{H}{\nabla \times \widetilde{E}} &= 0, &&\quad \widetilde{E} \in \mathring{H}_{\varepsilon}(\curl; \Omega), \label{eqn:maxwell_E_wf} \\
    \aInnerproduct{\mu \dfrac{\partial H}{\partial t}}{\widetilde{H}} + \aInnerproduct{\nabla \times E}{\widetilde{H}}, &= 0, &&\quad \widetilde{H} \in \mathring{H}_{\mu}(\divgn; \Omega), \label{eqn:maxwell_H_wf}
  \end{alignat}
\end{subequations}
The stability and uniqueness of solution for this variational formulation can be found in~\cite[Theorem 3.1, Remark 3.1]{ArKa2025}. As with the description therein, here too we assume that the solution $(p, E, H)$ to the variational formulation has sufficient regularity. 

With this brief background, we present below our two fourth-order time implicit discretizations. We denote the first scheme as LF$_4$ and it is derived using our spatial strategy. LF$_4$ is a generalization of the second-order time implicit LF$_2$ scheme of~\cite[Equations 1.4(a)--(c)]{ArKa2025}. We label our second scheme as TS$_4$ and it is derived using our temporal strategy. TS$_4$ should be reminiscent of a fourth-order Crank Nicholson scheme, for example, as in~\cite[Equation 6]{BrTuTs2018}.

To specify these schemes, we fix some initial notation. We consider time $[0, T]$ to be uniformly discretized with a timestep $\Delta t > 0$. Then, $t^n \coloneq n \Delta t$ or $t^{n + \frac{1}{2}} \coloneq (n + 1/2) \Delta t$ for $n = 0, 1, \dots, N$ where $N \Delta t \approx T$. We also set $u^n \coloneq u(x, t^n)$ (and likewise for $u^{n + \frac{1}{2}}$) for $x \in \Omega$ where $u = \{p, E, H\}$ as appropriate.

\subsection*{LF$_4$ scheme} Find $(p(t^{n + 1/2}), E(t^{n + 1/2}), H(t^n)) \in \mathring{H}^1_{\varepsilon^{-1}}(\Omega) \times \mathring{H}_{\varepsilon}(\curl; \Omega) \times \mathring{H}_{\mu}(\divgn; \Omega)$ such that:
 \begin{subequations}
 \begin{equation}
   \aInnerproduct{\dfrac{p^{n + \frac{1}{2}} - p^{n - \frac{1}{2}}}{\Delta t}}{\widetilde{p}} - \aInnerproduct{\dfrac{\varepsilon}{2} \left( E^{n + \frac{1}{2}} + E^{n - \frac{1}{2}} \right)}{\nabla \widetilde{p}}  + \dfrac{\Delta t^2}{12} \aInnerproduct{\dfrac{\varepsilon}{2} \nabla \nabla \cdot \left( E^{n + \frac{1}{2}} + E^{n - \frac{1}{2}} \right)}{\nabla \widetilde{p}} =0, \label{eqn:maxwell_p_lf4}
   \end{equation}\\
   \begin{multline}
    \aInnerproduct{\dfrac{1}{2} \nabla \left(p^{n + \frac{1}{2}} + p^{n - \frac{1}{2}} \right)}{\widetilde{E}} - \dfrac{\Delta t^2}{12} \aInnerproduct{\dfrac{1}{2} \nabla \left(p^{n + \frac{1}{2}} + p^{n - \frac{1}{2}} \right)}{\nabla \nabla \cdot \widetilde{E}} + \aInnerproduct{\varepsilon \dfrac{E^{n + \frac{1}{2}} - E^{n - \frac{1}{2}}}{\Delta t}}{\widetilde{E}} \\ - \aInnerproduct{\dfrac{1}{2} \left( H^{n + 1} + H^n \right)}{\nabla \times \widetilde{E}}  - \dfrac{\Delta t^2}{12} \aInnerproduct{\dfrac{1}{2} \mu^{-1}\varepsilon^{-1} \nabla \times \nabla \times \left( H^{n + 1} + H^n \right)}{\nabla \times \widetilde{E}} = 0, \label{eqn:maxwell_E_lf4} 
      \end{multline} \\
      \begin{multline}
   \aInnerproduct{\mu \dfrac{H^{n + 1} - H^n}{\Delta t}}{\widetilde{H}} +  \aInnerproduct{\dfrac{1}{2} \nabla \times \left( E^{n + \frac{1}{2}} + E^{n - \frac{1}{2}} \right)}{\widetilde{H}} \\+  \dfrac{\Delta t^2}{12} \aInnerproduct{\dfrac{1}{2} \varepsilon^{-1}\mu^{-1} \nabla \times \left( E^{n + \frac{1}{2}} + E^{n - \frac{1}{2}} \right)}{\nabla \times  \nabla \times \widetilde{H}} = 0, \label{eqn:maxwell_H_lf4}
 \end{multline}
 \end{subequations}
 for all $\widetilde{p} \in \mathring{H}^1_{\varepsilon^{-1}}(\Omega), \; \widetilde{E} \in \mathring{H}_{\varepsilon}(\curl; \Omega)$ and $\widetilde{H} \in \mathring{H}_{\mu}(\divgn; \Omega)$. In order to bootstrap the computations from the given initial conditions, we will use the following restricted version of the scheme for the first half time step for $p$ and $E$, and for the first time step for $H$:
 \begin{subequations}
   \begin{equation}
     \aInnerproduct{\dfrac{p^{\frac{1}{2}} - p_0}{\Delta t/2}}{\widetilde{p}} - \dfrac{1}{2} \aInnerproduct{\dfrac{\varepsilon}{2} \left( E^{\frac{1}{2}} + E_0 \right)}{\nabla \widetilde{p}} + \dfrac{1}{8} \cdot \dfrac{\Delta t^2}{12} \aInnerproduct{\dfrac{\varepsilon}{2} \nabla \nabla \cdot \left( E^{\frac{1}{2}} + E_0 \right)}{\nabla \widetilde{p}} = 0, \label{eqn:maxwell_p0_lf4}
     \end{equation} \\
     \begin{multline}
   \dfrac{1}{2}  \aInnerproduct{\dfrac{1}{2} \nabla \left(  p^{\frac{1}{2}} +  p_0 \right)}{\widetilde{E}} - \dfrac{1}{8} \cdot \dfrac{\Delta t^2}{12} \aInnerproduct{\dfrac{1}{2} \nabla \left(  p^{\frac{1}{2}} +  p_0 \right)}{\nabla \nabla \cdot \widetilde{E}} + \aInnerproduct{\varepsilon \dfrac{E^{\frac{1}{2}} - E_0}{\Delta t/2}}{\widetilde{E}}\\  - \aInnerproduct{\dfrac{1}{2} \left( H^1 + H_0 \right)}{\nabla \times \widetilde{E}} -\dfrac{1}{4} \cdot \dfrac{\Delta t^2}{12}  \aInnerproduct{\dfrac{1}{2} \mu^{-1}\varepsilon^{-1} \nabla \times \nabla \times \left( H^1 + H_0 \right)}{\nabla \times \widetilde{E}} = 0, \label{eqn:maxwell_E0_lf4} 
   \end{multline} \\
   \begin{multline}
     \aInnerproduct{\mu \dfrac{H^1 - H_0}{\Delta t}}{\widetilde{H}} + \dfrac{1}{2} \aInnerproduct{\dfrac{1}{2} \nabla \times \left(E^{\frac{1}{2}} + E_0 \right)}{\widetilde{H}} + \dfrac{1}{8} \cdot \dfrac{\Delta t^2}{12} \aInnerproduct{\dfrac{1}{2} \varepsilon^{-1}\mu^{-1} \nabla \times \left(E^{\frac{1}{2}} + E_0 \right)}{ \nabla \times \nabla \times  \widetilde{H}} = 0. \label{eqn:maxwell_H0_lf4}
   \end{multline}
 \end{subequations}

 \subsection*{TS$_4$ scheme} 
 
 Find $(p(t^{n + 1/2}), E(t^{n + 1/2}), H(t^n)) \in \mathring{H}^1_{\varepsilon^{-1}}(\Omega) \times \mathring{H}_{\varepsilon}(\curl; \Omega) \times \mathring{H}_{\mu}(\divgn; \Omega)$ such that:
    \begin{subequations}
    \begin{equation}
   \aInnerproduct{\varepsilon^{-1} \dfrac{p^{n + 1} - p^{n - 1}}{2 \Delta t}}{\widetilde{p}} - \dfrac{1}{6} \aInnerproduct{E^{n + 1} +4 E^n  + E^{n - 1}}{\nabla \widetilde{p}} = 0,  \label{eqn:maxwell_p_ts4}
   \end{equation}
\begin{equation}
   \dfrac{1}{6} \aInnerproduct{\nabla \left(p^{n + 1} + 4 p^n  + p^{n - 1} \right)}{\widetilde{E}}  + \aInnerproduct{\varepsilon \dfrac{E^{n + 1} - E^{n - 1}}{2 \Delta t}}{\widetilde{E}} - \dfrac{1}{6} \aInnerproduct{\nabla \times \left(H^{n + 1} + 4 H^n  + H^{n - 1} \right)}{\widetilde{E}} = 0,  \label{eqn:maxwell_E_ts4}
\end{equation}
\begin{equation}
  \aInnerproduct{\mu \dfrac{H^{n + 1} - H^{n - 1}}{2 \Delta t}}{\widetilde{H}} +  \dfrac{1}{6} \aInnerproduct{\nabla \times \left( E^{n + 1} + 4 E^n  + E^{n - 1} \right)}{\widetilde{H}} = 0.  \label{eqn:maxwell_H_ts4}
\end{equation}
\end{subequations}
 for all $\widetilde{p} \in \mathring{H}^1_{\varepsilon^{-1}}(\Omega), \; \widetilde{E} \in \mathring{H}_{\varepsilon}(\curl; \Omega)$ and $\widetilde{H} \in \mathring{H}_{\mu}(\divgn; \Omega)$. Note that this scheme requires no bootstrapping and computations can begin from the provided initial conditions.

\subsection{A Few Notes on Finite Element Spaces and Other Preliminaries}

We refer to \cite[Section 2]{ArKa2025} for a more detailed set of background for the function spaces, the de Rham sequence of finite-dimensional subspaces of these function spaces for the FEEC manner of compatible spatial discretization of the solution, and for a list of important inequalities that feature repeatedly in our analysis. However, for brevity yet completeness, we highlight some essential background here. 

To begin with, the diagrams shown next depict in $\R^2$ and $\R^3$, respectively, the de Rham complexes of function spaces for a compatible discretization of Maxwell's equations, and their finite dimensional subspaces spanned by arbitrary order polynomial functions (scalar or vector). Later on in this work, we will refer to these finite dimensional subspaces as $U_h$, $V_h$ and $W_h$.

\begin{center}
  \begin{tikzpicture}[baseline=(a).base]
    \node[scale=.9] (a) at (0,0){
      \begin{tikzcd}[column sep=3em, row sep=2em, text height=1ex, text depth=0ex, ampersand replacement=\&, every
        label/.style={font=\small, auto}]
        \mathring{H}_{\varepsilon^{-1}}^1 (\Omega) \arrow[r, -stealth, shift left, "\rot"] \arrow[d, -stealth, "\Pi_h^0"] \& \arrow[l, -stealth, shift left, "\curl"] \mathring{H}_{\varepsilon} (\divgn, \Omega) \arrow[r, -stealth, shift left, "-\divgn"] \arrow[d, -stealth, "\Pi_h^1"] \& \arrow[l, -stealth, shift
        left, "\grad"] L_{\mu}^2(\Omega) \arrow[d, -stealth, "\Pi_h^2"]  \\[2em] 
        \arrow[u, hookrightarrow, shift left, yshift = .1em] \mathring{H}_{\varepsilon^{-1}}^1(\Omega) \hspace{.5em}  \arrow[r, -stealth, shift left, yshift = -.7em, "\rot"]  \& \arrow[l, -stealth, shift left, yshift = -.7em, "\curl"] \arrow[u, hookrightarrow, shift left, yshift = .1em] \mathring{H}_{\varepsilon} (\divgn, \Omega) \arrow[r, -stealth, shift
        left, yshift = -.7em, "-\divgn"] \& \arrow[l, -stealth, shift
        left, yshift = -.7em, "\grad"] \arrow[u, hookrightarrow, shift left, yshift = .1em] \hspace{.5em} L_{\mu}^2(\Omega) \\[-2em]
        \cap \, \mathcal{P}_r^-(\Omega) \hspace{.5em}\& \cap \, \mathcal{P}_r^-(\Omega) \& \hspace{.5em} \cap \, \mathcal{P}_r^-(\Omega) 
      \end{tikzcd}
    };
  \end{tikzpicture}
\end{center}

\begin{center}
  \begin{tikzpicture}[baseline=(a).base]
    \node[scale=.9] (a) at (0,0){
      \begin{tikzcd}[column sep=3em, row sep=2em, text height=1ex, text depth=0ex, ampersand replacement=\&, every
        label/.style={font=\small, auto}]
        \mathring{H}_{\varepsilon^{-1}}^1 (\Omega) \arrow[r, -stealth, shift left, "\grad"] \arrow[d, -stealth, "\Pi_h^0"] \& \arrow[l, -stealth, shift left, "-\divgn"] \mathring{H}_{\varepsilon} (\curl, \Omega) \arrow[r, -stealth, shift left, "\curl"]  \arrow[d, -stealth, "\Pi_h^1"] \& \arrow[l, -stealth, shift left, "\curl"] \mathring{H}_{\mu} (\divgn, \Omega) \arrow[r, -stealth, shift left, "-\divgn"]  \arrow[d, -stealth, "\Pi_h^2"] \& \arrow[l, -stealth, shift left, "\grad"] L^2(\Omega) \arrow[d, -stealth, "\Pi_h^3"] \\[2em]
       \arrow[u, hookrightarrow, shift left, yshift = .1em]  \mathring{H}_{\varepsilon^{-1}}^1(\Omega) \hspace{.5em} \arrow[r, -stealth, shift left, yshift = -.7em, "\grad"] \& \arrow[l, -stealth, shift left, yshift = -.7em, "-\divgn"] \arrow[u, hookrightarrow, shift left, yshift = .1em] \mathring{H}_{\varepsilon} (\curl, \Omega) \arrow[r, -stealth, yshift = -.7em, shift left, "\curl"]  \& \arrow[l, -stealth, yshift = -.7em, shift left, "\curl"] \arrow[u, hookrightarrow, shift left, yshift = .1em] \mathring{H}_{\mu} (\divgn, \Omega) \arrow[r, -stealth, yshift = -.7em, shift left, "-\divgn"]  \& \arrow[l, -stealth, yshift = -.7em, shift left, "\grad"]  \arrow[u, hookrightarrow, shift left, yshift = .1em] \hspace{.5em} L^2(\Omega) \\[-2em]
        \cap \, \mathcal{P}_r^-(\Omega) \hspace{.5em}\& \cap \, \mathcal{P}_r^-(\Omega) \& \cap \, \mathcal{P}_r^-(\Omega) \& \hspace{.5em} \cap \, \mathcal{P}_r^-(\Omega)
      \end{tikzcd}
    };
  \end{tikzpicture}
\end{center}

The singularly crucial aspect of FEEC that we use is the following theorem due to Arnold, Falk and Winther~\cite{ArFaWi2006} that characterizes the approximation properties of these finite dimensional subspaces under a so-called smoothed $L^2$ projection operation from their respective smooth function spaces. We will rely on this result and adapt it for all our full error analysis for the LF$_4$ and TS$_4$ schemes. In performing such analysis, we shall always be assuming sufficient regularity for the solution of the three-field Maxwell's system. For clarity and concreteness, we wish to expound on this a bit and that it means if we use degree $r \ge 1$ polynomial subspaces, the solutions are construed to be in $\mathring{H}^1_{\varepsilon^{-1}}(\Omega) \cap H^r(\Omega)$ for the scalar field, and $\mathring{H}_{\varepsilon}(\curl; \Omega) \cap [H^r(\Omega)]^d$ or $\mathring{H}_{\mu}(\divgn; \Omega) \cap [H^r(\Omega)]^d$ for the vector fields where $d = 2, 3$ as the case may be. Here, $[H^r(\Omega)]^d$ denotes the class of component-wise $H^r(\Omega)$ vector fields.

\begin{theorem}~\cite[Theorem 5.3]{ArFaWi2006}
Denote by $\Pi_h$ the canonical projection of $\Lambda^k(\Omega)$ onto either $\mathcal{P}_r \Lambda^k(\mathcal{T}_h)$ or $\mathcal{P}_{r+1}^- \Lambda^k(\mathcal{T}_h)$. Let $1 \leq p \leq \infty$ and $(n-k)/p < s \leq r+1$. Then $\Pi_h$ extends boundedly to $W_p^s \Lambda^k(\Omega)$, and there exists a constant $c$ independent of $h$, such that 
\[
\norm{\omega - \Pi_h \omega}_{L^p\Lambda^k(\Omega)} \leq C h^s |\omega|_{W_p^s \Lambda^k(\Omega)}, \quad \omega \in W_p^s \Lambda^k(\Omega).
\]
\label{thm:arfawismoothedprojection}
\end{theorem} 

Another recurring result in our various analysis of stability and error convergence for the time discretizations is Gronwall's inequality. We state below a discrete version of it that is most pertinent to how we wield Gronwall's inequality. 

\begin{lemma}[Discrete Gronwall, {\cite[Lemma 2]{AnAn2019}}] \label{lemma:gronwall_dscrt}
  Let $\delta \ge 0$, $g_0 \ge 0$ and $(a_n)$, $(b_n)$, $(c_n)$ and $(\gamma_n)$ be sequences of nonnegative numbers such that:
  \[
    a_N + \delta \sum\limits_{n = 0}^N b_n \le \delta \sum\limits_{n = 0}^{N} \left( \gamma_n a_n + c_n \right) + g_0 \text{ for all } N = 0, 1, \dots.
  \]
  Assuming that $\gamma_n \delta < 1$ for all $n$ and setting $\sigma_n \coloneq \left( 1 - \gamma_n \delta \right)^{-1}$, for all $N \ge 0$ we have that:
  \[
    a_N + \delta \sum\limits_{n = 0}^N b_n \le \left( \delta \sum\limits_{n = 0}^N c_n + g_0 \right) \exp \left( \delta \sum\limits_{n = 0}^N \sigma_n \gamma_n \right).
  \]
\end{lemma}

In addition to these, we shall variously use other standard inequalities such as Cauchy-Schwarz, arithmetic mean-geometric mean (AM-GM) and polarization inequalities.

\subsection{Two Strategies for Higher-Order Time Derivatives}
\label{subsec:two_strategies}
We explain our strategies for obtaining a higher-order time discretization by using the following model problem. Let $u: \Omega \times [0, T] \longrightarrow \R$ belonging to an appropriate function space (for example, $u \in \mathring{H}_{\varepsilon}(\curl; \Omega; L^2(0, T)) \cap C^\infty(\Omega, (0,T))$) be the solution of the following partial differential equation (PDE):
\begin{equation}
\dfrac{\partial u}{\partial t} = \nabla \times u \quad \text{on } \Omega \times [0,T].
\label{eqn:modelPDE}
\end{equation}
We assume a uniform discretization of $[0, T]$ as specified earlier too and let $t^n \coloneq n \Delta t$, $n = 0, 1, \dots, N$ where $N \Delta t \approx T$ for some suitable choice of a fixed time step parameter $\Delta t > 0$. We now Taylor expand the solution around $t^n$ to obtain:
\[
\dfrac{u^{n+1} - u^{n-1}}{2 \Delta t} = \dfrac{\partial u}{\partial t}(t^n) + \dfrac{\Delta t^2}{6} \dfrac{\partial^3 u}{\partial t^3}(t^n) + \mathcal{O}(\Delta t^4),
\]
where $u^n \coloneq u(x, t^n)$ for $x \in \Omega$ and $\mathcal{O}$ is the standard Bachmann–Landau notation. Then, our two strategies can be explained as follows.

\begin{itemize}
  \item[] \textbf{Spatial Strategy:} In the above Taylor series, we transform the higher-order time derivatives into higher-order spatial derivatives using the PDE itself. Thus, in our example, $\dfrac{\partial^3 u}{\partial t^3} = \nabla \times \nabla \times \nabla \times u$, and so we now set up our spatially derived scheme to be a discretization of the following resultant PDE:
  \[
 \dfrac{u^{n+1} - u^{n-1}}{2 \Delta t} = \nabla \times u^n  + \dfrac{\Delta t^2}{6} \nabla \times \nabla \times \nabla \times u^n,
  \]
  and we use carefully use FEEC principles for discretizing the spatial solution. \\
  
  \item[] \textbf{Temporal Strategy:} In the temporal strategy, we first replace a single time derivative in all of the higher-order partial time derivatives using the PDE, and then discretize all remaining time derivatives using suitably higher-order finite difference schemes. The finite difference schemes for each of the time derivatives are so chosen as to not worsen the fourth-order accuracy in our time discretization. Thus, in our example, we replace one of the third-order time partial derivatives using Equation~\labelcref{eqn:modelPDE}, and the resulting second-order time partial derivative is discretized by a second-order accurate finite difference approximation. In general, we will have more time derivatives and each of them will potentially require a differen finite difference scheme. Thus, we obtain $\dfrac{\partial^3 u}{\partial t^3} = \nabla \times \dfrac{\partial^2 u}{\partial t^2} \approx \nabla \times \left(\dfrac{u^{n+1} - 2 u^n + u^{n-1}}{\Delta t^2} \right)$, and set up our discrete scheme to be the following:
\[
 \dfrac{u^{n+1} - u^{n-1}}{2 \Delta t} = \nabla \times u^n  + \dfrac{\Delta t^2}{6} \nabla \times \left(\dfrac{u^{n+1} - 2 u^n + u^{n-1}}{\Delta t^2} \right).
\]
\end{itemize}
In both our above strategies, the spatial component of the solution is discretized using a compatible and a suitably higher-order discretization using the de Rham finite element spaces as in FEEC. Applying these principles to our Maxwell's system in Equation~\labelcref{eqn:maxwells_eqns} leads us, respectively, to Equations~\labelcref{eqn:maxwell_p_lf4,eqn:maxwell_E_lf4,eqn:maxwell_H_lf4} and Equations~\labelcref{eqn:maxwell_p_ts4,eqn:maxwell_E_ts4,eqn:maxwell_H_ts4} as our proposed fourth-order time discretization schemes.

\section{Characterization of LF$_4$ Scheme} \label{sec:implicit_lf4}

\subsection{Time Discretization Stability}

For our implicit LF$_4$ time discretization scheme as in Equations~\labelcref{eqn:maxwell_p_lf4,eqn:maxwell_E_lf4,eqn:maxwell_H_lf4} with bootstrapping using Equations~\labelcref{eqn:maxwell_p0_lf4,eqn:maxwell_E0_lf4,eqn:maxwell_H0_lf4}, we now show that the semidiscretization in time yields a stable method. We also provide the error analysis to show its fourth-order convergence in time in terms of a sufficiently small time step $\Delta t > 0$. For the LF$_4$ scheme, we define the discrete energy at time $t^n$ to be the following:
\begin{equation}
\mathcal{E}^n \coloneq \norm{p^{N - \frac{1}{2}}}_{\varepsilon^{-1}}^2 + \norm{E^{N - \frac{1}{2}}}_{\varepsilon}^2 + \norm{H^N}_{\mu}^2.
\end{equation}

\begin{theorem}[Discrete Energy Conservation] \label{thm:dscrt_enrgy_estmt_lf4}
  For the semidiscretization using the LF$_4$ scheme as given in Equations~\labelcref{eqn:maxwell_p_lf4,eqn:maxwell_E_lf4,eqn:maxwell_H_lf4,eqn:maxwell_p0_lf4,eqn:maxwell_E0_lf4,eqn:maxwell_H0_lf4}, for any fixed time step $\Delta t > 0$ sufficiently small, we have that:
\[
  \mathcal{E}^n = \norm{p^{N - \frac{1}{2}}}^2_{\varepsilon^{-1}} + \norm{E^{N - \frac{1}{2}}}^2_{\varepsilon} + \norm{H^N}^2_{\mu} = \norm{p_0}^2_{\varepsilon^{-1}} + \norm{E_0}^2_{\varepsilon} + \norm{H_0}^2_{\mu} \eqcolon \mathcal{E}^0.
\]
\end{theorem}

\begin{proof}
  Since Equations~\labelcref{eqn:maxwell_p_lf4,eqn:maxwell_E_lf4,eqn:maxwell_H_lf4} are true for all $\widetilde{p}\in \mathring{H}^1_{\varepsilon^{-1}}(\Omega)$, $\widetilde{E} \in \mathring{H}_{\varepsilon}(\curl; \Omega)$, $\widetilde{H} \in \mathring{H}_{\mu}(\divgn; \Omega)$, using $\widetilde{p} = 2 \Delta t \varepsilon^{-1} \left( p^{n + \frac{1}{2}} + p^{n - \frac{1}{2}} \right)$, $\widetilde{E} = 2 \Delta t \left(E^{n + \frac{1}{2}} + E^{n - \frac{1}{2}} \right)$ and $\widetilde{H} = 2 \Delta t \left(H^{n + 1} + H^n \right)$ in them, we obtain the following:
  \begin{multline*}
    2 \aInnerproduct{p^{n+\frac{1}{2}} - p^{n - \frac{1}{2}}}{\varepsilon^{-1} \left( p^{n + \frac{1}{2}} + p^{n-\frac{1}{2}} \right)} - \Delta t \aInnerproduct{E^{n + \frac{1}{2}} + E^{n - \frac{1}{2}}}{\nabla \left( p^{n + \frac{1}{2}} + p^{n - \frac{1}{2}} \right)} \\ + \dfrac{\Delta t^3}{12} \aInnerproduct{\nabla \nabla \cdot \left(E^{n + \frac{1}{2}} + E^{n - \frac{1}{2}}\right)}{\nabla \left( p^{n + \frac{1}{2}} + p^{n - \frac{1}{2}} \right)} =0,
\end{multline*}
\vspace{-1em} \begin{multline*}
  \Delta t \aInnerproduct{\nabla \left( p^{n + \frac{1}{2}} + p^{n - \frac{1}{2}} \right)}{E^{n + \frac{1}{2}} + E^{n - \frac{1}{2}}} - \dfrac{\Delta t^3}{12} \aInnerproduct{\nabla \left( p^{n + \frac{1}{2}} + p^{n - \frac{1}{2}} \right)}{\nabla \nabla \cdot \left(E^{n + \frac{1}{2}} + E^{n - \frac{1}{2}}\right)} \\ + 2 \aInnerproduct{\varepsilon \left(E^{n + \frac{1}{2}} - E^{n - \frac{1}{2}} \right)}{E^{n + \frac{1}{2}} + E^{n - \frac{1}{2}}} \, -
  \Delta t \aInnerproduct{\left( H^{n + 1} + H^n \right)}{\nabla \times \left( E^{n + \frac{1}{2}} + E^{n - \frac{1}{2}} \right)} \\ -  \dfrac{\Delta t^3}{12} \aInnerproduct{\mu^{-1}\varepsilon^{-1} \nabla \times \nabla \times \left( H^{n + 1} + H^n \right)}{\nabla \times \left( E^{n + \frac{1}{2}} + E^{n - \frac{1}{2}} \right)} = 0,
\end{multline*}
\vspace{-1em} \begin{multline*}
  2 \aInnerproduct{\mu \left(H^{n + 1} - H^{n} \right)}{H^{n + 1} + H^n} + \Delta t \aInnerproduct{\nabla \times \left( E^{n + \frac{1}{2}} + E^{n - \frac{1}{2}} \right)}{H^{n + 1} + H^n} \\ + \dfrac{\Delta t^3}{12}  \aInnerproduct{\varepsilon^{-1}\mu^{-1} \nabla \times \left( E^{n + \frac{1}{2}} + E^{n - \frac{1}{2}} \right)}{\nabla \times \nabla \times \left(H^{n + 1} + H^n \right)} = 0.
\end{multline*}
Adding these equations together and using properties of the inner product, we get that:
\begin{multline*}
  2 \aInnerproduct{\varepsilon^{-1} \left( p^{n + \frac{1}{2}} - p^{n - \frac{1}{2}} \right)}{p^{n + \frac{1}{2}} + p^{n - \frac{1}{2}}} + 2 \aInnerproduct{\varepsilon \left(E^{n + \frac{1}{2}} - E^{n - \frac{1}{2}} \right)}{E^{n + \frac{1}{2}} + E^{n - \frac{1}{2}}} \\ + 2 \aInnerproduct{\mu \left( H^{n + 1} - H^{n} \right)}{H^{n + 1} + H^n} = 0.
\end{multline*}
Summing over $n = 1$ to $N - 1$ leads us to:
\[
  \norm{p^{N - \frac{1}{2}}}^2_{\varepsilon^{-1}} - \norm{p^\frac{1}{2}}^2_{\varepsilon^{-1}} + \norm{E^{N - \frac{1}{2}}}^2_{\varepsilon} - \norm{E^\frac{1}{2}}^2_{\varepsilon} + \norm{H^N}^2_{\mu} - \norm{H^1}^2_{\mu} = 0.
\]
Now, $p^{\frac{1}{2}}$, $E^{\frac{1}{2}}$ and $H^1$ satisfy Equations~\labelcref{eqn:maxwell_p0_lf4,eqn:maxwell_E0_lf4,eqn:maxwell_H0_lf4}, and so using $\widetilde{p} = 2 \Delta t \varepsilon^{-1} \left( p^{\frac{1}{2}} + p^0 \right)$, $\widetilde{E} = 2 \Delta t \left(E^{\frac{1}{2}} + E^0 \right)$ and $\widetilde{H} = 4 \Delta t \left(H^{1} + H^0 \right)$ and repeating the previous arguments leads us to the following estimate:
\[
  \norm{p^{\frac{1}{2}}}^2_{\varepsilon^{-1}}  - \norm{p^0}^2_{\varepsilon^{-1}} + \norm{E^{\frac{1}{2}}}^2_{\varepsilon} - \norm{E^0}^2_{\varepsilon} + \norm{H^{1}}^2_{\mu} - \norm{H^0}^2_{\mu} = 0.
\]
Adding these previous two equations, we get that:
\[
  \norm{p^{N - \frac{1}{2}}}^2_{\varepsilon^{-1}} + \norm{E^{N - \frac{1}{2}}}^2_{\varepsilon} + \norm{H^N}^2_{\mu} = \norm{p^0}^2_{\varepsilon^{-1}} + \norm{E^0}^2_{\varepsilon} + \norm{H^0}^2_{\mu}. \qedhere
\]
\end{proof}

\begin{theorem}[Discrete Error Estimate]\label{thm:dscrt_error_estmt_lf4}
For the semidiscretization using the LF$_4$ scheme as in Equations~\labelcref{eqn:maxwell_p_lf4,eqn:maxwell_E_lf4,eqn:maxwell_H_lf4}, and ~\labelcref{eqn:maxwell_p0_lf4,eqn:maxwell_E0_lf4,eqn:maxwell_H0_lf4}, for the solution $(p, E, H)$ of the variationally posed Maxwell's system as in Equations~\labelcref{eqn:maxwell_p_wf,eqn:maxwell_E_wf,eqn:maxwell_H_wf} with initial conditions as in Equation~\eqref{eqn:ICs}, assuming sufficient regularity with $p \in C^5(0, T; \mathring{H}^1_{\varepsilon^{-1}}(\Omega))$, $E \in C^5(0, T; \mathring{H}_{\varepsilon}(\curl; \Omega))$, and $H \in C^5(0, T; \mathring{H}_{\mu}(\divgn; \Omega))$, and for a time step $\Delta t > 0$ sufficiently small, there exists a positive bounded constant $C$ independent of $\Delta t$ such that:
\[
  \norm{e_p^{N - \frac{1}{2}}}_{\varepsilon^{-1}} + \norm{e_E^{N - \frac{1}{2}}}_{\varepsilon} + \norm{e_H^N}_{\mu} \le C \left[ \left(\Delta t\right)^4 + \norm{e_p^0}_{\varepsilon^{-1}} + \norm{e_E^0}_{\varepsilon} + \norm{e_H^0}_{\mu} \right],
\]
where $e_p^{n + \frac{1}{2}} \coloneq p(t^{n + \frac{1}{2}}) - p^{n+\frac{1}{2}}$, $e_E^{n + \frac{1}{2}} \coloneq E(t^{n + \frac{1}{2}}) - E^{n + \frac{1}{2}}$ and $e_H^n \coloneq H(t^n) - H^n$ are the errors in the time semidiscretization of $p$, $E$ and $H$, respectively.
\end{theorem}

\begin{proof}
Using the Taylor remainder theorem, and expressing $p(t)$ about $t = t^n$, we have that:
\[
  p(t) = p(t^n) + \dfrac{\partial p}{\partial t}(t^n)(t - t^n) + \dfrac{\partial^2 p}{\partial t^2}(t^n) \dfrac{(t - t^n)^2}{2!} + \dfrac{\partial^3 p}{\partial t^3}(t^n) \dfrac{(t - t^n)^3}{3!} + \dfrac{\partial^4 p}{\partial t^4}(t^n) \dfrac{(t - t^n)^4}{4!} + \int\limits_{t^n}^{t} \dfrac{(t - s)^4}{4!} \dfrac{\partial^5 p}{\partial t^5}(s) ds,
\]
which when evaluated at $t = t^{n + \frac{1}{2}}$ and $t = t^{n - \frac{1}{2}}$ yields:
\begin{multline*}
  p(t^{n + \frac{1}{2}}) = p(t^n) + \dfrac{\partial p}{\partial t}(t^n) (t^{n + \frac{1}{2}} - t^n) + \dfrac{\partial^2 p}{\partial t^2}(t^n) \dfrac{(t^{n + \frac{1}{2}} - t^n)^2}{2} + \dfrac{\partial^3 p}{\partial t^3}(t^n) \dfrac{(t^{n + \frac{1}{2}} - t^n)^3}{6} \\ + \dfrac{\partial^4 p}{\partial t^4}(t^n) \dfrac{(t^{n + \frac{1}{2}} - t^n)^4}{24} + \int\limits_{t^n}^{\mathclap{t^{n + \frac{1}{2}}}} \dfrac{(t^{n + \frac{1}{2}} - s)^4}{24} \dfrac{\partial^5 p}{\partial t^5}(s) ds, 
 \end{multline*}
\begin{multline*}
p(t^{n - \frac{1}{2}}) = p(t^n) + \dfrac{\partial p}{\partial t}(t^n)(t^{n - \frac{1}{2}} - t^n) + \dfrac{\partial^2 p}{\partial t^2}(t^n) \dfrac{(t^{n - \frac{1}{2}} - t^n)^2}{2} + \dfrac{\partial^3 p}{\partial t^3}(t^n) \dfrac{(t^{n - \frac{1}{2}} - t^n)^3}{6} \\ + \dfrac{\partial^4 p}{\partial t^4}(t^n) \dfrac{(t^{n - \frac{1}{2}} - t^n)^4}{24} + \int\limits_{t^n}^{\mathclap{t^{n - \frac{1}{2}}}} \dfrac{(t^{n - \frac{1}{2}} - s)^4}{24} \dfrac{\partial^5 p}{\partial t^5}(s) ds.
\end{multline*}
Subtracting these two equations, and using the result in the inner product term from the semidiscretization of the variational formulation leads to:
\[
  \ainnerproduct{\dfrac{p(t^{n + \frac{1}{2}}) - p(t^{n - \frac{1}{2}})}{\Delta t}}{\widetilde{p}} = \ainnerproduct{\dfrac{\partial p}{\partial t}(t^n)}{\widetilde{p}} + \dfrac{\Delta t^2}{24} \ainnerproduct{\dfrac{\partial^3 p}{\partial t^3}(t^n)}{\widetilde{p}} + \ainnerproduct{R^n_p}{\widetilde{p}},
\]
in which we have defined that:
\[
  R^n_p \coloneq \dfrac{1}{\Delta t} \left[\int\limits_{t^{n - \frac{1}{2}}}^{t^n} \dfrac{(t^{n - \frac{1}{2}} - s)^4}{24} \dfrac{\partial^5 p}{\partial t^5}(s) ds + \int\limits_{t^n}^{t^{n + \frac{1}{2}}} \dfrac{(t^{n + \frac{1}{2}} - s)^4}{24} \dfrac{\partial^5 p}{\partial t^5}(s) ds \right].
\]
Similarly, for $E$ and $H$, we have the following:
\begin{align*}
  \ainnerproduct{\varepsilon \dfrac{E(t^{n + \frac{1}{2}}) - E(t^{n - \frac{1}{2}})}{\Delta t}}{\widetilde{E}} &= \ainnerproduct{\varepsilon \dfrac{\partial E}{\partial t}(t^n)}{\widetilde{E}} + \dfrac{\Delta t^2}{24} \ainnerproduct{\varepsilon \dfrac{\partial^3 E}{\partial t^3}(t^n)}{\widetilde{E}} + \ainnerproduct{\varepsilon R^n_E}{\widetilde{E}}, \\
  \ainnerproduct{\mu \dfrac{H(t^{n + 1}) - H(t^n)}{\Delta t}}{\widetilde{H}} &= \ainnerproduct{\mu \dfrac{\partial H}{\partial t}(t^{n + \frac{1}{2}})}{\widetilde{H}} + \dfrac{\Delta t^2}{24} \ainnerproduct{\mu \dfrac{\partial^3 H}{\partial t^3}(t^{n + \frac{1}{2}})}{\widetilde{H}} + \ainnerproduct{\mu R^{n + \frac{1}{2}}_H}{\widetilde{H}},
\end{align*}
and in each of which we have defined that:
\begin{align*}
  R^n_E &\coloneq \dfrac{1}{\Delta t} \left[ \int\limits_{t^{n - \frac{1}{2}}}^{t^n} \dfrac{(t^{n - \frac{1}{2}} - s)^4}{24} \dfrac{\partial^5 E}{\partial t^5}(s) ds + \int\limits_{t^n}^{t^{n + \frac{1}{2}}} \dfrac{(t^{n + \frac{1}{2}} - s)^4}{24} \dfrac{\partial^5 E}{\partial t^5}(s) ds \right], \\
R^{n + \frac{1}{2}}_H &\coloneq \dfrac{1}{\Delta t} \left[ \int\limits_{t^n}^{t^{n + \frac{1}{2}}} \dfrac{(t^n - s)^4}{24} \dfrac{\partial^5 H}{\partial t^5}(s) ds + \int\limits_{t^{n + \frac{1}{2}}}^{t^{n + 1}} \dfrac{(t^{n + 1} - s)^4}{24} \dfrac{\partial^5 H}{\partial t^5}(s) ds\right].
\end{align*}
Using these terms in the weak formulation as in Equations~\labelcref{eqn:maxwell_p_wf,eqn:maxwell_E_wf,eqn:maxwell_H_wf}  at time $t = t^n$ for $p$ and $E$ terms, and at time $t = t^{n + \frac{1}{2}}$ for $H$ and using the Taylor remainder theorem again with:
\[
  u(t^n) \coloneq \dfrac{u(t^{n + \frac{1}{2}}) + u(t^{n - \frac{1}{2}})}{2} - r_{1,u}^n, \quad H(t^{n + \frac{1}{2}}) \coloneq \dfrac{H(t^{n+1}) + H(t^n)}{2} - r_{1,H}^{n + \frac{1}{2}},
\] 
\[
  u(t^n) \coloneq \dfrac{u(t^{n + \frac{1}{2}}) + u(t^{n - \frac{1}{2}})}{2} - \dfrac{\Delta t^2}{8}  \dfrac{\partial^2 u}{\partial t^2}(t^n) - r_{2,u}^n, \quad H(t^{n + \frac{1}{2}}) \coloneq \dfrac{H(t^{n+1}) + H(t^n)}{2} - \dfrac{\Delta t^2}{8}  \dfrac{\partial^2 H}{\partial t^2}(t^{n + \frac{1}{2}}) - r_{2,H}^{n + \frac{1}{2}},
\] 
where $u$ is either of $p$ or $E$, and in which we have defined that:
\[
r_{1,u}^n \coloneqq \dfrac{1}{2} \left[\int\limits_{\mathclap{t^n}}^{t^{n + \frac{1}{2}}} (t^{n + \frac{1}{2}} - s) \dfrac{\partial^2 u}{\partial t^2}(s) ds - \int\limits_{t^{n - \frac{1}{2}}}^{t^n} (t^{n - \frac{1}{2}} - s) \dfrac{\partial^2 u}{\partial t^2}(s) ds\right],
\]
\[
r_{1,H}^{n + \frac{1}{2}} \coloneqq \dfrac{1}{2} \left[\,\, \int\limits_{\mathclap{t^{n + \frac{1}{2}}}}^{t^{n+1}} (t^{n+1} - s) \dfrac{\partial^2 H}{\partial t^2}(s) ds - \int\limits_{t^n}^{t^{n + \frac{1}{2}}} (t^n - s) \dfrac{\partial^2 H}{\partial t^2}(s) ds\right],
\]
\[
r_{2,u}^n \coloneqq \dfrac{1}{2} \left[\int\limits_{\mathclap{t^n}}^{t^{n + \frac{1}{2}}} \dfrac{(t^{n + \frac{1}{2}} - s)^3}{6} \dfrac{\partial^4 u}{\partial t^4}(s) ds - \int\limits_{t^{n - \frac{1}{2}}}^{t^n} \dfrac{(t^{n - \frac{1}{2}} - s)^3}{6} \dfrac{\partial^4 u}{\partial t^4}(s) ds\right],
\]
\[
r_{2,H}^{n + \frac{1}{2}} \coloneqq \dfrac{1}{2} \left[\,\, \int\limits_{\mathclap{t^{n + \frac{1}{2}}}}^{t^{n+1}} \dfrac{(t^{n+1} - s)^3}{6} \dfrac{\partial^4 H}{\partial t^4}(s) ds - \int\limits_{t^n}^{t^{n + \frac{1}{2}}} \dfrac{(t^n - s)^3}{6} \dfrac{\partial^4 H}{\partial t^4}(s) ds\right].
\]
Using these, we obtain the following:
\begin{subequations}
 \begin{multline}
  \aInnerproduct{\dfrac{p(t^{n + \frac{1}{2}}) - p(t^{n - \frac{1}{2}})}{\Delta t}}{\widetilde{p}} - \aInnerproduct{\dfrac{\varepsilon}{2} \left( E(t^{n + \frac{1}{2}}) + E(t^{n - \frac{1}{2}}) \right)}{\nabla \widetilde{p}} + \dfrac{\Delta t^2}{12} \aInnerproduct{\dfrac{\varepsilon}{2} \nabla \nabla \cdot \left( E(t^{n + \frac{1}{2}}) + E(t^{n - \frac{1}{2}}) \right)}{\nabla \widetilde{p}} \\ =\aInnerproduct{R_p^{n} + \varepsilon \nabla \cdot r_{2,E}^n - \dfrac{\Delta t^2}{12} \varepsilon \nabla \cdot \nabla \nabla \cdot r_{1,E}^n }{\widetilde{p}}, \label{eqn:remainder_p_lf4}
  \end{multline} \\
  \begin{multline}
   \aInnerproduct{\dfrac{1}{2} \nabla \left(p(t^{n + \frac{1}{2}}) + p(t^{n - \frac{1}{2}}) \right)}{\widetilde{E}} - \dfrac{\Delta t^2}{12} \aInnerproduct{\dfrac{1}{2} \nabla \left(p(t^{n + \frac{1}{2}}) + p(t^{n - \frac{1}{2}}) \right)}{\nabla \nabla \cdot \widetilde{E}} + \aInnerproduct{\varepsilon \dfrac{E(t^{n + \frac{1}{2}}) - E(t^{n - \frac{1}{2}})}{\Delta t}}{\widetilde{E}} \\ - \aInnerproduct{\dfrac{1}{2} \left( H(t^{n + 1}) + H(t^n) \right)}{\nabla \times \widetilde{E}}  - \dfrac{\Delta t^2}{12} \aInnerproduct{\dfrac{1}{2} \mu^{-1}\varepsilon^{-1} \nabla \times \nabla \times \left( H(t^{n + 1}) + H(t^n) \right)}{\nabla \times \widetilde{E}} \\ =  \aInnerproduct{\varepsilon R_E^n + \nabla r_{2,p}^n - \nabla \times r_{2,H}^{n+\frac{1}{2}} - \dfrac{\Delta t^2}{12} \nabla \nabla \cdot \nabla r_{1,p}^n - \dfrac{\Delta t^2}{12}  \mu^{-1}\varepsilon^{-1} \nabla \times \nabla \times \nabla \times r_{1,H}^{n+\frac{1}{2}}}{\widetilde{E}}, \label{eqn:remainder_E_lf4} 
     \end{multline} \\
      \begin{multline}
  \aInnerproduct{\mu \dfrac{H(t^{n + 1}) - H(t^n)}{\Delta t}}{\widetilde{H}} +  \aInnerproduct{\dfrac{1}{2} \nabla \times \left( E(t^{n + \frac{1}{2}}) + E(t^{n - \frac{1}{2}}) \right)}{\widetilde{H}} \\ +  \dfrac{\Delta t^2}{12} \aInnerproduct{\dfrac{1}{2} \varepsilon^{-1}\mu^{-1} \nabla \times \left( E(t^{n + \frac{1}{2}}) + E(t^{n - \frac{1}{2}}) \right)}{\nabla \times  \nabla \times \widetilde{H}} \\ =  \aInnerproduct{\mu R_H^{n + \frac{1}{2}} + \nabla \times r_{2,E}^n + \dfrac{\Delta t^2}{12} \varepsilon^{-1}\mu^{-1} \nabla \times \nabla \times \nabla \times r_{1,E}^n}{\widetilde{H}}. \label{eqn:remainder_H_lf4}
  \end{multline}
\end{subequations}
Then, subtracting these from each of their respective LF$_4$ scheme equations as in Equations~\labelcref{eqn:maxwell_p_lf4,eqn:maxwell_E_lf4,eqn:maxwell_H_lf4} leads us to the following set of equations:
\begin{multline*}
\aInnerproduct{\dfrac{e_p^{n + \frac{1}{2}} - e_p^{n - \frac{1}{2}}}{\Delta t}}{\widetilde{p}} - \dfrac{1}{2} \aInnerproduct{ \varepsilon \left(e_E^{n + \frac{1}{2}} + e_E^{n - \frac{1}{2}} \right)}{\nabla \widetilde{p}} + \dfrac{\Delta t^2}{24} \aInnerproduct{\varepsilon \nabla \nabla \cdot \left( e_E^{n + \frac{1}{2}} + e_E^{n - \frac{1}{2}} \right)}{\nabla \widetilde{p}} \\ = \aInnerproduct{R_p^{n} + \varepsilon \nabla \cdot r_{2,E}^n - \dfrac{\Delta t^2}{12} \varepsilon \nabla \cdot \nabla \nabla \cdot r_{1,E}^n }{\widetilde{p}}, 
 \end{multline*} 
\begin{multline*}
 \dfrac{1}{2} \aInnerproduct{\nabla \left(e_p^{n + \frac{1}{2}} + e_p^{n - \frac{1}{2}}\right)}{\widetilde{E}}  - \dfrac{\Delta t^2}{24} \aInnerproduct{\nabla \left(e_p^{n + \frac{1}{2}} + e_p^{n - \frac{1}{2}} \right)}{\nabla \nabla \cdot \widetilde{E}} + \aInnerproduct{\varepsilon \dfrac{e_E^{n + \frac{1}{2}} - e_E^{n - \frac{1}{2}}}{\Delta t}}{\widetilde{E}} \\ - \dfrac{1}{2} \aInnerproduct{e_H^{n + 1} + e_H^n}{\nabla \times \widetilde{E}}  - \dfrac{\Delta t^2}{24} \aInnerproduct{ \mu^{-1}\varepsilon^{-1} \nabla \times \nabla \times \left( e_H^{n + 1} + e_H^n \right)}{\nabla \times \widetilde{E}} \\ = \aInnerproduct{\varepsilon R_E^n + \nabla r_{2,p}^n - \nabla \times r_{2,H}^{n+\frac{1}{2}} - \dfrac{\Delta t^2}{12} \nabla \nabla \cdot \nabla r_{1,p}^n - \dfrac{\Delta t^2}{12}  \mu^{-1}\varepsilon^{-1} \nabla \times \nabla \times \nabla \times r_{1,H}^{n+\frac{1}{2}}}{\widetilde{E}},
  \end{multline*} 
 \begin{multline*}
\aInnerproduct{\mu \dfrac{e_H^{n + 1} - e_H^{n}}{\Delta t}}{\widetilde{H}} +  \dfrac{1}{2}\aInnerproduct{\nabla \times\left(e_E^{n+\frac{1}{2}} + e_E^{n - \frac{1}{2}} \right)}{\widetilde{H}} +  \dfrac{\Delta t^2}{24} \aInnerproduct{\varepsilon^{-1}\mu^{-1} \nabla \times \left(e_E^{n + \frac{1}{2}} + e_E^{n - \frac{1}{2}} \right)}{\nabla \times  \nabla \times \widetilde{H}} \\ = \aInnerproduct{\mu R_H^{n + \frac{1}{2}} + \nabla \times r_{2,E}^n + \dfrac{\Delta t^2}{12} \varepsilon^{-1}\mu^{-1} \nabla \times \nabla \times \nabla \times r_{1,E}^n}{\widetilde{H}}.
 \end{multline*}
Likewise, for the semidiscrete approximation of the initial system as in Equations~\labelcref{eqn:maxwell_p0_lf4,eqn:maxwell_E0_lf4,eqn:maxwell_H0_lf4}, we obtain for their errors the following system of equations:
 \begin{multline*}
  \aInnerproduct{\dfrac{e_p^{\frac{1}{2}} - e_p^0}{\Delta t/2}}{\widetilde{p}} - \dfrac{1}{4} \aInnerproduct{ \varepsilon \left(e_E^{\frac{1}{2}} + e_E^0 \right)}{\nabla \widetilde{p}} + \dfrac{\Delta t^2}{192} \aInnerproduct{\varepsilon \nabla \nabla \cdot \left( e_E^{\frac{1}{2}} + e_E^0 \right)}{\nabla \widetilde{p}} \\ = \aInnerproduct{\dfrac{1}{2} R_p^0 -   \dfrac{\varepsilon}{2} \nabla \cdot r_{2,E}^0 - \dfrac{\Delta t^2}{96} \varepsilon \nabla \cdot \nabla \nabla \cdot r_{1,E}^0}{\widetilde{p}}, 
  \end{multline*} 
 \begin{multline*}
  \dfrac{1}{4} \aInnerproduct{\nabla \left(e_p^{\frac{1}{2}} + e_p^0 \right)}{\widetilde{E}} - \dfrac{\Delta t^2}{192} \aInnerproduct{ \nabla \left( e_p^{\frac{1}{2}} +  e_p^0 \right)}{\nabla \nabla \cdot \widetilde{E}} + \aInnerproduct{\varepsilon \dfrac{e_E^{\frac{1}{2}} - e_E^0}{\Delta t/2}}{\widetilde{E}}  - \dfrac{1}{2} \aInnerproduct{e_H^1 + e_H^0}{\nabla \times \widetilde{E}} \\ - \dfrac{\Delta t^2}{96}  \aInnerproduct{\mu^{-1}\varepsilon^{-1}  \nabla \times \nabla \times \left( e_H^1 + e_H^0 \right)}{\widetilde{E}} \\ = \aInnerproduct{\dfrac{\varepsilon}{2} R_E^0 + \dfrac{1}{2} \nabla r_{2,p}^0 - \nabla \times r_{2,H}^{\frac{1}{2}} - \dfrac{\Delta t^2}{96} \nabla \nabla \cdot \nabla r_{1,p}^0 - \dfrac{\Delta t^2}{48}  \mu^{-1}\varepsilon^{-1} \nabla \times \nabla \times \nabla \times r_{1,H}^{\frac{1}{2}}}{\widetilde{E}}, 
\end{multline*} 
 \begin{multline*}
 \aInnerproduct{\mu \dfrac{e_H^1 - e_H^0}{\Delta t}}{\widetilde{H}} + \dfrac{1}{4} \aInnerproduct{\nabla \times \left( e_E^{\frac{1}{2}} + e_E^0 \right)}{\widetilde{H}} + \dfrac{\Delta t^2}{192} \aInnerproduct{ \varepsilon^{-1}\mu^{-1} \nabla \times \left(e_E^{\frac{1}{2}} + e_E^0 \right)}{ \nabla \times \nabla \times  \widetilde{H}}\\  = \aInnerproduct{\mu R_H^{\frac{1}{2}}+ \dfrac{1}{2} \nabla \times r_{2,E}^0 + \dfrac{\Delta t^2}{96} \varepsilon^{-1}\mu^{-1} \nabla \times \nabla \times \nabla \times r_{1,E}^0}{\widetilde{H}},
\end{multline*}
and here we define the initial remainder terms to be as follows:
\begin{align*}
R^0_u & \coloneq \dfrac{1}{\Delta t} \left[ \int\limits_{t^0}^{t^\frac{1}{4}} \dfrac{(t^0 - s)^4}{24} \dfrac{\partial^5 u}{\partial t^5}(s) ds + \int\limits_{t^\frac{1}{4}}^{t^{\frac{1}{2}}} \dfrac{(t^{\frac{1}{2}} - s)^4}{24} \dfrac{\partial^5 u}{\partial t^5}(s) ds \right], \\
R^{ \frac{1}{2}}_H & \coloneq \dfrac{1}{\Delta t} \left[ \int\limits_{t^0}^{t^{ \frac{1}{2}}} \dfrac{(t^0 - s)^4}{24} \dfrac{\partial^5 H}{\partial t^5}(s) ds + \int\limits_{t^{\frac{1}{2}}}^{t^{1}} \dfrac{(t^{1} - s)^4}{24} \dfrac{\partial^5 H}{\partial t^5}(s) ds\right], \\
r_{1,u}^0 & \coloneqq \dfrac{1}{2} \left[\int\limits_{\mathclap{t^\frac{1}{4}}}^{t^{ \frac{1}{2}}} (t^{ \frac{1}{2}} - s) \dfrac{\partial^2 u}{\partial t^2}(s) ds - \int\limits_{t^0}^{t^\frac{1}{4}} (t^0 - s) \dfrac{\partial^2 u}{\partial t^2}(s) ds\right], \\
r_{1,H}^{ \frac{1}{2}} & \coloneqq \dfrac{1}{2} \left[\,\, \int\limits_{\mathclap{t^{\frac{1}{2}}}}^{t^{1}} (t^{1} - s) \dfrac{\partial^2 H}{\partial t^2}(s) ds - \int\limits_{t^0}^{t^{\frac{1}{2}}} (t^0 - s) \dfrac{\partial^2 H}{\partial t^2}(s) ds\right], \\
r_{2,u}^0 & \coloneqq \dfrac{1}{2} \left[\int\limits_{\mathclap{t^\frac{1}{4}}}^{t^{ \frac{1}{2}}} \dfrac{(t^{ \frac{1}{2}} - s)^3}{6} \dfrac{\partial^4 u}{\partial t^4}(s) ds - \int\limits_{t^0}^{t^\frac{1}{4}} \dfrac{(t^0 - s)^3}{6} \dfrac{\partial^4 u}{\partial t^4}(s) ds\right], \\
r_{2,H}^{ \frac{1}{2}} & \coloneqq \dfrac{1}{2} \left[\,\, \int\limits_{\mathclap{t^{\frac{1}{2}}}}^{t^{1}} \dfrac{(t^{1} - s)^3}{6} \dfrac{\partial^4 H}{\partial t^4}(s) ds - \int\limits_{t^0}^{t^{\frac{1}{2}}} \dfrac{(t^0 - s)^3}{6} \dfrac{\partial^4 H}{\partial t^4}(s) ds\right].
\end{align*}
Now, in this set of weak formulation equations for the errors, we choose the test functions to be $\widetilde{p} = 2 \Delta t \varepsilon^{-1} \left( e_p^{n + \frac{1}{2}} + e_p^{n - \frac{1}{2}} \right)$, $\widetilde{E} = 2 \Delta t \left( e_E^{n + \frac{1}{2}} + e_E^{n - \frac{1}{2}} \right)$ and $\widetilde{H} = 2 \Delta t \left( e_H^n + e_H^{n - 1} \right)$. Next, by following essentially the same sequence of steps as in Theorem~\ref{thm:dscrt_enrgy_estmt_lf4}, we obtain the estimate for these error terms to be:
\begin{multline*}
  \norm{e_p^{n + \frac{1}{2}}}^2_{\varepsilon^{-1}} - \norm{e_p^{n - \frac{1}{2}}}^2_{\varepsilon^{-1}} + \norm{e_E^{n + \frac{1}{2}}}^2_{\varepsilon} - \norm{e_E^{n - \frac{1}{2}}}^2_{\varepsilon} + \norm{e_H^{n+1}}^2_{\mu} -  \norm{e_H^n}^2_{\mu} \\ \le \Delta t \left[ \norm{e_p^{n + \frac{1}{2}}}^2_{\varepsilon^{-1}} + \norm{e_p^{n - \frac{1}{2}}}^2_{\varepsilon^{-1}} + \norm{e_E^{n + \frac{1}{2}}}^2_{\varepsilon} + \norm{e_E^{n - \frac{1}{2}}}^2_{\varepsilon} + \norm{e_H^{n+1}}^2_{\mu} +  \norm{e_H^n}^2_{\mu}\right] \\ +
 \Delta t \left[ \norm{R_p^n}^2_{\varepsilon^{-1}} + \norm{R_E^n}^2_{\varepsilon} + \norm{R_H^{n + \frac{1}{2}}}^2_{\mu} + \norm{\nabla r_{2,p}^n} ^2_{\varepsilon^{-1}} + \norm{\nabla \cdot r_{2,E}^n}^2_{\varepsilon} + \varepsilon^{-1} \mu^{-1} \norm{\nabla \times r_{2,E}^n}^2_{\varepsilon} \right. \\ +  \varepsilon^{-1} \mu^{-1} \norm{\nabla \times r_{2,H}^{n + \frac{1}{2}}}^2_{\mu} 
\left. + \dfrac{\Delta t^4}{144} \left( \norm{\nabla \nabla \cdot \nabla r_{1,p}^n} ^2_{\varepsilon^{-1}} + \norm{\nabla \cdot \nabla \nabla \cdot r_{1,E}^n}^2_{\varepsilon} + \varepsilon^{-3} \mu^{-3} \norm{\nabla \times \nabla \times \nabla \times r_{1,E}^n}^2_{\varepsilon} \right. \right. \\ \left. \left.+  \varepsilon^{-3} \mu^{-3} \norm{\nabla \times \nabla \times \nabla \times r_{1,H}^{n + \frac{1}{2}}}^2_{\mu} \right) \right].
\end{multline*}
Summing over $n = 1$ to $N-1$, we get:
\begin{multline*}
  \norm{e_p^{N - \frac{1}{2}}}^2_{\varepsilon^{-1}} - \norm{e_p^{ \frac{1}{2}}}^2_{\varepsilon^{-1}} + \norm{e_E^{N - \frac{1}{2}}}^2_{\varepsilon} - \norm{e_E^{\frac{1}{2}}}^2_{\varepsilon} + \norm{e_H^N}^2_{\mu} -  \norm{e_H^1}^2_{\mu} \le \Delta t \left[ \norm{e_p^{N - \frac{1}{2}}}^2_{\varepsilon^{-1}} + \norm{e_p^{ \frac{1}{2}}}^2_{\varepsilon^{-1}} \right. \\ \left. + \norm{e_E^{N - \frac{1}{2}}}^2_{\varepsilon} + \norm{e_E^{\frac{1}{2}}}^2_{\varepsilon} + \norm{e_H^N}^2_{\mu} + \norm{e_H^1}^2_{\mu}\right] + 2 \Delta t \sum\limits_{n = 1}^{N - 2} \left[ \norm{e_p^{n + \frac{1}{2}}}^2_{\varepsilon^{-1}} + \norm{e_E^{n + \frac{1}{2}}}^2_{\varepsilon} + \norm{e_H^{n+1}}^2_{\mu}\right] \\ +
 \Delta t \sum\limits_{n = 1}^{N - 1} \left[ \norm{R_p^n}^2_{\varepsilon^{-1}} + \norm{R_E^n}^2_{\varepsilon} + \norm{R_H^{n + \frac{1}{2}}}^2_{\mu} + \norm{\nabla r_{2,p}^n} ^2_{\varepsilon^{-1}} + \norm{\nabla \cdot r_{2,E}^n}^2_{\varepsilon} + \varepsilon^{-1} \mu^{-1} \norm{\nabla \times r_{2,E}^n}^2_{\varepsilon} \right. \\ \left. +  \varepsilon^{-1} \mu^{-1} \norm{\nabla \times r_{2,H}^{n + \frac{1}{2}}}^2_{\mu} + \dfrac{\Delta t^4}{144} \left( \norm{\nabla \nabla \cdot \nabla r_{1,p}^n} ^2_{\varepsilon^{-1}} + \norm{\nabla \cdot \nabla \nabla \cdot r_{1,E}^n}^2_{\varepsilon} + \varepsilon^{-3} \mu^{-3} \norm{\nabla \times \nabla \times \nabla \times r_{1,E}^n}^2_{\varepsilon} \right. \right. \\ \left. \left. +  \varepsilon^{-3} \mu^{-3} \norm{\nabla \times \nabla \times \nabla \times r_{1,H}^{n + \frac{1}{2}}}^2_{\mu} \right) \right].
\end{multline*}
Similarly, for the initial Equations~\labelcref{eqn:maxwell_p0_lf4,eqn:maxwell_E0_lf4,eqn:maxwell_H0_lf4} which bootstrap our computations, by choosing test functions to be $\widetilde{p} = 2 \Delta t \varepsilon^{-1} \left( e_p^\frac{1}{2} + e_p^0 \right)$, $\widetilde{E} = 2 \Delta t \left( e_E^\frac{1}{2}  + e_E^0 \right)$ and $\widetilde{H} = 4 \Delta t \left( e_H^1 + e_H^0 \right)$, we obtain:
\begin{multline*}
  \norm{e_p^{\frac{1}{2}}}^2_{\varepsilon^{-1}} - \norm{e_p^0}^2_{\varepsilon^{-1}} + \norm{e_E^{\frac{1}{2}}}^2_{\varepsilon} - \norm{e_E^0}^2_{\varepsilon} + \norm{e_H^{1}}^2_{\mu} -  \norm{e_H^0}^2_{\mu} \\ \le \Delta t \left[   \norm{e_p^{\frac{1}{2}}}^2_{\varepsilon^{-1}} + \norm{e_p^0}^2_{\varepsilon^{-1}} + \norm{e_E^{\frac{1}{2}}}^2_{\varepsilon} + \norm{e_E^0}^2_{\varepsilon} + \norm{e_H^{1}}^2_{\mu} +  \norm{e_H^0}^2_{\mu} \right] \\ +
 \Delta t \left[ \norm{R_p^0}^2_{\varepsilon^{-1}} + \norm{R_E^0}^2_{\varepsilon} + \norm{R_H^{\frac{1}{2}}}^2_{\mu} + \norm{\nabla r_{2,p}^0} ^2_{\varepsilon^{-1}} + \norm{\nabla \cdot r_{2,E}^0}^2_{\varepsilon} + \varepsilon^{-1} \mu^{-1} \norm{\nabla \times r_{2,E}^0}^2_{\varepsilon}  \right. \\ 
\left.+  \varepsilon^{-1} \mu^{-1} \norm{\nabla \times r_{2,H}^{\frac{1}{2}}}^2_{\mu} + \dfrac{\Delta t^4}{144} \left( \norm{\nabla \nabla \cdot \nabla r_{1,p}^0} ^2_{\varepsilon^{-1}} + \norm{\nabla \cdot \nabla \nabla \cdot r_{1,E}^0}^2_{\varepsilon} + \varepsilon^{-3} \mu^{-3} \norm{\nabla \times \nabla \times \nabla \times r_{1,E}^0}^2_{\varepsilon}  \right.  \right. \\ \left. \left. +  \varepsilon^{-3} \mu^{-3} \norm{\nabla \times \nabla \times \nabla \times r_{1,H}^{\frac{1}{2}}}^2_{\mu} \right) \right].
\end{multline*}
By next adding the previous two equations, using the initial conditions as in Equation~\eqref{eqn:ICs}, and the positivity of the right hand side terms in the resulting equation, we get that:
\begin{multline*}
  \norm{e_p^{N - \frac{1}{2}}}^2_{\varepsilon^{-1}} + \norm{e_E^{N - \frac{1}{2}}}^2_{\varepsilon} + \norm{e_H^N}^2_{\mu} \le \dfrac{1 + \Delta t}{1 - \Delta t} \left[ \norm{e_p^0}^2_{\varepsilon^{-1}} + \norm{e_E^0}^2_{\varepsilon} + \norm{e_H^0}^2_{\mu} \right] + \\
  \dfrac{\Delta t}{1 - \Delta t} \sum\limits_{n = 0}^{N - 1} \Big[ 2 \left( \norm{e_p^{n + \frac{1}{2}}}^2_{\varepsilon^{-1}} + \norm{e_E^{n + \frac{1}{2}}}^2_{\varepsilon} + \norm{e_H^{n + 1}}^2_{\mu} \right) + \left( \norm{R_p^n}^2_{\varepsilon^{-1}} + \norm{R_E^n}^2_{\varepsilon} + \norm{R_H^{n + \frac{1}{2}}}^2_{\mu} \right. \\ \left.
+ \norm{\nabla r_{2,p}^n} ^2_{\varepsilon^{-1}} + \norm{\nabla \cdot r_{2,E}^n}^2_{\varepsilon} + \varepsilon^{-1} \mu^{-1} \norm{\nabla \times r_{2,E}^n}^2_{\varepsilon} +  \varepsilon^{-1} \mu^{-1} \norm{\nabla \times r_{2,H}^{n + \frac{1}{2}}}^2_{\mu} \right)
+ \dfrac{\Delta t^4}{144} \left( \norm{\nabla \nabla \cdot \nabla r_{1,p}^n} ^2_{\varepsilon^{-1}}  \right. \\ 
\left. + \norm{\nabla \cdot \nabla \nabla \cdot r_{1,E}^n}^2_{\varepsilon} + \varepsilon^{-3} \mu^{-3} \norm{\nabla \times \nabla \times \nabla \times r_{1,E}^n}^2_{\varepsilon} +  \varepsilon^{-3} \mu^{-3} \norm{\nabla \times \nabla \times \nabla \times r_{1,H}^{n + \frac{1}{2}}}^2_{\mu} \right) \Big].
\end{multline*}
Applying the discrete Gronwall inequality, we obtain the following estimate:
\begin{multline*}
  \norm{e_p^{N - \frac{1}{2}}}^2_{\varepsilon^{-1}} + \norm{e_E^{N - \frac{1}{2}}}^2_{\varepsilon} +\norm{e_H^N}^2_{\mu} \le \Bigg[ \dfrac{6 \Delta t}{5} \sum\limits_{n = 0}^{N - 1} \left(\norm{R_p^n}^2_{\varepsilon^{-1}} + \norm{R_E^n}^2_{\varepsilon} + \norm{R_H^{n + \frac{1}{2}}}^2_{\mu} \right. \\ + \left. \norm{\nabla r_{2,p}^n} ^2_{\varepsilon^{-1}} + \norm{\nabla \cdot r_{2,E}^n}^2_{\varepsilon} + \varepsilon^{-1} \mu^{-1} \norm{\nabla \times r_{2,E}^n}^2_{\varepsilon} +  \varepsilon^{-1} \mu^{-1} \norm{\nabla \times r_{2,H}^{n + \frac{1}{2}}}^2_{\mu} \right. \\ 
+ \left. \dfrac{\Delta t^4}{144} \left( \norm{\nabla \nabla \cdot \nabla r_{1,p}^n} ^2_{\varepsilon^{-1}} + \norm{\nabla \cdot \nabla \nabla \cdot r_{1,E}^n}^2_{\varepsilon} + \varepsilon^{-3} \mu^{-3} \norm{\nabla \times \nabla \times \nabla \times r_{1,E}^n}^2_{\varepsilon} \right. \right. \\  \left. \left. +  \varepsilon^{-3} \mu^{-3} \norm{\nabla \times \nabla \times \nabla \times r_{1,H}^{n + \frac{1}{2}}}^2_{\mu} \right)  \right) +
  \dfrac{7}{5} \left(\norm{e_p^0}^2_{\varepsilon^{-1}} + \norm{e_E^0}^2_{\varepsilon}  + \norm{e_H^0}^2_\mu \right)\Bigg] \exp\left( 4 T \right).
\end{multline*}
Now, we need to obtain bounding estimates for each of the Taylor remainder terms and to do so, we first consider the first remainder term corresponding to $p$ and argue as follows:
\begin{align*}
\norm{R^n_p}^2_{\varepsilon^{-1}} &= \dfrac{1}{(24)^2 \Delta t^2} \norm[\bigg]{\int\limits_{\mathclap{t^{n - 1}}}^{\mathclap{t^{n - \frac{1}{2}}}} (t^{n - 1} - s)^4 \dfrac{\partial^5 p}{\partial t^5}(s) ds + \int\limits_{\mathclap{t^{n - \frac{1}{2}}}}^{t^n} (t^n - s)^4 \dfrac{\partial^5 p}{\partial t^5}(s) ds}^2_{\varepsilon^{-1}}, \\
&\le \dfrac{1}{(24)^2 \Delta t^2} \norm[\bigg]{\int\limits_{t^{n - 1}}^{t^n} (t^n - s)^4 \dfrac{\partial^5 p}{\partial t^5}(s) ds}^2_{\varepsilon^{-1}}, \quad \text{(using $t^{n - 1} < t^n$)} \\
&\le \dfrac{1}{(24)^2 \left(\Delta t\right)^2} \int\limits_{t^{n - 1}}^{t^n} (s - t^n)^8 ds \int\limits_{\mathclap{t^{n - 1}}}^{t^n} \norm[\bigg]{\dfrac{\partial^5 p}{\partial t^5}(s)}^2_{\varepsilon^{-1}} ds, \quad \text{(by Cauchy-Schwarz)} \\
&= \dfrac{\Delta t^7}{(24)^2 \cdot 9} \int\limits_{\mathclap{t^{n - 1}}}^{t^n} \norm[\bigg]{\dfrac{\partial^5 p}{\partial t^5}(s)}^2_{\varepsilon^{-1}} ds.
\end{align*}
Summing both sides over $n = 0$ to $N$, we have that:
\begin{equation*}
  \sum\limits_{n = 0}^N \norm{R^n_p}^2_{\varepsilon^{-1}} \le \dfrac{\Delta t^7}{(24)^2 \cdot 9} \int\limits_0^T \norm[\bigg]{\dfrac{\partial^5 p}{\partial t^5}(s)}^2_{\varepsilon^{-1}} ds = \dfrac{\Delta t^7}{(24)^2 \cdot 9} \norm[\bigg]{\dfrac{\partial^5 p}{\partial t^5}}^2_{L^2(0, T; L^2_{\varepsilon^{-1}}(\Omega))}.
\end{equation*}
Similarly, for the other Taylor remainder terms, we get that:
\begin{align*}
\norm{R^n_E}^2_{\varepsilon} & \le \dfrac{\Delta t^7}{(24)^2 \cdot 9} \norm[\bigg]{\dfrac{\partial^5 E}{\partial t^5}}^2_{L^2(0, T; L^2_\varepsilon(\Omega))}, \\
  \sum\limits_{n = 0}^{N - 1} \norm{R^{n + \frac{1}{2}}_H}^2_\mu &\le \dfrac{\Delta t^7}{(24)^2 \cdot 9} \norm[\bigg]{\dfrac{\partial^5 H}{\partial t^5}}^2_{L^2(0, T; L^2_\mu(\Omega))}, \\ \sum\limits_{n = 0}^{N - 1} \norm{\nabla r^n_{2,p}}^2_{\varepsilon^{-1}} &\le  \dfrac{\Delta t^7}{2^7 \cdot 72} \norm[\bigg]{\dfrac{\partial^4 \left(\nabla p \right)}{\partial t^4}}^2_{L^2(0, T; L^2_{\varepsilon^{-1}}(\Omega))}, \\
\sum\limits_{n = 0}^{N - 1} \norm{\nabla \cdot r^n_{2,E}}^2_{\varepsilon} &\le \dfrac{\Delta t^7}{2^7 \cdot 72} \norm[\bigg]{\dfrac{\partial^4 (\nabla \cdot E)}{\partial t^4}}^2_{L^2(0, T; L^2_\varepsilon(\Omega))},  \\
 \sum\limits_{n = 0}^{N - 1} \norm{\nabla \times r^n_{2,E}}^2_{\varepsilon} &\le \dfrac{\Delta t^7}{2^7 \cdot 72} \norm[\bigg]{\dfrac{\partial^4 (\nabla \times E)}{\partial t^4}}^2_{L^2(0, T; L^2_{\varepsilon}(\Omega))}, \\
\sum\limits_{n = 0}^{N - 1} \norm{\nabla \times r^{n + \frac{1}{2}}_{2,H}}^2_\mu &\le \dfrac{\Delta t^7}{2^7 \cdot 72} \norm[\bigg]{\dfrac{\partial^4 (\nabla \times H)}{\partial t^4}}^2_{L^2(0, T; L^2_\mu(\Omega))}, \\ \sum\limits_{n = 0}^{N - 1} \norm{\nabla \nabla \cdot \nabla r^n_{1,p}}^2_{\varepsilon^{-1}} &\le  \dfrac{\Delta t^3}{48} \norm[\bigg]{\dfrac{\partial^2 \left(\nabla \nabla \cdot \nabla p \right)}{\partial t^2}}^2_{L^2(0, T; L^2_{\varepsilon^{-1}}(\Omega))}, \\
\sum\limits_{n = 0}^{N - 1} \norm{\nabla \cdot \nabla \nabla \cdot r^n_{1,E}}^2_{\varepsilon} &\le \dfrac{\Delta t^3}{48} \norm[\bigg]{\dfrac{\partial^2 (\nabla \cdot \nabla \nabla \cdot E)}{\partial t^2}}^2_{L^2(0, T; L^2_\varepsilon(\Omega))},  \\
 \sum\limits_{n = 0}^{N - 1} \norm{\nabla \times \nabla \times \nabla \times r^n_{1,E}}^2_{\varepsilon} &\le \dfrac{\Delta t^3}{48} \norm[\bigg]{\dfrac{\partial^2 (\nabla \times \nabla \times \nabla \times E)}{\partial t^2}}^2_{L^2(0, T; L^2_{\varepsilon}(\Omega))}, \\
\sum\limits_{n = 0}^{N - 1} \norm{\nabla \times \nabla \times \nabla \times r^{n + \frac{1}{2}}_{1,H}}^2_\mu &\le \dfrac{\Delta t^3}{48} \norm[\bigg]{\dfrac{\partial^2 (\nabla \times \nabla \times \nabla \times H)}{\partial t^2}}^2_{L^2(0, T; L^2_\mu(\Omega))}.
\end{align*}
Finally, using the regularity assumptions for $p$, $E$ and $H$, and $1$- and $2$-norm equivalence, we obtain our required result:
\[
  \norm{e_p^{N - \frac{1}{2}}}_{\varepsilon^{-1}} + \norm{e_E^{N - \frac{1}{2}}}_{\varepsilon} + \norm{e_H^N}_{\mu} \le C \left[ \Delta t^4 + \norm{e_p^0}_{\varepsilon^{-1}} + \norm{e_E^0}_{\varepsilon} + \norm{e_H^0}_{\mu} \right]. \qedhere
\]
\end{proof}

\subsection{Error Estimate for Full Discretization}

We now present the error analysis for the full discretization of the three-field formulation of the Maxwell's equations using a compatible sequence of arbitrary order de Rham finite elements in conjunction with the implicit LF$_4$ time integration method. To do so, first we let $\Pi_h^0$, $\Pi_h^1$ and $\Pi_h^2$ denote the respective smoothed $L^2$ projection operators as in Theorem~\labelcref{thm:arfawismoothedprojection}. That is, let $\Pi_h^0: \mathring{H}^1_{\varepsilon^{-1}}(\Omega) \longto U_h$, $\Pi_h^1: \mathring{H}_{\varepsilon}(\curl; \Omega) \longto V_h$ and $\Pi_h^2: \mathring{H}_{\mu}(\divgn; \Omega) \longto W_h$ denote these smoothed $L^2$ projection operators. Details of these operators is now standard and can be found in many places such as \cite{Schoberl2008,Christiansen2007}, \cite[Lemma 4.3.8]{Brenner2008},~\cite[Theorem 5.3]{ArFaWi2006}, and~\cite[Lemma 11.9, Corollary 11.11, Theorem 16.10, Theorem 17.5]{ErGu2021}. 

We now define the errors for $p$, $E$ and $H$ at time $(n + 1/2) \Delta t$ or $n \Delta t$ under the full discretization to be the following and this is essentially identical to the sequence of analysis for the full error as in \cite[Section 5]{ArKa2025} for the method described in that work. So, we  have that:
\begin{alignat}{2}
  e_{p_h}^{n + \frac{1}{2}} &\coloneq p(t^{n + \frac{1}{2}}) - p_h^{n + \frac{1}{2}} &&= \eta^{n+\frac{1}{2}} - \eta_h^{n+\frac{1}{2}}, \label{eqn:p_fullerror_lf4} \\
  e_{E_h}^{n + \frac{1}{2}} &\coloneq E(t^{n + \frac{1}{2}}) - E_h^{n + \frac{1}{2}} &&= \zeta^{n + \frac{1}{2}} - \zeta_h^{n + \frac{1}{2}}, \label{eqn:E_fullerror_lf4} \\
  e_{H_h}^n &\coloneq H(t^n) - H_h^n &&= \xi^n - \xi_h^n, \label{eqn:H_fullerror_lf4}
\end{alignat}
and in which we now have the following definitions for the newly introduced terms:
\begin{alignat}{3}
  \eta^{n + \frac{1}{2}} &\coloneq p(t^{n + \frac{1}{2}}) - \Pi_h^0 p(t^{n + \frac{1}{2}}), &&\qquad \eta_h^{n + \frac{1}{2}} &&\coloneq p_h^{n + \frac{1}{2}}  - \Pi_h^0 p(t^{n + \frac{1}{2}}), \label{eqn:p_fullerror_sub_lf4} \\
  \zeta^{n + \frac{1}{2}} &\coloneq E(t^{n + \frac{1}{2}}) - \Pi_h^1 E(t^{n + \frac{1}{2}}), &&\qquad \zeta_h^{n + \frac{1}{2}} &&\coloneq E_h^{n + \frac{1}{2}} - \Pi_h^1 E(t^{n + \frac{1}{2}}), \label{eqn:E_fullerror_sub_lf4} \\
  \xi^n &\coloneq H(t^n) - \Pi_h^2 H(t^n), &&\qquad \xi_h^n &&\coloneq H_h^n - \Pi_h^2 H(t^n). \label{eqn:H_fullerror_sub_lf4}
\end{alignat}
For the LF$_4$ scheme as in Equations~\labelcref{eqn:maxwell_p_lf4,eqn:maxwell_E_lf4,eqn:maxwell_H_lf4}, using a de Rham sequence of finite dimensional subspaces of the corresponding function spaces for the spatial discretization of $(p^{n + \frac{1}{2}}, E^{n + \frac{1}{2}}, H^{n + 1})$, we obtain the following discrete problem: find $(p_h^{n + \frac{1}{2}}, E_h^{n + \frac{1}{2}}, H_h^{n + 1}) \in U_h \times V_h \times W_h \subseteq \mathring{H}_{\varepsilon^{-1}}^1 \times \mathring{H}_{\varepsilon}(\curl; \Omega) \times \mathring{H}_{\mu}(\divgn; \Omega)$ such that:
\begin{subequations}
\begin{equation}
  \aInnerproduct{\dfrac{p_h^{n + \frac{1}{2}} - p_h^{n - \frac{1}{2}}}{\Delta t}}{\widetilde{p}} - \aInnerproduct{\dfrac{\varepsilon}{2} \left( E_h^{n + \frac{1}{2}} + E_h^{n - \frac{1}{2}} \right)}{\nabla \widetilde{p}} + \dfrac{\Delta t^2}{12} \aInnerproduct{\dfrac{\varepsilon}{2} \nabla \nabla \cdot \left( E_h^{n + \frac{1}{2}} + E_h^{n - \frac{1}{2}} \right)}{\nabla \widetilde{p}} =0, \label{eqn:maxwell_p_lf4_full}
  \end{equation} \\
  \begin{multline}
   \aInnerproduct{\dfrac{1}{2} \nabla \left(p_h^{n + \frac{1}{2}} + p_h^{n - \frac{1}{2}} \right)}{\widetilde{E}} - \dfrac{\Delta t^2}{12} \aInnerproduct{\dfrac{1}{2} \nabla \left(p_h^{n + \frac{1}{2}} + p_h^{n - \frac{1}{2}} \right)}{\nabla \nabla \cdot \widetilde{E}} + \aInnerproduct{\varepsilon \dfrac{E_h^{n + \frac{1}{2}} - E_h^{n - \frac{1}{2}}}{\Delta t}}{\widetilde{E}} \\ - \aInnerproduct{\dfrac{1}{2} \left( H_h^{n + 1} + H_h^n \right)}{\nabla \times \widetilde{E}}  - \dfrac{\Delta t^2}{12} \aInnerproduct{\dfrac{1}{2} \mu^{-1}\varepsilon^{-1} \nabla \times \nabla \times \left( H_h^{n + 1} + H_h^n \right)}{\nabla \times \widetilde{E}} = 0, \label{eqn:maxwell_E_lf4_full} 
     \end{multline} \\
     \begin{multline}
  \aInnerproduct{\mu \dfrac{H_h^{n + 1} - H_h^n}{\Delta t}}{\widetilde{H}} +  \aInnerproduct{\dfrac{1}{2} \nabla \times \left( E_h^{n + \frac{1}{2}} + E_h^{n - \frac{1}{2}} \right)}{\widetilde{H}} \\ +  \dfrac{\Delta t^2}{12} \aInnerproduct{\dfrac{1}{2} \varepsilon^{-1}\mu^{-1} \nabla \times \left( E_h^{n + \frac{1}{2}} + E_h^{n - \frac{1}{2}} \right)}{\nabla \times  \nabla \times \widetilde{H}} = 0, \label{eqn:maxwell_H_lf4_full}
\end{multline}
\end{subequations}
for all $(\widetilde{p}, \widetilde{E}, \widetilde{H}) \in U_h \times V_h \times W_h$, $n = 1, \dots, N - 1$. The $n = 0$ bootstrapping as in Equations~\labelcref{eqn:maxwell_p0_lf4,eqn:maxwell_E0_lf4,eqn:maxwell_H0_lf4} leads also to the following discrete problem: find $(p_h^{\frac{1}{2}}, E_h^{\frac{1}{2}}, H_h^1) \in U_h \times V_h \times W_h \subseteq \mathring{H}_{\varepsilon^{-1}}^1 \times \mathring{H}_{\varepsilon}(\curl; \Omega) \times \mathring{H}_{\mu}(\divgn; \Omega)$ such that:
\begin{subequations}
  \begin{equation}
    \aInnerproduct{\dfrac{p_h^{\frac{1}{2}} - p_h^0}{\Delta t/2}}{\widetilde{p}} - \dfrac{1}{2} \aInnerproduct{\dfrac{\varepsilon}{2} \left( E_h^{\frac{1}{2}} + E_h^0 \right)}{\nabla \widetilde{p}} + \dfrac{1}{8} \cdot \dfrac{\Delta t^2}{12} \aInnerproduct{\dfrac{\varepsilon}{2} \nabla \nabla \cdot \left( E_h^{\frac{1}{2}} + E_h^0 \right)}{\nabla \widetilde{p}} = 0, \label{eqn:maxwell_p0_lf4_full}
    \end{equation} \\
    \begin{multline}
  \dfrac{1}{2}  \aInnerproduct{\dfrac{1}{2} \nabla \left(  p_h^{\frac{1}{2}} +  p_h^0 \right)}{\widetilde{E}} - \dfrac{1}{8} \cdot \dfrac{\Delta t^2}{12} \aInnerproduct{\dfrac{1}{2} \nabla \left(  p_h^{\frac{1}{2}} +  p_h^0 \right)}{\nabla \nabla \cdot \widetilde{E}} + \aInnerproduct{\varepsilon \dfrac{E_h^{\frac{1}{2}} - E_h^0}{\Delta t/2}}{\widetilde{E}}\\  - \aInnerproduct{\dfrac{1}{2} \left( H_h^1 + H_h^0 \right)}{\nabla \times \widetilde{E}} -\dfrac{1}{4} \cdot \dfrac{\Delta t^2}{12}  \aInnerproduct{\dfrac{1}{2} \nabla \times \nabla \times \left( H_h^1 + H_h^0 \right)}{\nabla \times \widetilde{E}} = 0, \label{eqn:maxwell_E0_lf4_full} 
  \end{multline} \\
  \begin{equation}
    \aInnerproduct{\mu \dfrac{H_h^1 - H_h^0}{\Delta t}}{\widetilde{H}} + \dfrac{1}{2} \aInnerproduct{\dfrac{1}{2} \nabla \times \left(E_h^{\frac{1}{2}} + E_h^0 \right)}{\widetilde{H}} + \dfrac{1}{8} \cdot \dfrac{\Delta t^2}{12} \aInnerproduct{\dfrac{1}{2} \nabla \times \left(E_h^{\frac{1}{2}} + E_h^0 \right)}{ \nabla \times \nabla \times  \widetilde{H}} = 0, \label{eqn:maxwell_H0_lf4_full}
  \end{equation}
for all $(\widetilde{p}, \widetilde{E}, \widetilde{H}) \in U_h \times V_h \times W_h$ given $(p_h^0, E_h^0, H_h^0) \in U_h \times V_h \times W_h$. For the initial conditions, we set: 
\begin{equation}
p_h^0 \coloneq \Pi_h^0 p_0, \, E_h^0 \coloneq \Pi_h^1 E_0 \text{~and~} H_h^0 \coloneq \Pi_h^2 H_0. \label{eqn:initial_lf4}
  \end{equation}
\end{subequations}

With this setup, we now state and prove our theorem for the convergence of errors in the full discretization of our system of Maxwell's equations using LF$_4$ and arbitrary order de Rham finite elements.

\begin{theorem}[Full Error Estimate]\label{thm:full_error_estmt_lf4}
Let $p \in C^5(0, T; \mathring{H}^1_{\varepsilon^{-1}}(\Omega))$, $E \in C^5(0, T; \mathring{H}_{\varepsilon}(\curl; \Omega))$, and $H \in C^5(0, T; \mathring{H}_{\mu}(\divgn; \Omega))$ be the solution to the variational formulation of the Maxwell's equations as in Equations~\labelcref{eqn:maxwell_p_wf,eqn:maxwell_E_wf,eqn:maxwell_H_wf} assuming sufficient regularity, and let $(p_h^{n + \frac{1}{2}}, E_h^{n + \frac{1}{2}}, H_h^{n + 1})$ be the solution of the fully discretized Maxwell's equations using the LF$_4$ scheme as in Equations~\labelcref{eqn:maxwell_p_lf4_full,eqn:maxwell_E_lf4_full,eqn:maxwell_H_lf4_full,eqn:maxwell_p0_lf4_full,eqn:maxwell_E0_lf4_full,eqn:maxwell_H0_lf4_full,eqn:initial_lf4}. If the time step $\Delta t > 0$ and the mesh parameter $h > 0$ are sufficiently small, then there exists a positive bounded constant $C$ independent of both $\Delta t$ and $h$ such that the following error estimate holds:
\[
  \norm{e_{p_h}^{N - \frac{1}{2}}}_{\varepsilon^{-1}} + \norm{e_{E_h}^{N - \frac{1}{2}}}_{\varepsilon} + \norm{e_{H_h}^N}_{\mu} \le C \left[ \Delta t^4 + h^r + h^r \Delta t^4 \right],
\]
where the finite element subspaces $U_h$, $V_h$ and $W_h$ are each spanned by their respective Whitney form basis of polynomial order $r \ge 1$.
\end{theorem}

\begin{proof}
 First, we shall subtract the set of equations for the full discretization as in Equations~\labelcref{eqn:maxwell_p_lf4_full,eqn:maxwell_E_lf4_full,eqn:maxwell_H_lf4_full} from Equations~\labelcref{eqn:remainder_p_lf4,eqn:remainder_E_lf4,eqn:remainder_H_lf4},  and then use the error terms in Equations~\labelcref{eqn:p_fullerror_lf4,eqn:E_fullerror_lf4,eqn:H_fullerror_lf4} and thereby obtain:
 \begin{multline*}
\aInnerproduct{\dfrac{e_{p_h}^{n + \frac{1}{2}} - e_{p_h}^{n - \frac{1}{2}}}{\Delta t}}{\widetilde{p}} - \dfrac{1}{2} \aInnerproduct{ \varepsilon \left(e_{E_h}^{n + \frac{1}{2}} + e_{E_h}^{n - \frac{1}{2}} \right)}{\nabla \widetilde{p}} + \dfrac{\Delta t^2}{24} \aInnerproduct{\varepsilon \nabla \nabla \cdot \left( e_{E_h}^{n + \frac{1}{2}} + e_{E_h}^{n - \frac{1}{2}} \right)}{\nabla \widetilde{p}} \\ = \aInnerproduct{R_p^{n} - \varepsilon \nabla \cdot r_{2,E}^n - \dfrac{\Delta t^2}{12} \varepsilon \nabla \cdot \nabla \nabla \cdot r_{1,E}^n }{\widetilde{p}}, 
 \end{multline*} 
\begin{multline*}
 \dfrac{1}{2} \aInnerproduct{\nabla \left(e_{p_h}^{n + \frac{1}{2}} + e_{p_h}^{n - \frac{1}{2}}\right)}{\widetilde{E}}  - \dfrac{\Delta t^2}{24} \aInnerproduct{\nabla \left(e_{p_h}^{n + \frac{1}{2}} + e_{p_h}^{n - \frac{1}{2}} \right)}{\nabla \nabla \cdot \widetilde{E}} + \aInnerproduct{\varepsilon \dfrac{e_{E_h}^{n + \frac{1}{2}} - e_{E_h}^{n - \frac{1}{2}}}{\Delta t}}{\widetilde{E}} \\ - \dfrac{1}{2} \aInnerproduct{e_{H_h}^{n + 1} + e_{H_h}^n}{\nabla \times \widetilde{E}}  - \dfrac{\Delta t^2}{24} \aInnerproduct{ \mu^{-1}\varepsilon^{-1} \nabla \times \nabla \times \left( e_{H_h}^{n + 1} + e_{H_h}^n \right)}{\nabla \times \widetilde{E}} \\ = \aInnerproduct{\varepsilon R_E^n + \nabla r_{2,p}^n - \nabla \times r_{2,H}^{n+\frac{1}{2}} - \dfrac{\Delta t^2}{12} \nabla \nabla \cdot \nabla r_{1,p}^n - \dfrac{\Delta t^2}{12}  \mu^{-1}\varepsilon^{-1} \nabla \times \nabla \times \nabla \times r_{1,H}^{n+\frac{1}{2}}}{\widetilde{E}},
  \end{multline*} 
 \begin{multline*}
\aInnerproduct{\mu \dfrac{e_{H_h}^{n + 1} - e_{H_h}^{n}}{\Delta t}}{\widetilde{H}} +  \dfrac{1}{2}\aInnerproduct{\nabla \times\left(e_{E_h}^{n+\frac{1}{2}} + e_{E_h}^{n - \frac{1}{2}} \right)}{\widetilde{H}} +  \dfrac{\Delta t^2}{24} \aInnerproduct{\varepsilon^{-1}\mu^{-1} \nabla \times \left(e_{E_h}^{n + \frac{1}{2}} + e_{E_h}^{n - \frac{1}{2}} \right)}{\nabla \times  \nabla \times \widetilde{H}} \\ = \aInnerproduct{\mu R_H^{n + \frac{1}{2}} + \nabla \times r_{2,E}^n + \dfrac{\Delta t^2}{12} \varepsilon^{-1}\mu^{-1} \nabla \times \nabla \times \nabla \times r_{1,E}^n}{\widetilde{H}}.
 \end{multline*}
 Next, using the values of the error terms $e_{p_h}^n$, $e_{E_h}^n$ and $e_{H_h}^n$ as in Equations~\labelcref{eqn:p_fullerror_lf4,eqn:E_fullerror_lf4,eqn:H_fullerror_lf4} in the above equations, we get:
 \begin{multline*}
  \aInnerproduct{\dfrac{\left(\eta^{n + \frac{1}{2}} - \eta^{n - \frac{1}{2}}\right) - \left( \eta^{n + \frac{1}{2}}_h - \eta^{n - \frac{1}{2}}_h \right)}{\Delta t}}{\widetilde{p}}  - \dfrac{1}{2} \aInnerproduct{ \varepsilon \left( \left( \zeta^{n + \frac{1}{2}} + \zeta^{n - \frac{1}{2}} \right) - \left(\zeta_h^{n + \frac{1}{2}} + \zeta^{n - \frac{1}{2}}_h \right) \right)}{\nabla \widetilde{p}} \\ + \dfrac{\Delta t^2}{24} \aInnerproduct{\varepsilon \nabla \nabla \cdot \left( \left( \zeta^{n + \frac{1}{2}} + \zeta^{n - \frac{1}{2}} \right) - \left(\zeta_h^{n + \frac{1}{2}} + \zeta^{n - \frac{1}{2}}_h \right) \right)}{\nabla \widetilde{p}} = \aInnerproduct{R_p^{n} - \varepsilon \nabla \cdot r_{2,E}^n - \dfrac{\Delta t^2}{12} \varepsilon \nabla \cdot \nabla \nabla \cdot r_{1,E}^n }{\widetilde{p}},
\end{multline*}
\begin{multline*}
  \dfrac{1}{2} \aInnerproduct{\nabla \left( \left( \eta^{n + \frac{1}{2}} + \eta^{n - \frac{1}{2}} \right) - \left( \eta^{n + \frac{1}{2}}_h + \eta^{n - \frac{1}{2}}_h \right) \right)}{\widetilde{E}} - \dfrac{\Delta t^2}{24} \aInnerproduct{\nabla \left( \left( \eta^{n + \frac{1}{2}} + \eta^{n - \frac{1}{2}} \right) - \left( \eta^{n + \frac{1}{2}}_h + \eta^{n - \frac{1}{2}}_h \right)\right)}{\nabla \nabla \cdot \widetilde{E}} \\ + \aInnerproduct{\varepsilon \dfrac{\left( \zeta^{n + \frac{1}{2}} - \zeta^{n - \frac{1}{2}} \right) - \left(\zeta^{n + \frac{1}{2}}_h - \zeta^{n - \frac{1}{2}}_h \right)}{\Delta t}}{\widetilde{E}} \, -
  \dfrac{1}{2} \aInnerproduct{\left( \left( \xi^{n + 1} + \xi^{n} \right) - \left(\xi^{n + 1}_h - \xi^{n}_h \right) \right)}{\nabla \times \widetilde{E}} \\ - \dfrac{\Delta t^2}{24} \aInnerproduct{ \mu^{-1}\varepsilon^{-1} \nabla \times \nabla \times \left( \left( \xi^{n + 1} + \xi^{n} \right) - \left(\xi^{n + 1}_h - \xi^{n}_h \right) \right)}{\nabla \times \widetilde{E}} \\ = \aInnerproduct{\varepsilon R_E^n + \nabla r_{2,p}^n - \nabla \times r_{2,H}^{n+\frac{1}{2}} - \dfrac{\Delta t^2}{12} \nabla \nabla \cdot \nabla r_{1,p}^n - \dfrac{\Delta t^2}{12}  \mu^{-1}\varepsilon^{-1} \nabla \times \nabla \times \nabla \times r_{1,H}^{n+\frac{1}{2}}}{\widetilde{E}},
\end{multline*}
\begin{multline*}
  \aInnerproduct{\mu \dfrac{\left(\xi^{n + 1} - \xi^{n} \right) - \left( \xi^{n + 1}_h - \xi^{n}_h \right)}{\Delta t}}{\widetilde{H}}  + \dfrac{1}{2} \aInnerproduct{\nabla \times \left( \left( \zeta^{n + \frac{1}{2}} + \zeta^{n - \frac{1}{2}} \right) - \left( \zeta^{n + \frac{1}{2}}_h + \zeta^{n - \frac{1}{2}}_h \right) \right)}{\widetilde{H}} \\ +  \dfrac{\Delta t^2}{24} \aInnerproduct{\varepsilon^{-1}\mu^{-1} \nabla \times \left(\left( \zeta^{n + \frac{1}{2}} + \zeta^{n - \frac{1}{2}} \right) - \left( \zeta^{n + \frac{1}{2}}_h + \zeta^{n - \frac{1}{2}}_h \right) \right)}{\nabla \times  \nabla \times \widetilde{H}} \\ = \aInnerproduct{\mu R_H^{n + \frac{1}{2}} + \nabla \times r_{2,E}^n + \dfrac{\Delta t^2}{12} \varepsilon^{-1}\mu^{-1} \nabla \times \nabla \times \nabla \times r_{1,E}^n}{\widetilde{H}}.
\end{multline*}
Since these equations are true for all $(\widetilde{p}, \widetilde{E}, \widetilde{H}) \in U_h \times V_h \times W_h$, we choose $\widetilde{p} = -2 \Delta t \varepsilon^{-1} \left( \eta_h^{n + \frac{1}{2}} + \eta_h^{n - \frac{1}{2}}\right)$, $\widetilde{E} = -2 \Delta t \left( \zeta_h^{n + \frac{1}{2}} + \zeta_h^{n - \frac{1}{2}} \right)$ and $\widetilde{H} = -2 \Delta t \left( \xi_h^n + \xi_h^{n - 1} \right)$ and using the fact that $\nabla U_h \subseteq V_h$ and $\nabla \times V_h \subseteq W_h$, we obtain:
\begin{multline}
2 \ainnerproduct{\varepsilon^{-1} \left( \eta^{n + \frac{1}{2}}_h - \eta^{n - \frac{1}{2}}_h \right)}{\eta^{n + \frac{1}{2}}_h + \eta^{n - \frac{1}{2}}_h} + 2 \ainnerproduct{\varepsilon \left( \zeta^{n + \frac{1}{2}}_h - \zeta^{n - \frac{1}{2}}_h \right)}{\zeta^{n + \frac{1}{2}}_h + \zeta^{n - \frac{1}{2}}_h} \, + 2 \ainnerproduct{\mu \left( \xi^n_h - \xi^{n - 1}_h \right)}{\xi^n_h + \xi^{n - 1}_h} \\ =
2 \ainnerproduct{\varepsilon^{-1} \left( \eta^{n + \frac{1}{2}} - \eta^{n -  \frac{1}{2}} \right)}{\eta^{n + \frac{1}{2}}_h + \eta^{n - \frac{1}{2}}_h} \, + 2 \ainnerproduct{\varepsilon \left( \zeta^{n + \frac{1}{2}} - \zeta^{n - \frac{1}{2}} \right)}{\zeta^{n + \frac{1}{2}}_h + \zeta^{n - \frac{1}{2}}_h} + 2 \ainnerproduct{\mu \left( \xi^n - \xi^{n - 1} \right)}{\xi^n_h + \xi^{n - 1}_h} \, \\ +
2 \Delta t \ainnerproduct{- R_p^n + \varepsilon \nabla \cdot r_{2,E}^n + \dfrac{\Delta t^2}{12}  \varepsilon \nabla \cdot \nabla \nabla \cdot r_{1,E}^n}{\varepsilon^{-1} \left( \eta^{n + \frac{1}{2}}_h + \eta^{n - \frac{1}{2}}_h \right)} + 2 \Delta t \ainnerproduct{-\varepsilon R_E^n - \nabla r_{2,p}^n + \nabla \times r_{2,H}^{n + \frac{1}{2}} \\ + \dfrac{\Delta t^2}{12} \nabla \nabla \cdot \nabla r_{1,p}^n + \dfrac{\Delta t^2}{12} \mu^{-1} \varepsilon^{-1} \nabla \times \nabla \times \nabla \times r_{1,H}^{n + \frac{1}{2}}}{\zeta^{n + \frac{1}{2}}_h + \zeta^{n - \frac{1}{2}}_h}  + 2 \Delta t \ainnerproduct{-\mu R_H^{n + \frac{1}{2}} - \nabla \times r_{2,E}^n  \\+ \dfrac{\Delta t^2}{12} \mu^{-1} \varepsilon^{-1} \nabla \times \nabla \times \nabla \times r_{1,E}^n}{\xi^n_h + \xi^{n - 1}_h}, \label{eqn:suberror_p+E+H_lf4}
\end{multline}
Consider that  $\varepsilon^{-1} \left(\eta^{n + \frac{1}{2}} - \eta^{n - \frac{1}{2}}\right) = \varepsilon^{-1} \left(I - \Pi_h^0\right) \left( p(t^{n + \frac{1}{2}}) - p(t^{n - \frac{1}{2}})\right)$ by Equation~\eqref{eqn:p_fullerror_sub_lf4}. Using the Taylor theorem with remainder as in Theorem~\ref{thm:dscrt_error_estmt_lf4}, applying the Cauchy-Schwarz, AM-GM, and Triangle inequalities, we have the following resulting inequality:
\begin{align*}
2 \ainnerproduct{\varepsilon^{-1} \left( \eta^{n + \frac{1}{2}} - \eta^{n - \frac{1}{2}} \right)}{\eta^{n + \frac{1}{2}}_h + \eta^{n - \frac{1}{2}}_h} & = 2 \Delta t \ainnerproduct{\varepsilon^{-1} \left( I - \Pi_h^0 \right) \left( \dfrac{\partial p}{\partial t} (t^n) +  \dfrac{\Delta t^2}{24} \dfrac{\partial^3 p}{\partial t^3 (t^n) }\right)}{\eta^{n + \frac{1}{2}}_h + \eta^{n - \frac{1}{2}}_h} \\ & + 2 \Delta t \ainnerproduct{\varepsilon^{-1} \left( I - \Pi_h^0 \right) R_p^n}{\eta^{n + \frac{1}{2}}_h + \eta^{n - \frac{1}{2}}_h} \\ 
 & \le \Delta t \bigg[ \norm[\bigg]{(I - \Pi_h^0) \dfrac{\partial p}{\partial t}(t^n)}^2_{\varepsilon^{-1}} \!\! + \dfrac{\Delta t^4}{(24)^2} \norm[\bigg]{(I - \Pi_h^0) \dfrac{\partial^3 p}{\partial t^3}(t^n)}^2_{\varepsilon^{-1}} \!\! \\ & +  \norm[\bigg]{(I - \Pi_h^0) R_p^n}^2_{\varepsilon^{-1}} \bigg] \! + 6 \Delta t \bigg[ \norm{\eta_h^{n + \frac{1}{2}}}^2_{\varepsilon^{-1}} + \norm{\eta_h^{n - \frac{1}{2}}}^2_{\varepsilon^{-1}} \bigg].
\end{align*}
Similarly, using Equations~\labelcref{eqn:E_fullerror_sub_lf4,eqn:H_fullerror_sub_lf4} for the error terms for $E$ and $H$, we obtain:
\begin{gather*}
\begin{split}
2  \ainnerproduct{\varepsilon \left( \zeta^{n + \frac{1}{2}} - \zeta^{n - \frac{1}{2}} \right)}{\zeta^{n + \frac{1}{2}}_h + \zeta^{n - \frac{1}{2}}_h} \le \Delta t \bigg[ \norm[\bigg]{(I - \Pi_h^1) \dfrac{\partial E}{\partial t}(t^n)}^2_{\varepsilon} \!\! + \dfrac{\Delta t^4}{(24)^2} \norm[\bigg]{(I - \Pi_h^1) \dfrac{\partial^3 E}{\partial t^3}(t^n)}^2_{\varepsilon} \!\! \\  +  \norm[\bigg]{(I - \Pi_h^1) R_E^n}^2_{\varepsilon} \bigg] \! + 6 \Delta t \bigg[ \norm{\zeta_h^{n + \frac{1}{2}}}^2_{\varepsilon} + \norm{\zeta_h^{n - \frac{1}{2}}}^2_{\varepsilon} \bigg],
\end{split} \\ 
\begin{split}
2 \ainnerproduct{\mu \left( \xi^{n + 1} - \xi^n \right)}{\xi^{n + 1}_h + \xi^n_h} \le \Delta t \bigg[ \norm[\bigg]{(I - \Pi_h^2) \dfrac{\partial H}{\partial t}(t^{n + \frac{1}{2}})}^2_{\mu} \!\! + \dfrac{\Delta t^4}{(24)^2} \norm[\bigg]{(I - \Pi_h^2) \dfrac{\partial^3 H}{\partial t^3}(t^{n + \frac{1}{2}})}^2_{\mu} \!\! \\  +  \norm[\bigg]{(I - \Pi_h^2) R_H^{n + \frac{1}{2}}}^2_{\mu} \bigg] \! + 6 \Delta t \bigg[ \norm{\xi_h^{n + 1}}^2_{\mu} + \norm{\xi_h^n}^2_{\mu} \bigg].
\end{split}
\end{gather*}
Using these inequalities for $\eta$, $\zeta$ and $\xi$ in Equation~\eqref{eqn:suberror_p+E+H_lf4}, we thus obtain the following estimate:
\begin{multline*}
2 \bigg[ \norm{\eta_h^{n + \frac{1}{2}}}^2_{\varepsilon^{-1}} - \norm{\eta_h^{n - \frac{1}{2}}}^2_{\varepsilon^{-1}} + \norm{\zeta_h^{n + \frac{1}{2}}}^2_{\varepsilon} - \norm{\zeta_h^{n - \frac{1}{2}}}^2_{\varepsilon} + \norm{\xi_h^{n+1}}^2_{\mu} - \norm{\xi_h^n}^2_{\mu} \bigg] \le \\
\Delta t \bigg[ \norm[\bigg]{(I - \Pi_h^0) \dfrac{\partial p}{\partial t}(t^n)}^2_{\varepsilon^{-1}} + \dfrac{\Delta t^4}{(24)^2} \norm[\bigg]{(I - \Pi_h^0) \dfrac{\partial^3 p}{\partial t^3}(t^n)}^2_{\varepsilon^{-1}} + \norm[\bigg]{(I - \Pi_h^0) R_p^n}^2_{\varepsilon^{-1}} \\
+ \norm[\bigg]{(I - \Pi_h^1) \dfrac{\partial E}{\partial t}(t^n)}^2_{\varepsilon} + \dfrac{\Delta t^4}{(24)^2} \norm[\bigg]{(I - \Pi_h^1) \dfrac{\partial^3 E}{\partial t^3}(t^n)}^2_{\varepsilon} + \norm[\bigg]{(I - \Pi_h^1) R_E^n}^2_{\varepsilon} + \\
  \norm[\bigg]{(I - \Pi_h^2) \dfrac{\partial H}{\partial t}(t^{n + \frac{1}{2}})}^2_\mu + \dfrac{\Delta t^4}{(24)^2} \norm[\bigg]{(I - \Pi_h^2) \dfrac{\partial^3 H}{\partial t^3}(t^{n + \frac{1}{2}})}^2_{\mu} + \norm[\bigg]{(I - \Pi_h^2) R_H^{n + \frac{1}{2}}}^2_\mu  \\ + 8 \big( \norm{\eta_h^n}^2_{\varepsilon^{-1}} + \norm{\eta_h^{n - 1}}^2_{\varepsilon^{-1}} + \norm{\zeta_h^n}^2_{\varepsilon} + \norm{\zeta_h^{n - 1}}^2_{\varepsilon} + \norm{\xi_h^n}^2_\mu + \norm{\xi_h^{n - 1}}^2_\mu \big)  + \norm{R_p^n}^2_{\varepsilon^{-1}} + \norm{R_E^n}^2_{\varepsilon} \\ + \norm{R_H^{n + \frac{1}{2}}}^2_{\mu} + \norm{\nabla r_{1,p}^n} ^2_{\varepsilon^{-1}} + \norm{\nabla \cdot r_{1,E}^n}^2_{\varepsilon} + \varepsilon^{-1} \mu^{-1} \norm{\nabla \times r_{1,E}^n}^2_{\varepsilon} +  \varepsilon^{-1} \mu^{-1} \norm{\nabla \times r_{1,H}^{n + \frac{1}{2}}}^2_{\mu} \\
+ \dfrac{\Delta t^4}{144} \left( \norm{\nabla \nabla \cdot \nabla r_{2,p}^n} ^2_{\varepsilon^{-1}} + \norm{\nabla \cdot \nabla \nabla \cdot r_{2,E}^n}^2_{\varepsilon} + \varepsilon^{-3} \mu^{-3} \norm{\nabla \times \nabla \times \nabla \times r_{2,E}^n}^2_{\varepsilon}  \right. \\ \left. +  \varepsilon^{-3} \mu^{-3} \norm{\nabla \times \nabla \times \nabla \times r_{2,H}^{n + \frac{1}{2}}}^2_{\mu} \right)  \bigg].
\end{multline*}
Summing from $n = 1$ to $N-1$, we get:
\begin{multline*}
  \norm{\eta_h^{N - \frac{1}{2}}}^2_{\varepsilon^{-1}} + \norm{\zeta_h^{N - \frac{1}{2}}}^2_{\varepsilon} + \norm{\xi_h^N}^2_{\mu} \le 8 \Delta t \sum\limits_{n = 1}^{N-1} \bigg[ \norm{\eta_h^{n + \frac{1}{2}}}^2_{\varepsilon^{-1}} +\norm{\zeta_h^{n + \frac{1}{2}}}^2_{\varepsilon} + \norm{\xi_h^n}^2_{\mu} \bigg] \\ + \Delta t \sum\limits_{n = 0}^{N} \bigg[ \norm[\bigg]{(I - \Pi_h^0) \dfrac{\partial p}{\partial t}(t^n)}^2_{\varepsilon^{-1}} + \dfrac{\Delta t^4}{(24)^2} \norm[\bigg]{(I - \Pi_h^0) \dfrac{\partial^3 p}{\partial t^3}(t^n)}^2_{\varepsilon^{-1}} + 
  \norm[\bigg]{(I - \Pi_h^0) R_p^n}^2_{\varepsilon^{-1}} \\ + \norm[\bigg]{(I - \Pi_h^1) \dfrac{\partial E}{\partial t}(t^{n - \frac{1}{2}})}^2_{\varepsilon}  + \dfrac{\Delta t^4}{(24)^2} \norm[\bigg]{(I - \Pi_h^1) \dfrac{\partial^3 E}{\partial t^3}(t^n)}^2_{\varepsilon} + \norm[\bigg]{(I - \Pi_h^1) R_E^n}^2_{\varepsilon} + \norm[\bigg]{(I - \Pi_h^2) \dfrac{\partial H}{\partial t}(t^{n + \frac{1}{2}})}^2_\mu \\ + \dfrac{\Delta t^4}{(24)^2} \norm[\bigg]{(I - \Pi_h^2) \dfrac{\partial^3 H}{\partial t^3}(t^{n + \frac{1}{2}})}^2_{\mu}  + \norm[\bigg]{(I - \Pi_h^2) R_H^n}^2_\mu  + \norm{R_p^n}^2_{\varepsilon^{-1}} + \norm{R_E^n}^2_{\varepsilon} + \norm{R_H^{n + \frac{1}{2}}}^2_{\mu} + \norm{\nabla r_{1,p}^n} ^2_{\varepsilon^{-1}} \\ + \norm{\nabla \cdot r_{1,E}^n}^2_{\varepsilon} + \varepsilon^{-1} \mu^{-1} \norm{\nabla \times r_{1,E}^n}^2_{\varepsilon} +  \varepsilon^{-1} \mu^{-1} \norm{\nabla \times r_{1,H}^{n + \frac{1}{2}}}^2_{\mu} 
+  \dfrac{\Delta t^4}{144} \left( \norm{\nabla \nabla \cdot \nabla r_{2,p}^n} ^2_{\varepsilon^{-1}} + \norm{\nabla \cdot \nabla \nabla \cdot r_{2,E}^n}^2_{\varepsilon} \right. \\ 
\left. + \varepsilon^{-3} \mu^{-3} \norm{\nabla \times \nabla \times \nabla \times r_{2,E}^n}^2_{\varepsilon} +  \varepsilon^{-3} \mu^{-3} \norm{\nabla \times \nabla \times \nabla \times r_{2,H}^{n + \frac{1}{2}}}^2_{\mu} \right)  \bigg] + \bigg[ \norm{\eta_h^\frac{1}{2}}^2_{\varepsilon^{-1}} + \norm{\zeta_h^\frac{1}{2}}^2_{\varepsilon} + \norm{\xi_h^1}^2_{\mu} \bigg].
\end{multline*}
Likewise, for the semidiscrete approximation of the initial system as in Equations~\labelcref{eqn:maxwell_p0_lf4,eqn:maxwell_E0_lf4,eqn:maxwell_H0_lf4}, we obtain for their errors the following system of equations:
 \begin{multline*}
  \aInnerproduct{\dfrac{e_{p_h}^{\frac{1}{2}} - e_{p_h}^0}{\Delta t/2}}{\widetilde{p}} - \dfrac{1}{4} \aInnerproduct{ \varepsilon \left(e_{E_h}^{\frac{1}{2}} + e_{E_h}^0 \right)}{\nabla \widetilde{p}} + \dfrac{\Delta t^2}{192} \aInnerproduct{\varepsilon \nabla \nabla \cdot \left( e_{E_h}^{\frac{1}{2}} + e_{E_h}^0 \right)}{\nabla \widetilde{p}} \\ = \aInnerproduct{\dfrac{1}{2} R_p^0 -   \dfrac{\varepsilon}{2} \nabla \cdot r_{2,E}^0 - \dfrac{\Delta t^2}{96} \varepsilon \nabla \cdot \nabla \nabla \cdot r_{1,E}^0}{\widetilde{p}}, 
  \end{multline*} 
 \begin{multline*}
  \dfrac{1}{4} \aInnerproduct{\nabla \left(e_p^{\frac{1}{2}} + e_p^0 \right)}{\widetilde{E}} - \dfrac{\Delta t^2}{192} \aInnerproduct{ \nabla \left( e_{p_h}^{\frac{1}{2}} +  e_{p_h}^0 \right)}{\nabla \nabla \cdot \widetilde{E}} + \aInnerproduct{\varepsilon \dfrac{e_{E_h}^{\frac{1}{2}} - e_{E_h}^0}{\Delta t/2}}{\widetilde{E}}  - \dfrac{1}{2} \aInnerproduct{e_{H_h}^1 + e_{H_h}^0}{\nabla \times \widetilde{E}} \\ - \dfrac{\Delta t^2}{96}  \aInnerproduct{\mu^{-1}\varepsilon^{-1}  \nabla \times \nabla \times \left( e_{H_h}^1 + e_{H_h}^0 \right)}{\widetilde{E}} \\ = \aInnerproduct{\dfrac{\varepsilon}{2} R_E^0 + \dfrac{1}{2} \nabla r_{2,p}^0 - \nabla \times r_{2,H}^{\frac{1}{2}} - \dfrac{\Delta t^2}{96} \nabla \nabla \cdot \nabla r_{1,p}^0 - \dfrac{\Delta t^2}{48}  \mu^{-1}\varepsilon^{-1} \nabla \times \nabla \times \nabla \times r_{1,H}^{\frac{1}{2}}}{\widetilde{E}}, 
\end{multline*} 
 \begin{multline*}
 \aInnerproduct{\mu \dfrac{e_{H_h}^1 - e_{H_h}^0}{\Delta t}}{\widetilde{H}} + \dfrac{1}{4} \aInnerproduct{\nabla \times \left( e_{E_h}^{\frac{1}{2}} + e_{E_h}^0 \right)}{\widetilde{H}} + \dfrac{\Delta t^2}{192} \aInnerproduct{ \varepsilon^{-1}\mu^{-1} \nabla \times \left(e_{E_h}^{\frac{1}{2}} + e_{E_h}^0 \right)}{ \nabla \times \nabla \times  \widetilde{H}}\\  = \aInnerproduct{\mu R_H^{\frac{1}{2}}+ \dfrac{1}{2} \nabla \times r_{2,E}^0 + \dfrac{\Delta t^2}{96} \varepsilon^{-1}\mu^{-1} \nabla \times \nabla \times \nabla \times r_{1,E}^0}{\widetilde{H}},
\end{multline*}
Then, following the same sequence of steps for the initial time step and furthermore adding the resulting equation to the inequality for the general case, we arrive at the following:
\begin{multline*}
  \norm{\eta_h^{N - \frac{1}{2}}}^2_{\varepsilon^{-1}} + \norm{\zeta_h^{N - \frac{1}{2}}}^2_{\varepsilon} + \norm{\xi_h^N}^2_{\mu} \le 
  \dfrac{1 + 4 \Delta t}{1 - 4 \Delta t} \Bigg[ \norm{\eta_h^0}^2_{\varepsilon^{-1}} + \norm{\zeta_h^0}^2_{\varepsilon} + \norm{\xi_h^0}^2_{\mu} \Bigg] + \dfrac{8 \Delta t}{1 - 4\Delta t} \sum\limits_{n = 0}^{N - 1} \Bigg[ \norm{\eta_h^{n + \frac{1}{2}}}^2_{\varepsilon^{-1}} \\ + \norm{\zeta_h^{n + \frac{1}{2}}}^2_{\varepsilon} +  \norm{\xi_h^{n + 1}}^2_{\mu} \Bigg] + \dfrac{ \Delta t}{1 - 4 \Delta t} \sum\limits_{n = 0}^{N - 1} \Bigg[ \norm[\bigg]{(I - \Pi_h^0) \dfrac{\partial p}{\partial t}(t^n)}^2_{\varepsilon^{-1}} + \dfrac{\Delta t^4}{(24)^2} \norm[\bigg]{(I - \Pi_h^0) \dfrac{\partial^3 p}{\partial t^3}(t^n)}^2_{\varepsilon^{-1}} \\ + \norm[\bigg]{(I - \Pi_h^0) R_p^n}^2_{\varepsilon^{-1}} + \norm[\bigg]{(I - \Pi_h^1) \dfrac{\partial E}{\partial t}(t^n)}^2_{\varepsilon} + \dfrac{\Delta t^4}{(24)^2} \norm[\bigg]{(I - \Pi_h^1) \dfrac{\partial^3 E}{\partial t^3}(t^n)}^2_{\varepsilon} + \norm[\bigg]{(I - \Pi_h^1) R_E^n}^2_{\varepsilon} \\ + \norm[\bigg]{(I - \Pi_h^2) \dfrac{\partial H}{\partial t}(t^{n + \frac{1}{2}})}^2_\mu + \dfrac{\Delta t^4}{(24)^2} \norm[\bigg]{(I - \Pi_h^2) \dfrac{\partial^3 H}{\partial t^3}(t^{n + \frac{1}{2}})}^2_{\mu} + \norm[\bigg]{(I - \Pi_h^2) R_H^{n + \frac{1}{2}}}^2_{\mu} + \norm{R_p^n}^2_{\varepsilon^{-1}} + \norm{R_E^n}^2_{\varepsilon} \\ + \norm{R_H^{n + \frac{1}{2}}}^2_{\mu} + \norm{\nabla r_{1,p}^n} ^2_{\varepsilon^{-1}} + \norm{\nabla \cdot r_{1,E}^n}^2_{\varepsilon} + \varepsilon^{-1} \mu^{-1} \norm{\nabla \times r_{1,E}^n}^2_{\varepsilon} +  \varepsilon^{-1} \mu^{-1} \norm{\nabla \times r_{1,H}^{n + \frac{1}{2}}}^2_{\mu} + \dfrac{\Delta t^4}{144} \left( \norm{\nabla \nabla \cdot \nabla r_{2,p}^n} ^2_{\varepsilon^{-1}} \right. \\ \left. + \norm{\nabla \cdot \nabla \nabla \cdot r_{2,E}^n}^2_{\varepsilon} + \varepsilon^{-3} \mu^{-3} \norm{\nabla \times \nabla \times \nabla \times r_{2,E}^n}^2_{\varepsilon} +  \varepsilon^{-3} \mu^{-3} \norm{\nabla \times \nabla \times \nabla \times r_{2,H}^{n + \frac{1}{2}}}^2_{\mu} \right)  \Bigg].
\end{multline*}
Applying the discrete Gronwall inequality with $\Delta t < 1/24$, we get that:
\begin{multline*}
  \norm{\eta_h^{N - \frac{1}{2}}}^2_{\varepsilon^{-1}} + \norm{\zeta_h^{N - \frac{1}{2}}}^2_{\varepsilon} + \norm{\xi_h^N}^2_{\mu} \le \Bigg[ \dfrac{6 \Delta t}{5} \sum\limits_{n = 0}^{N - 1} \Bigg( \norm[\bigg]{(I - \Pi_h^0) \dfrac{\partial p}{\partial t}(t^n)}^2_{\varepsilon^{-1}} + \norm[\bigg]{(I - \Pi_h^0) \dfrac{\partial^3 p}{\partial t^3}(t^n)}^2_{\varepsilon^{-1}} \\ + \norm[\bigg]{(I - \Pi_h^0) R_p^n}^2_{\varepsilon^{-1}}
 + \norm[\bigg]{(I - \Pi_h^1) \dfrac{\partial E}{\partial t}(t^n)}^2_{\varepsilon} + \norm[\bigg]{(I - \Pi_h^1) \dfrac{\partial^3 E}{\partial t^3}(t^n)}^2_{\varepsilon} + \norm[\bigg]{(I - \Pi_h^1) R_E^n}^2_{\varepsilon} + \norm[\bigg]{(I - \Pi_h^2) \dfrac{\partial H}{\partial t}(t^{n + \frac{1}{2}})}^2_\mu \\ + \norm[\bigg]{(I - \Pi_h^2) \dfrac{\partial^3 H}{\partial t^3}(t^{n + \frac{1}{2}})}^2_{\mu} + \norm[\bigg]{(I - \Pi_h^2) R_H^{n + \frac{1}{2}}}^2_{\mu} + \norm{R_p^n}^2_{\varepsilon^{-1}} + \norm{R_E^n}^2_{\varepsilon} + \norm{R_H^{n + \frac{1}{2}}}^2_{\mu} + \norm{\nabla r_{1,p}^n} ^2_{\varepsilon^{-1}} + \norm{\nabla \cdot r_{1,E}^n}^2_{\varepsilon} \\ + \varepsilon^{-1} \mu^{-1} \norm{\nabla \times r_{1,E}^n}^2_{\varepsilon} +  \varepsilon^{-1} \mu^{-1} \norm{\nabla \times r_{1,H}^{n + \frac{1}{2}}}^2_{\mu} + \dfrac{\Delta t^4}{144} \left( \norm{\nabla \nabla \cdot \nabla r_{2,p}^n} ^2_{\varepsilon^{-1}} + \norm{\nabla \cdot \nabla \nabla \cdot r_{2,E}^n}^2_{\varepsilon} \right. \\ \left. + \varepsilon^{-3} \mu^{-3} \norm{\nabla \times \nabla \times \nabla \times r_{2,E}^n}^2_{\varepsilon} +  \varepsilon^{-3} \mu^{-3} \norm{\nabla \times \nabla \times \nabla \times r_{2,H}^{n + \frac{1}{2}}}^2_{\mu} \right) \Bigg) \\ + \dfrac{7}{5} \Big( \norm{\eta_h^0}^2_{\varepsilon^{-1}} + \norm{\zeta_h^0}^2_{\varepsilon} + \norm{\xi_h^0}^2_{\mu} \Big) \Bigg] \exp \left( 16 T \right).
\end{multline*}
Using our estimates for the Taylor remainders from Theorem~\labelcref{thm:dscrt_error_estmt_lf4} for these terms on the right hand side of the above inequality, we further get that:
\[
  \Delta t \sum\limits_{n = 0}^N \Big[ \norm{R^n_p}^2_{\varepsilon^{-1}} + \norm{R^n_E}^2_{\varepsilon} + \norm{R^n_H}^2_\mu \Big] \le \Delta t^8 \Bigg[ \norm[\bigg]{\dfrac{\partial^5 p}{\partial t^5}}^2_{L^2(0, T; L^2_{\varepsilon^{-1}}(\Omega))} \!\! + \, \norm[\bigg]{\dfrac{\partial^5 E}{\partial t^5}}^2_{L^2(0, T; L^2_\varepsilon(\Omega))} \!\! + \, \norm[\bigg]{\dfrac{\partial^5 H}{\partial t^5}}^2_{L^2(0, T; L^2_\mu(\Omega))} \Bigg],
\]
and that:
\begin{multline*}
  \Delta t \sum\limits_{n = 0}^N \Big[\norm{\nabla r_{2,p}^n} ^2_{\varepsilon^{-1}} + \norm{\nabla \cdot r_{2,E}^n}^2_{\varepsilon} + \varepsilon^{-1} \mu^{-1} \norm{\nabla \times r_{2,E}^n}^2_{\varepsilon} +  \varepsilon^{-1} \mu^{-1} \norm{\nabla \times r_{2,H}^{n + \frac{1}{2}}}^2_{\mu} + \dfrac{\Delta t^4}{144} \left( \norm{\nabla \nabla \cdot \nabla r_{1,p}^n} ^2_{\varepsilon^{-1}} \right. \\ \left. + \norm{\nabla \cdot \nabla \nabla \cdot r_{1,E}^n}^2_{\varepsilon} + \varepsilon^{-3} \mu^{-3} \norm{\nabla \times \nabla \times \nabla \times r_{1,E}^n}^2_{\varepsilon} +  \varepsilon^{-3} \mu^{-3} \norm{\nabla \times \nabla \times \nabla \times r_{1,H}^{n + \frac{1}{2}}}^2_{\mu} \right) \Big] \\ \le \Delta t^8 \Bigg[ \norm[\bigg]{\dfrac{\partial^4 \left(\nabla p \right)}{\partial t^4}}^2_{L^2(0, T; L^2_{\varepsilon^{-1}}(\Omega))} \!\! + 
  \norm[\bigg]{\dfrac{\partial^4 (\nabla \cdot E)}{\partial t^4}}^2_{L^2(0, T; L^2_\varepsilon(\Omega))} \!\! + \varepsilon^{-1} \mu^{-1} \norm[\bigg]{\dfrac{\partial^4 (\nabla \times E)}{\partial t^4}}^2_{L^2(0, T; L^2_{\varepsilon}(\Omega))} \!\! \\ + \varepsilon^{-1} \mu^{-1} \norm[\bigg]{\dfrac{\partial^4 (\nabla \times H)}{\partial t^4}}^2_{L^2(0, T; L^2_\mu(\Omega))} + \norm[\bigg]{\dfrac{\partial^2 \left(\nabla \nabla \cdot \nabla p \right)}{\partial t^2}}^2_{L^2(0, T; L^2_{\varepsilon^{-1}}(\Omega))} \!\! + 
  \norm[\bigg]{\dfrac{\partial^2 (\nabla \cdot \nabla \nabla \cdot E)}{\partial t^2}}^2_{L^2(0, T; L^2_\varepsilon(\Omega))} \!\! \\ + \varepsilon^{-3} \mu^{-3} \norm[\bigg]{\dfrac{\partial^2 (\nabla \times \nabla \times \nabla \times E)}{\partial t^2}}^2_{L^2(0, T; L^2_{\varepsilon}(\Omega))} \!\! + \varepsilon^{-3} \mu^{-3} \norm[\bigg]{\dfrac{\partial^2 (\nabla \times \nabla \times \nabla \times H)}{\partial t^2}}^2_{L^2(0, T; L^2_\mu(\Omega))} \Bigg].
\end{multline*}
Now, for $p \in \mathring{H}_{\varepsilon^{-1}}^1(\Omega)$, $E \in \mathring{H}_\varepsilon(\curl; \Omega)$ and $H \in \mathring{H}_\mu(\divgn; \Omega)$, there exists positive bounded constants $C_{1, p}$, $C_{2, p}$, $C_{3, p}$, $C_{1, E}$, $C_{2, E}$, $C_{3, E}$, $C_{1, H}$, $C_{2, H}$, and $C_{3, H}$ such that we have the following error bounds for the $L^2$ projections:
\begin{gather*}
\begin{alignat*}{3}
  \norm[\bigg]{(I - \Pi_h^0) \dfrac{\partial p}{\partial t}(t^n)}_{\varepsilon^{-1}} &\le C_{1, p} h^r \norm[\bigg]{\dfrac{\partial p}{\partial t}(t^n)}_{\varepsilon^{-1}}, && \qquad \norm[\bigg]{(I - \Pi_h^0) \dfrac{\partial^3 p}{\partial t^3}(t^n)}_{\varepsilon^{-1}} &&\le C_{2, p} h^r \norm[\bigg]{\dfrac{\partial^3 p}{\partial t^3}(t^n)}_{\varepsilon^{-1}},  \\
  \norm[\bigg]{(I - \Pi_h^1) \dfrac{\partial E}{\partial t}(t^n)}_{\varepsilon} & \le C_{1, E} h^r \norm[\bigg]{\dfrac{\partial E}{\partial t}(t^n)}_{\varepsilon}, && \qquad \norm[\bigg]{(I - \Pi_h^1) \dfrac{\partial^3 E}{\partial t^3}(t^n)}_{\varepsilon} &&\le C_{2, E} h^r \norm[\bigg]{\dfrac{\partial^3 E}{\partial t^3}(t^n)}_{\varepsilon} \\
  \norm[\bigg]{(I - \Pi_h^2) \dfrac{\partial H}{\partial t}(t^{n + \frac{1}{2}})}_{\mu} &\le C_{1, H} h^r \norm[\bigg]{\dfrac{\partial H}{\partial t}(t^{n + \frac{1}{2}})}_{\mu}, && \qquad \norm[\bigg]{(I - \Pi_h^2) \dfrac{\partial^3 H}{\partial t^3}(t^{n + \frac{1}{2}})}_{\mu} && \le C_{2, H} h^r \norm[\bigg]{\dfrac{\partial^3 H}{\partial t^3}(t^{n + \frac{1}{2}})}_{\mu},
\end{alignat*} \\
 \norm[\bigg]{(I - \Pi_h^0) R_p^n}_{\varepsilon^{-1}} \le C_{3, p} h^r \norm[\bigg]{R_p^n}_{\varepsilon^{-1}}, \norm[\bigg]{(I - \Pi_h^1) R_E^n}_{\varepsilon} \le C_{3, E} h^r \norm[\bigg]{R_E^n}_{\varepsilon}, \norm[\bigg]{(I - \Pi_h^2) R_H^{n + \frac{1}{2}}}_{\mu} \le C_{3, H} h^r \norm[\bigg]{R_H^{n + \frac{1}{2}}}_{\mu}.
\end{gather*}
Set $C_0 \coloneq \max \{C_{1, p}, C_{2, p}, C_{3, p}, C_{1, E}, C_{2, E}, C_{3, E}, C_{1, H}, C_{2, H}, C_{3, H} \}$. We therefore have that:
 \begin{multline*}
    \Delta t \sum\limits_{n = 0}^{N-1} \left[ \norm[\bigg]{(I - \Pi_h^0) \dfrac{\partial p}{\partial t}(t^n)}^2_{\varepsilon^{-1}} \!\! + \norm[\bigg]{(I - \Pi_h^1) \dfrac{\partial E}{\partial t}(t^n)}^2_{\varepsilon} \!\! + \norm[\bigg]{(I - \Pi_h^2) \dfrac{\partial H}{\partial t}(t^{n + \frac{1}{2}})}^2_\mu \right] \\
\le C_0 h^{2 r} \sum\limits_{n = 0}^{N-1} \Delta t \left[ \norm[\bigg]{\dfrac{\partial p}{\partial t}(t^{n - \frac{1}{2}})}^2_{\varepsilon^{-1}} \!\! + \norm[\bigg]{\dfrac{\partial E}{\partial t}(t^{n - \frac{1}{2}})}^2_{\varepsilon} \!\! + \norm[\bigg]{\dfrac{\partial H}{\partial t}(t^{n - \frac{1}{2}})}^2_{\mu} \right] \\
\le C_0 h^{2 r} \int\limits_0^T \left[ \norm[\bigg]{\dfrac{\partial p}{\partial t}(t^{n - \frac{1}{2}})}^2_{\varepsilon^{-1}} \!\! + \norm[\bigg]{\dfrac{\partial E}{\partial t}(t^{n - \frac{1}{2}})}^2_{\varepsilon} \!\! + \norm[\bigg]{\dfrac{\partial H}{\partial t}(t^{n - \frac{1}{2}})}^2_{\mu} \right] dt \\
= C_0 h^{2 r} \left[ \norm[\bigg]{\dfrac{\partial p}{\partial t}}^2_{L^2(0, T; L^2_{\varepsilon^{-1}}(\Omega))} \!\! + \norm[\bigg]{\dfrac{\partial E}{\partial t}}^2_{L^2(0, T; L^2_{\varepsilon}(\Omega))} \!\! + \norm[\bigg]{\dfrac{\partial H}{\partial t}}^2_{L^2(0, T; L^2_{\mu}(\Omega))} \right],
\end{multline*}
 \begin{multline*}
    \Delta t \sum\limits_{n = 0}^{N-1} \left[ \norm[\bigg]{(I - \Pi_h^0) \dfrac{\partial^3 p}{\partial t^3}(t^n)}^2_{\varepsilon^{-1}} \!\! + \norm[\bigg]{(I - \Pi_h^1) \dfrac{\partial^3 E}{\partial t^3}(t^n)}^2_{\varepsilon} \!\! + \norm[\bigg]{(I - \Pi_h^2) \dfrac{\partial^3 H}{\partial t^3}(t^{n + \frac{1}{2}})}^2_\mu \right] \\
\le C_0 h^{2 r} \sum\limits_{n = 0}^{N-1} \Delta t \left[ \norm[\bigg]{\dfrac{\partial^3 p}{\partial t^3}(t^n)}^2_{\varepsilon^{-1}} \!\! + \norm[\bigg]{\dfrac{\partial^3 E}{\partial t^3}(t^n)}^2_{\varepsilon} \!\! + \norm[\bigg]{\dfrac{\partial^3 H}{\partial t^3}(t^{n + \frac{1}{2}})}^2_{\mu} \right] \\
\le C_0 h^{2 r} \int\limits_0^T \left[ \norm[\bigg]{\dfrac{\partial^3 p}{\partial t^3}(t^n)}^2_{\varepsilon^{-1}} \!\! + \norm[\bigg]{\dfrac{\partial^3 E}{\partial t^3}(t^n)}^2_{\varepsilon} \!\! + \norm[\bigg]{\dfrac{\partial^3 H}{\partial t^3}(t^{n + \frac{1}{2}})}^2_{\mu} \right] dt \\
= C_0 h^{2 r} \left[ \norm[\bigg]{\dfrac{\partial^3 p}{\partial t^3}}^2_{L^2(0, T; L^2_{\varepsilon^{-1}}(\Omega))} \!\! + \norm[\bigg]{\dfrac{\partial^3 E}{\partial t^3}}^2_{L^2(0, T; L^2_{\varepsilon}(\Omega))} \!\! + \norm[\bigg]{\dfrac{\partial^3 H}{\partial t^3}}^2_{L^2(0, T; L^2_{\mu}(\Omega))} \right],
\end{multline*}
and likewise for the $L^2$ projection of the remainder terms:
\begin{multline*}
  \Delta t \sum\limits_{n = 0}^N \left[ \norm[\bigg]{(I - \Pi_h^0) R_p^n}_{\varepsilon^{-1}} \!\! + \norm[\bigg]{(I - \Pi_h^1) R_E^n}_{\varepsilon} \!\! + \norm[\bigg]{(I - \Pi_h^2) R_H^{n + \frac{1}{2}}}_{\mu} \right], \\
  \le C_0 h^{2 r} \sum\limits_{n = 0}^N  \Delta t \left[ \norm{R_p^n}_{L^2_{\varepsilon^{-1}}(\Omega)} + \norm{R_E^n}_{L^2_{\varepsilon}(\Omega)} + \norm{R_H^{n + \frac{1}{2}}}_{L^2_{\mu}(\Omega)} \right], \\
  \le C_0 h^{2 r} \Delta t^8 \left[ \norm[\bigg]{\dfrac{\partial^5 p}{\partial t^5}}^2_{L^2(0, T; L^2_{\varepsilon^{-1}}(\Omega))} \!\! + \norm[\bigg]{\dfrac{\partial^5 E}{\partial t^5}}^2_{L^2(0, T; L^2_{\varepsilon}(\Omega))} \!\! + \norm[\bigg]{\dfrac{\partial^5 H}{\partial t^5}}^2_{L^2(0, T; L^2_{\mu}(\Omega))} \right].
\end{multline*}
Now, we take $p_h^0 = \Pi_h^0 p_0$, $E_h^0 = \Pi_h^1 E_0$, and $H_h^0 = \Pi_h^2 H_0$, and note that there exists positive bounded constants $M_1, \, M_2, \, M_3$ and $M_4$ such that:
\begin{align*}
  \norm[\bigg]{\dfrac{\partial p}{\partial t}}^2_{L^2(0, T; L^2_{\varepsilon^{-1}}(\Omega))} \!\! + \norm[\bigg]{\dfrac{\partial E}{\partial t}}^2_{L^2(0, T; L^2_{\varepsilon}(\Omega))} \!\! + \norm[\bigg]{\dfrac{\partial H}{\partial t}}^2_{L^2(0, T; L^2_{\mu}(\Omega))} &\le M_1, \\
  \norm[\bigg]{\dfrac{\partial^3 p}{\partial t^3}}^2_{L^2(0, T; L^2_{\varepsilon^{-1}}(\Omega))} \!\! + \norm[\bigg]{\dfrac{\partial^3 E}{\partial t^3}}^2_{L^2(0, T; L^2_{\varepsilon}(\Omega))} \!\! + \norm[\bigg]{\dfrac{\partial^3 H}{\partial t^3}}^2_{L^2(0, T; L^2_{\mu}(\Omega))} &\le M_2, \\
  \norm[\bigg]{\dfrac{\partial^5 p}{\partial t^5}}^2_{L^2(0, T; L^2_{\varepsilon^{-1}}(\Omega))} \!\! + \norm[\bigg]{\dfrac{\partial^5 E}{\partial t^5}}^2_{L^2(0, T; L^2_{\varepsilon}(\Omega))} \!\! + \norm[\bigg]{\dfrac{\partial^5 H}{\partial t^5}}^2_{L^2(0, T; L^2_{\mu}(\Omega))} &\le M_3,
\end{align*}
\vspace{-1em} \begin{multline*}
 \norm[\bigg]{\dfrac{\partial^4 \left(\nabla p \right)}{\partial t^4}}^2_{L^2(0, T; L^2_{\varepsilon^{-1}}(\Omega))} \!\! + \norm[\bigg]{\dfrac{\partial^4 (\nabla \cdot E)}{\partial t^4}}^2_{L^2(0, T; L^2_\varepsilon(\Omega))} \!\! + \varepsilon^{-1} \mu^{-1} \norm[\bigg]{\dfrac{\partial^4 (\nabla \times E)}{\partial t^4}}^2_{L^2(0, T; L^2_{\varepsilon}(\Omega))} \!\! \\ + \varepsilon^{-1} \mu^{-1} \norm[\bigg]{\dfrac{\partial^4 (\nabla \times H)}{\partial t^4}}^2_{L^2(0, T; L^2_\mu(\Omega))} + \norm[\bigg]{\dfrac{\partial^2 \left(\nabla \nabla \cdot \nabla p \right)}{\partial t^2}}^2_{L^2(0, T; L^2_{\varepsilon^{-1}}(\Omega))} \!\! + 
  \norm[\bigg]{\dfrac{\partial^2 (\nabla \cdot \nabla \nabla \cdot E)}{\partial t^2}}^2_{L^2(0, T; L^2_\varepsilon(\Omega))} \!\! \\ + \varepsilon^{-1} \mu^{-1} \norm[\bigg]{\dfrac{\partial^2 (\nabla \times \nabla \times \nabla \times E)}{\partial t^2}}^2_{L^2(0, T; L^2_{\varepsilon}(\Omega))} \!\! + \varepsilon^{-1} \mu^{-1} \norm[\bigg]{\dfrac{\partial^2 (\nabla \times \nabla \times \nabla \times H)}{\partial t^2}}^2_{L^2(0, T; L^2_\mu(\Omega))} \leq M_4.
\end{multline*}
Consequently, we have the following estimate:
\[
  \norm{\eta_h^{N-\frac{1}{2}}}^2_{\varepsilon^{-1}} + \norm{\zeta_h^{N-\frac{1}{2}}}^2_{\varepsilon} + \norm{\xi_h^N}^2_{\mu} \le \widetilde{C} \left[h^{2 r} + h^{2 r} \Delta t^8 + \Delta t^8 \right],
\]
where $\widetilde{C} = 2 \max\{C_0 M_1, C_0 M_2, C_0 M_3, M_3, M_4 \} \exp(16 T)$ which then using the equivalence between $1$- and $2$-norms gives us that:
\[
  \norm{\eta_h^{N-\frac{1}{2}}}_{\varepsilon^{-1}} + \norm{\zeta_h^{N-\frac{1}{2}}}_{\varepsilon} + \norm{\xi_h^N}_{\mu} \le C_1 \left[h^{r} + h^r \Delta t^4 + \Delta t^4 \right],
\]
where $C_1 = 2 \sqrt{\widetilde{C}}$. Also, there are positive bounded constants $\widetilde{C_2}$, $\widetilde{C_3}$ and $\widetilde{C_4}$ such that:
\begin{alignat*}{4}
  \norm{\eta^{N-\frac{1}{2}}}_{\varepsilon^{-1}} &= \norm{(I - \Pi_h^0) p(t^{N-\frac{1}{2}})}_{\varepsilon^{-1}} &&\le \widetilde{C_2} h^r \norm{p(t^{N-\frac{1}{2}})}_{L^2_{\varepsilon^{-1}}(\Omega)} &&\implies \norm{\eta^{N-\frac{1}{2}}}_{\varepsilon^{-1}} &&\le C_2 h^r, \\
  \norm{\zeta^{N-\frac{1}{2}}}_{\varepsilon} &= \norm{(I - \Pi_h^1) E(t^{N-\frac{1}{2}})}_{\varepsilon} &&\le \widetilde{C_3} h^r \norm{E(t^{N-\frac{1}{2}})}_{L^2_{\varepsilon}(\Omega)} &&\implies \norm{\zeta^{N-\frac{1}{2}}}_{\varepsilon} &&\le C_3 h^r, \\
  \norm{\xi^N}_\mu &= \norm{(I - \Pi_h^2) H(t^N)}_\mu &&\le \widetilde{C_4} h^r \norm{H(t^N)}_{L^2_{\mu}(\Omega)} &&\implies \norm{\xi^N}_\mu &&\le C_4 h^r,
\end{alignat*}
in which $C_2 = \widetilde{C_2} \norm{p(t^{N-\frac{1}{2}})}_{L^2_{\varepsilon^{-1}}(\Omega)}$, $C_3 = \widetilde{C_3} \norm{E(t^{N-\frac{1}{2}})}_{L^2_{\varepsilon}(\Omega)}$ and $C_4 = \widetilde{C_4} \norm{H(t^N)}_{L^2_{\mu}(\Omega)}$ are all bounded positive constants due to Theorem~\ref{thm:dscrt_enrgy_estmt_lf4}. Finally, this provides us with our desired result by choosing $C = C_1 + C_2 + C_3 + C_4$:
\begin{align*}
  \norm{e_{p_h}^{N - \frac{1}{2}}}_{\varepsilon^{-1}} + \norm{e_{E_h}^{N - \frac{1}{2}}}_{\varepsilon} + \norm{e_{H_h}^{N}}_{\mu} &= \norm{\eta^{N - \frac{1}{2}} - \eta^{N - \frac{1}{2}}_h}_{\varepsilon^{-1}} + \norm{\zeta^{N - \frac{1}{2}} - \zeta^{N - \frac{1}{2}}_h}_{\varepsilon} + \norm{\xi^N - \xi^N_h}_\mu \\
  &\le C \left[ (\Delta t )^4 + h^r + h^r \Delta t^4 \right]. \qedhere
\end{align*}
\end{proof}

\section{Characterization of TS$_4$ Scheme} \label{sec:ts4}

\subsection{Time Discretization Stability}
For our implicit TS$_4$ scheme as in Equations~\labelcref{eqn:maxwell_p_ts4,eqn:maxwell_E_ts4,eqn:maxwell_H_ts4}, just like for LF$_4$ in Section~\labelcref{sec:implicit_lf4}, we can show that it is stable and converges with the correct fourth order of approximation in time. The analysis here is merely an appropriate update of the theorems and their proofs in Section~\labelcref{sec:implicit_lf4}. However, for the TS$_4$ scheme, we define the discrete energy at time $t^n$ to be the following and it is motivated by its desired conservation:
\begin{equation}
\mathcal{E}^n \coloneq \norm{p^{n+1} +p^n}^2_{\varepsilon^{-1}} + 2 \aInnerproduct{p^{n+1}}{p^n}_{\varepsilon^{-1}} + \norm{E^{n+1} +E^n}^2_\varepsilon + 2 \aInnerproduct{E^{n+1}}{E^n}_\varepsilon + \norm{H^{n+1} +H^n}^2_\mu + 2 \aInnerproduct{H^{n+1}}{H^n}_\mu.
\end{equation}
In addition, we require as assumption on the inner products which we state as a hypothesis next and we are motivated by our computational experiments to assert this.

\begin{assumption}[Positivity of Inner products] \label{assume:innerproducts}
The various $\aInnerproduct{\cdot}{\cdot}_\alpha$ for $\alpha = \varepsilon^{-1}, \varepsilon$ and $\mu$ are all positive for all $t \in [0, T]$ for the Maxwell's solution.
\end{assumption}

\begin{theorem}[Discrete Energy Conservation]\label{thm:dscrt_enrgy_cnsrvtn_ts4}
  For the semidiscretization using the TS$_4$ scheme as given in Equations~\labelcref{eqn:maxwell_p_ts4,eqn:maxwell_E_ts4,eqn:maxwell_H_ts4} and with Assumption~\labelcref{assume:innerproducts}, for any fixed time step $\Delta t > 0$ sufficiently small, the energy $\mathcal{E}^n$ is conserved, that is, $\mathcal{E}^n = \mathcal{E}^{n-1}$ for $n = 1, \dots, N-1$.
\end{theorem}

\begin{proof}
  Since Equations~\labelcref{eqn:maxwell_p_ts4,eqn:maxwell_E_ts4,eqn:maxwell_H_ts4} are true for all $\widetilde{p}\in \mathring{H}^1_{\varepsilon^{-1}}(\Omega)$, $\widetilde{E} \in \mathring{H}_{\varepsilon}(\curl; \Omega)$, $\widetilde{H} \in \mathring{H}_{\mu}(\divgn; \Omega)$, using $\widetilde{p} = 2 \Delta t \left(p^{n + 1} + 4 p^n + p^{n - 1} \right)$, $\widetilde{E} = 2 \Delta t \left(E^{n + 1} + 4 E^n + E^{n - 1} \right)$ and $\widetilde{H} = 2 \Delta t \left(H^{n + 1} + 4 H^n + H^{n - 1} \right)$ in them and adding the resultant equations using properties of the inner product, we obtain:
\begin{multline*}
  \aInnerproduct{\varepsilon^{-1} \left( p^{n + 1} - p^{n - 1} \right)}{p^{n + 1} + 4 p^n + p^{n - 1}} + \aInnerproduct{\varepsilon \left(E^{n + 1} - E^{n -1} \right)}{E^{n +1} + 4 E^n + E^{n - 1}} \\ + \aInnerproduct{\mu \left( H^{n + 1} - H^{n - 1} \right)}{H^{n + 1} + 4 H^n + H^{n -1}} = 0.
\end{multline*}
Consider next that:
\begin{align*}
\aInnerproduct{\varepsilon^{-1} \left( p^{n + 1} - p^{n - 1} \right)}{p^{n + 1} + 4 p^n + p^{n - 1}} & = \aInnerproduct{\varepsilon^{-1} \left( p^{n + 1} + p^n - p^n - p^{n - 1} \right)}{p^{n + 1} + p^n + p^n + p^{n - 1}} \\ & \quad  + 2\aInnerproduct{\varepsilon^{-1} \left( p^{n + 1} - p^{n - 1} \right)}{p^n} \\
& = \aInnerproduct{\varepsilon^{-1} \left( p^{n + 1} + p^n \right)}{p^{n + 1} + p^n} - \aInnerproduct{\varepsilon^{-1} \left(p^n + p^{n - 1} \right)}{p^n + p^{n - 1}}  \\ & \quad + 2\aInnerproduct{\varepsilon^{-1} p^{n + 1}}{p^n} - 2\aInnerproduct{\varepsilon^{-1} p^n}{p^{n - 1}} \\
& = \norm{p^{n + 1} + p^n}_\varepsilon^{-1}  + 2\aInnerproduct{p^{n + 1}}{p^n}_\varepsilon^{-1} - \norm{p^n + p^{n - 1}}_\varepsilon^{-1} - 2\aInnerproduct{p^n}{p^{n - 1}}_\varepsilon^{-1}.
\end{align*}
So, we have the following:
\[
\aInnerproduct{\varepsilon^{-1} \left( p^{n + 1} - p^{n - 1} \right)}{p^{n + 1} + 4 p^n + p^{n - 1}} = \norm{p^{n + 1} + p^n}_\varepsilon^{-1}  + 2\aInnerproduct{p^{n + 1}}{p^n}_\varepsilon^{-1} - \norm{p^n + p^{n - 1}}_\varepsilon^{-1} - 2\aInnerproduct{p^n}{p^{n - 1}}_\varepsilon^{-1}.
\]
We can similarly have the corresponding expressions for $E$ and $H$. Using these various expressions in the first equation in our proof, we get that:
\begin{multline*}
\norm{p^{n + 1} + p^n}_{\varepsilon^{-1}}  + 2\aInnerproduct{p^{n + 1}}{p^n}_{\varepsilon^{-1}} - \norm{p^n + p^{n - 1}}_{\varepsilon^{-1}} - 2\aInnerproduct{p^n}{p^{n - 1}}_{\varepsilon^{-1}} \\ + \norm{E^{n + 1} + E^n}_\varepsilon + 2\aInnerproduct{E^{n + 1}}{E^n}_\varepsilon - \norm{E^n + E^{n - 1}}_\varepsilon - 2\aInnerproduct{E^n}{E^{n - 1}}_\varepsilon \\ + \norm{H^{n + 1} + H^n}_\mu  + 2\aInnerproduct{H^{n + 1}}{H^n}_\mu - \norm{H^n + H^{n - 1}}_\mu - 2\aInnerproduct{H^n}{H^{n - 1}}_\mu = 0,
\end{multline*}
which yields the desired result:
\begin{multline*}
 \norm{p^{n + 1} + p^n}_{\varepsilon^{-1}}  + 2\aInnerproduct{p^{n + 1}}{p^n}_{\varepsilon^{-1}} + \norm{E^{n + 1} + E^n}_\varepsilon +  2\aInnerproduct{E^{n + 1}}{E^n}_\varepsilon + \\
 \norm{H^{n + 1} + H^n}_\mu  + 2\aInnerproduct{H^{n + 1}}{H^n}_\mu = \norm{p^n + p^{n - 1}}_{\varepsilon^{-1}} + 2\aInnerproduct{p^n}{p^{n - 1}}_{\varepsilon^{-1}} + \\
 \norm{E^n + E^{n - 1}}_\varepsilon + 2\aInnerproduct{E^n}{E^{n - 1}}_\varepsilon  + \norm{H^n + H^{n - 1}}_\mu + 2\aInnerproduct{H^n}{H^{n - 1}}_\mu. \qedhere
\end{multline*}
\end{proof}

\begin{theorem}[Discrete Error Estimate]\label{thm:dscrt_error_estmt_ts4}
For the semidiscretization TS$_4$ as given in Equations~\labelcref{eqn:maxwell_p_ts4,eqn:maxwell_E_ts4,eqn:maxwell_H_ts4}, for the solution $(p, E, H)$ of the variationally posed Maxwell's system as in Equations~\labelcref{eqn:maxwell_p_wf,eqn:maxwell_E_wf,eqn:maxwell_H_wf} with initial conditions as in Equation~\eqref{eqn:ICs}, with Assumption~\labelcref{assume:innerproducts} and sufficient regularity $p \in C^5(0, T; \mathring{H}^1_{\varepsilon^{-1}}(\Omega))$, $E \in C^5(0, T; \mathring{H}_{\varepsilon}(\curl; \Omega))$, and $H \in C^5(0, T; \mathring{H}_{\mu}(\divgn; \Omega))$, and for the time step $\Delta t > 0$ sufficiently small, there exists a positive bounded constant $C$ independent of $\Delta t$ such that:
\[
  \norm{e_p^{N - 1}}_{\varepsilon^{-1}} + \norm{e_E^{N - 1}}_{\varepsilon} + \norm{e_H^{N - 1}}_{\mu} \le C \left[ \left(\Delta t\right)^4 + \norm{e_p^0}_{\varepsilon^{-1}} + \norm{e_E^0}_{\varepsilon} + \norm{e_H^0}_{\mu} + \norm{e_p^1}_{\varepsilon^{-1}} + \norm{e_E^1}_{\varepsilon} + \norm{e_H^1}_{\mu} \right],
\]
where $e_p^n \coloneq p(t^n) - p^{n}$, $e_E^n \coloneq E(t^n) - E^n$ and $e_H^n \coloneq H(t^n) - H^n$ are the errors in the time semidiscretization of $p$, $E$ and $H$, respectively.
\end{theorem}

\begin{proof}
Using the Taylor remainder theorem, and expressing $u(t)$ about $t = t^n$ for $u = \{p, E, H\}$ and $n = 1, \dots, N-1$, we have that:
\[
  u(t) = u(t^n) + \dfrac{\partial u}{\partial t}(t^n)(t - t^n) + \dfrac{\partial^2 u}{\partial t^2}(t^n) \dfrac{(t - t^n)^2}{2!} + \dfrac{\partial^3 u}{\partial t^3}(t^n) \dfrac{(t - t^n)^3}{3!} + \dfrac{\partial^4 u}{\partial t^4}(t^n) \dfrac{(t - t^n)^4}{4!} + \int\limits_{t^n}^{t} \dfrac{(t - s)^4}{4!} \dfrac{\partial^5 u}{\partial t^5}(s) ds,
\]
which when evaluated at $t = t^{n + 1}$ and $t = t^{n - 1}$ yields:
\begin{multline*}
  u(t^{n +1}) = u(t^n) + \dfrac{\partial u}{\partial t}(t^n) (t^{n + 1} - t^n) + \dfrac{\partial^2 u}{\partial t^2}(t^n) \dfrac{(t^{n +1} - t^n)^2}{2} + \dfrac{\partial^3 u}{\partial t^3}(t^n) \dfrac{(t^{n + 1} - t^n)^3}{6} \\ + \dfrac{\partial^4 u}{\partial t^4}(t^n) \dfrac{(t^{n + 1} - t^n)^4}{24} + \int\limits_{t^n}^{\mathclap{t^{n + 1}}} \dfrac{(t^{n +1} - s)^4}{24} \dfrac{\partial^5 u}{\partial t^5}(s) ds, 
 \end{multline*}
\begin{multline*}
u(t^{n - 1}) = u(t^n) + \dfrac{\partial u}{\partial t}(t^n)(t^{n - 1} - t^n) + \dfrac{\partial^2 u}{\partial t^2}(t^n) \dfrac{(t^{n - 1} - t^n)^2}{2} + \dfrac{\partial^3 u}{\partial t^3}(t^n) \dfrac{(t^{n - 1} - t^n)^3}{6} \\ + \dfrac{\partial^4 u}{\partial t^4}(t^n) \dfrac{(t^{n - 1} - t^n)^4}{24} + \int\limits_{t^n}^{\mathclap{t^{n - 1}}} \dfrac{(t^{n - 1} - s)^4}{24} \dfrac{\partial^5 u}{\partial t^5}(s) ds.
\end{multline*}
Subtracting the previous two equations and using the result in the variational formulation as in Equations~\labelcref{eqn:maxwell_p_wf,eqn:maxwell_E_wf,eqn:maxwell_H_wf} leads to the following:
\[
  \ainnerproduct{\dfrac{u(t^{n + 1}) - u(t^{n - 1})}{2 \Delta t}}{\widetilde{u}} = \ainnerproduct{\dfrac{\partial u}{\partial t}(t^n)}{\widetilde{u}} + \dfrac{\Delta t^2}{6} \ainnerproduct{\dfrac{\partial^3 u}{\partial t^3}(t^n)}{\widetilde{u}} + \ainnerproduct{R^n_u}{\widetilde{u}},
\]
in which we have defined that:
\[
  R^n_u \coloneq 
  \begin{cases}
  0, & n = 0,\\
  \dfrac{1}{2 \Delta t} \left[\displaystyle\int\limits_{t^{n - 1}}^{t^n} \dfrac{(t^{n - 1} - s)^4}{24} \dfrac{\partial^5 u}{\partial t^5}(s) ds + \int\limits_{t^n}^{t^{n + 1}} \dfrac{(t^{n + 1} - s)^4}{24} \dfrac{\partial^5 u}{\partial t^5}(s) ds \right], & n = 1, \dots, N-1,
  \end{cases}
\]
for $u = \{p, E, H\}$. Again, using the Taylor remainder theorem, and expressing $u(t)$ about $t = t^n$ for $u = \{p, E, H\}$ with $n = 0, 1, \dots, N-1$, ad evaluating this at $t = t^{n + 1}$ and $t = t^{n - 1}$ yields:
\[
  u(t^{n +1}) = u(t^n) + \dfrac{\partial u}{\partial t}(t^n) (t^{n + 1} - t^n) + \dfrac{\partial^2 u}{\partial t^2}(t^n) \dfrac{(t^{n +1} - t^n)^2}{2} + \dfrac{\partial^3 u}{\partial t^3}(t^n) \dfrac{(t^{n + 1} - t^n)^3}{6} + \int\limits_{t^n}^{\mathclap{t^{n + 1}}} \dfrac{(t^{n +1} - s)^3}{6} \dfrac{\partial^4 u}{\partial t^4}(s) ds, 
 \]
\[
u(t^{n - 1}) = u(t^n) + \dfrac{\partial u}{\partial t}(t^n)(t^{n - 1} - t^n) + \dfrac{\partial^2 u}{\partial t^2}(t^n) \dfrac{(t^{n - 1} - t^n)^2}{2} + \dfrac{\partial^3 u}{\partial t^3}(t^n) \dfrac{(t^{n - 1} - t^n)^3}{6} + \int\limits_{t^n}^{\mathclap{t^{n - 1}}} \dfrac{(t^{n - 1} - s)^3}{6} \dfrac{\partial^4 u}{\partial t^4}(s) ds.
\]
Adding these two last equations, and using the result again in the variational formulation as in Equations~\labelcref{eqn:maxwell_p_wf,eqn:maxwell_E_wf,eqn:maxwell_H_wf} leads to:
\[
  \ainnerproduct{\dfrac{u(t^{n + 1}) - 2 u(t^n) + u(t^{n - 1})}{\Delta t^2}}{\widetilde{u}} = \ainnerproduct{\dfrac{\partial^2 u}{\partial t^2}(t^n)}{\widetilde{u}} + \ainnerproduct{r^n_u}{\widetilde{u}},
\]
in which we have again defined that:
\[
  r^n_u \coloneq 
  \begin{cases}
  0, & n = 0,\\
  \dfrac{1}{\Delta t^2} \left[-\displaystyle\int\limits_{t^{n - 1}}^{t^n} \dfrac{(t^{n - 1} - s)^3}{6} \dfrac{\partial^4 u}{\partial t^4}(s) ds + \int\limits_{t^n}^{t^{n + 1}} \dfrac{(t^{n + 1} - s)^3}{6} \dfrac{\partial^4 u}{\partial t^4}(s) ds \right], & n = 1, \dots, N-1,
  \end{cases}
\]
for $u = p, E$ or $H$. Using these terms in the weak formulation as in Equations~\labelcref{eqn:maxwell_p_wf,eqn:maxwell_E_wf,eqn:maxwell_H_wf} and evaluating at $t = t^n$ for $p$, $E$, and $H$ with $n = 1, \dots, N-1$, we are led to:
    \begin{subequations}
    \begin{equation}
   \aInnerproduct{\varepsilon^{-1} \dfrac{p(t^{n + 1}) - p(t^{n - 1})}{2 \Delta t}}{\widetilde{p}} - \dfrac{1}{6} \aInnerproduct{E(t^{n + 1}) +4 E(t^n)  + E(t^{n - 1})}{\nabla \widetilde{p}} = 
   \aInnerproduct{\varepsilon^{-1} R_p^n + \dfrac{\Delta t^2}{6} \nabla \cdot r_E^n}{\widetilde{p}},  \label{eqn:remainder_p_ts4}
   \end{equation}
\begin{equation}
\begin{split}
   \dfrac{1}{6} \aInnerproduct{\nabla \left(p(t^{n + 1}) + 4 p(t^n)  + p(t^{n - 1}) \right)}{\widetilde{E}}  + \aInnerproduct{\varepsilon \dfrac{E(t^{n + 1}) - E(t^{n - 1})}{2 \Delta t}}{\widetilde{E}}\\ - \dfrac{1}{6} \aInnerproduct{\nabla \times \left(H(t^{n + 1}) + 4 H(t^n)  + H(t^{n - 1}) \right)}{\widetilde{E}} = \aInnerproduct{\varepsilon R_E^n + \dfrac{\Delta t^2}{6} \left(\nabla r_p^n - \nabla \times r_H^n \right)}{\widetilde{E}},  \label{eqn:remainder_E_ts4}
   \end{split}
\end{equation}
\begin{equation}
  \aInnerproduct{\mu \dfrac{H(t^{n + 1}) - H(t^{n - 1})}{2 \Delta t}}{\widetilde{H}} +  \dfrac{1}{6} \aInnerproduct{\nabla \times \left( E(t^{n + 1}) + 4 E(t^n)  + E(t^{n - 1}) \right)}{\widetilde{H}} = \aInnerproduct{\mu R_H^n + \dfrac{\Delta t^2}{6} \nabla \times r_E^n}{\widetilde{H}}.  \label{eqn:remainder_H_ts4}
\end{equation}
\end{subequations}
Then, subtracting the TS$_4$ scheme equations as in Equations~\labelcref{eqn:maxwell_p_ts4,eqn:maxwell_E_ts4,eqn:maxwell_H_ts4} leads us to the following set of equations:
    \begin{equation*}
   \aInnerproduct{\varepsilon^{-1} \dfrac{e_p^{n + 1} - e_p^{n - 1}}{2 \Delta t}}{\widetilde{p}} - \dfrac{1}{6} \aInnerproduct{e_E^{n + 1} +4 e_E^n  + e_E^{n - 1}}{\nabla \widetilde{p}} = \aInnerproduct{\varepsilon^{-1} R_p^n + \dfrac{\Delta t^2}{6} \nabla \cdot r_E^n}{\widetilde{p}}, 
   \end{equation*}
\begin{equation*}
\begin{split}
   \dfrac{1}{6} \aInnerproduct{\nabla \left(e_p^{n + 1} + 4 e_p^n  + e_p^{n - 1} \right)}{\widetilde{E}}  + \aInnerproduct{\varepsilon \dfrac{e_E^{n + 1} - e_E^{n - 1}}{2 \Delta t}}{\widetilde{E}} - \dfrac{1}{6} \aInnerproduct{\nabla \times \left(e_H^{n + 1} + 4 e_H^n  + e_H^{n - 1} \right)}{\widetilde{E}} \\ = \aInnerproduct{\varepsilon R_E^n + \dfrac{\Delta t^2}{6} \left(\nabla r_p^n - \nabla \times r_H^n \right)}{\widetilde{E}}, 
   \end{split}
\end{equation*}
\begin{equation*}
  \aInnerproduct{\mu \dfrac{e_H^{n + 1} - e_H^{n - 1}}{2 \Delta t}}{\widetilde{H}} +  \dfrac{1}{6} \aInnerproduct{\nabla \times \left( e_E^{n + 1} + 4 e_E^n  + e_E^{n - 1} \right)}{\widetilde{H}} = \aInnerproduct{\mu R_H^n + \dfrac{\Delta t^2}{6} \nabla \times r_E^n}{\widetilde{H}}. 
\end{equation*}
Now, in this set of weak formulation equations for the errors, we choose the test functions to be $\widetilde{p} = 2 \Delta t \left(e_p^{n + 1} + 4 e_p^n + e_p^{n - 1} \right)$, $\widetilde{E} = 2 \Delta t \left(e_E^{n + 1} + 4 e_E^n + e_E^{n - 1} \right)$ and $\widetilde{H} = 2 \Delta t \left(e_H^{n + 1} + 4 e_H^n + e_H^{n - 1} \right)$. Next, by following essentially the same sequence of steps as in Theorem~\ref{thm:dscrt_enrgy_cnsrvtn_ts4}, we obtain the estimates for these error terms to be:
\begin{multline*}
\norm{e_p^{n + 1} + e_p^n}_{\varepsilon^{-1}}  + 2\aInnerproduct{e_p^{n + 1}}{e_p^n}_{\varepsilon^{-1}} - \norm{e_p^n + e_p^{n - 1}}_{\varepsilon^{-1}} - 2\aInnerproduct{e_p^n}{e_p^{n - 1}}_{\varepsilon^{-1}} + \norm{e_E^{n + 1} + e_E^n}_\varepsilon + 2\aInnerproduct{e_E^{n + 1}}{e_E^n}_\varepsilon \\ - \norm{e_E^n + e_E^{n - 1}}_\varepsilon - 2\aInnerproduct{e_E^n}{e_E^{n - 1}}_\varepsilon + \norm{e_H^{n + 1} + e_H^n}_\mu  + 2\aInnerproduct{e_H^{n + 1}}{e_H^n}_\mu - \norm{e_H^n + e_H^{n - 1}}_\mu - 2\aInnerproduct{e_H^n}{e_H^{n - 1}}_\mu \\ = 2 \Delta t\left(\aInnerproduct{\varepsilon^{-1} R_p^n + \dfrac{\Delta t^2}{6} \nabla \cdot r_E^n}{e_p^{n + 1} + 4 e_p^n + e_p^{n - 1}} + \aInnerproduct{\varepsilon R_E^n + \dfrac{\Delta t^2}{6} \left(\nabla r_p^n - \nabla \times r_H^n \right)}{e_E^{n + 1} + 4 e_E^n + e_E^{n - 1}} \right. \\ \left.+ \aInnerproduct{\mu R_H^n + \dfrac{\Delta t^2}{6} \nabla \times r_E^n}{e_H^{n + 1} + 4 e_H^n + e_H^{n - 1}} \right).
\end{multline*}
Applying the Cauchy-Schwarz, triangle and AM-GM inequalities appropriately on right-hand side terms, and using that $\norm{a+b}^2 = \norm{a}^2 + \norm{b}^2 + 2 \aInnerproduct{a}{b}$ for the left-hand side terms of the above equation, we obtain:
\begin{multline*}
  \norm{e_p^{n +1}}^2_{\varepsilon^{-1}} + \norm{e_p^n}^2_{\varepsilon^{-1}} + 4 \aInnerproduct{e_p^{n+1}}{e_p^n}_{\varepsilon^{-1}} -  \norm{e_p^n}^2_{\varepsilon^{-1}} - \norm{e_p^{n - 1}}^2_{\varepsilon^{-1}} -  4 \aInnerproduct{e_p^n}{e_p^{n-1}}_{\varepsilon^{-1}} \\ + \norm{e_E^{n +1}}^2_\varepsilon + \norm{e_E^n}^2_\varepsilon + 4 \aInnerproduct{e_E^{n+1}}{e_E^n}_\varepsilon - \norm{e_E^n}^2_\varepsilon - \norm{e_E^{n - 1}}^2_\varepsilon -  4 \aInnerproduct{e_E^n}{e_E^{n-1}}_\varepsilon \\ + \norm{e_H^{n +1}}^2_\mu + \norm{e_H^n}^2_\mu + 4 \aInnerproduct{e_H^{n+1}}{e_H^n}_\mu -  \norm{e_H^n}^2_\mu - \norm{e_H^{n - 1}}^2_\mu -  4 \aInnerproduct{e_H^n}{e_H^{n-1}}_\mu \\
  \le 3 \Delta t \left[ \norm{e_p^{n +1}}^2_{\varepsilon^{-1}} + \norm{e_p^n}^2_{\varepsilon^{-1}} + \norm{e_p^{n - 1}}^2_{\varepsilon^{-1}} + \norm{e_E^{n + 1}}^2_\varepsilon + \norm{e_E^n}^2_\varepsilon + \norm{e_E^{n - 1}}^2_\varepsilon + \norm{e_H^{n+1}}^2_\mu +  \norm{e_H^n}^2_\mu + \norm{e_H^{n-1}}^2_\mu \right] \\ +
18 \Delta t \left[ \norm{R_p^n}^2_{\varepsilon^{-1}} + \norm{R_E^n}^2_{\varepsilon} + \norm{R_H^n}^2_{\mu} + \dfrac{\Delta t^4}{36} \left( \norm{\nabla r_p^n} ^2_{\varepsilon^{-1}} + \norm{\nabla \cdot r_E^n}^2_{\varepsilon} + \varepsilon^{-1} \mu^{-1} \left(\norm{\nabla \times r_E^n}^2_{\varepsilon} + \norm{\nabla \times r_H^n}^2_{\mu} \right) \right) \right].
\end{multline*}
Summing over $n = 1$ to $N-2$ on both sides, we get:
\begin{multline*}
  \norm{e_p^{N - 1}}^2_{\varepsilon^{-1}} + \norm{e_p^{N-2}}^2_{\varepsilon^{-1}} + 4 \aInnerproduct{e_p^{N - 1}}{e_p^{N-2}}_{\varepsilon^{-1}} -  \norm{e_p^1}^2_{\varepsilon^{-1}} - \norm{e_p^0}^2_{\varepsilon^{-1}} -  4 \aInnerproduct{e_p^1}{e_p^0}_{\varepsilon^{-1}} + \norm{e_E^{N-1}}^2_\varepsilon + \norm{e_E^{N-2}}^2_\varepsilon \\ + 4 \aInnerproduct{e_E^{N-1}}{e_E^{N-2}}_\varepsilon - \norm{e_E^1}^2_\varepsilon - \norm{e_E^0}^2_\varepsilon -  4 \aInnerproduct{e_E^1}{e_E^0}_\varepsilon + \norm{e_H^{N-1}}^2_\mu + \norm{e_H^{N-2}}^2_\mu + 4 \aInnerproduct{e_H^{N-1}}{e_H^{N-2}}_\mu \\ -  \norm{e_H^1}^2_\mu - \norm{e_H^0}^2_\mu -  4 \aInnerproduct{e_H^1}{e_H^0}_\mu 
  \le 9 \Delta t \sum\limits_{n=0}^{N-1} \left[ \norm{e_p^n}^2_{\varepsilon^{-1}} + \norm{e_E^n}^2_\varepsilon + \norm{e_H^n}^2_\mu \right] +
18 \Delta t \sum\limits_{n=1}^{N-2} \left[ \norm{R_p^n}^2_{\varepsilon^{-1}} + \norm{R_E^n}^2_{\varepsilon} \right. \\ \left. + \norm{R_H^n}^2_{\mu} + \dfrac{\Delta t^4}{36} \left( \norm{\nabla r_p^n} ^2_{\varepsilon^{-1}} + \norm{\nabla \cdot r_E^n}^2_{\varepsilon} + \varepsilon^{-1} \mu^{-1} \left(\norm{\nabla \times r_E^n}^2_{\varepsilon} + \norm{\nabla \times r_H^n}^2_{\mu} \right) \right) \right].
\end{multline*}
Since $\aInnerproduct{\cdot}{\cdot}_\alpha$ are non-negative for $\alpha = \{\varepsilon^{-1}, \varepsilon, \mu\}$, thus using the Cauchy-Schwarz inequality we obtain that:
\begin{multline*}
  \norm{e_p^{N - 1}}^2_{\varepsilon^{-1}} + \norm{e_E^{N-1}}^2_\varepsilon  + \norm{e_H^{N-1}}^2_\mu 
  \le 3\left(\norm{e_p^0}^2_{\varepsilon^{-1}} + \norm{e_E^0}^2_\varepsilon + \norm{e_H^0}^2_\mu + \norm{e_p^1}^2_{\varepsilon^{-1}} + \norm{e_E^1}^2_\varepsilon + \norm{e_H^1}^2_\mu \right) \\ + 18 \Delta t \sum\limits_{n=0}^{N-1} \left[ \norm{e_p^n}^2_{\varepsilon^{-1}} + \norm{e_E^n}^2_\varepsilon + \norm{e_H^n}^2_\mu \right] +
18 \Delta t \sum\limits_{n=0}^{N-1} \left[ \norm{R_p^n}^2_{\varepsilon^{-1}} + \norm{R_E^n}^2_{\varepsilon} + \norm{R_H^n}^2_{\mu} \right. \\ \left. + \dfrac{\Delta t^4}{36} \left( \norm{\nabla r_p^n} ^2_{\varepsilon^{-1}} + \norm{\nabla \cdot r_E^n}^2_{\varepsilon} + \varepsilon^{-1} \mu^{-1} \left(\norm{\nabla \times r_E^n}^2_{\varepsilon} + \norm{\nabla \times r_H^n}^2_{\mu} \right) \right) \right].
\end{multline*}
Applying the discrete Gronwall inequality similarly as in Theorem~\ref{thm:dscrt_error_estmt_lf4}, we obtain the estimate:
\begin{multline*}
  \norm{e_p^{N - 1}}^2_{\varepsilon^{-1}} + \norm{e_E^{N-1}}^2_\varepsilon  + \norm{e_H^{N-1}}^2_\mu 
  \le 3\left(\norm{e_p^0}^2_{\varepsilon^{-1}} + \norm{e_E^0}^2_\varepsilon + \norm{e_H^0}^2_\mu + \norm{e_p^1}^2_{\varepsilon^{-1}} + \norm{e_E^1}^2_\varepsilon + \norm{e_H^1}^2_\mu \right) \\ +
18 \Delta t \sum\limits_{n=0}^{N-1} \left[ \norm{R_p^n}^2_{\varepsilon^{-1}} + \norm{R_E^n}^2_{\varepsilon} + \norm{R_H^n}^2_{\mu} \right. \\ \left. + \dfrac{\Delta t^4}{36} \left( \norm{\nabla r_p^n} ^2_{\varepsilon^{-1}} + \norm{\nabla \cdot r_E^n}^2_{\varepsilon} + \varepsilon^{-1} \mu^{-1} \left(\norm{\nabla \times r_E^n}^2_{\varepsilon} + \norm{\nabla \times r_H^n}^2_{\mu} \right) \right) \right] \exp\left( 36 T \right).
\end{multline*}
Now, we need to obtain bounding estimates for each of the Taylor remainder terms. So with $u = \{p, E, H\}$ and $n = 0, 1, \dots, N-1$, we argue as following:
\begin{align*}
\norm{R^n_u}^2_\alpha & = \dfrac{1}{4 (24)^2 \Delta t^2} \norm[\bigg]{\int\limits_{\mathclap{t^{n - 1}}}^{\mathclap{t^n}} (t^{n - 1} - s)^4 \dfrac{\partial^5 u}{\partial t^5}(s) ds + \int\limits_{\mathclap{t^n}}^{t^{n + 1}} (t^{n + 1} - s)^4 \dfrac{\partial^5 u}{\partial t^5}(s) ds}^2_\alpha, \\
&\le \dfrac{1}{(48)^2 \left(\Delta t\right)^2} \int\limits_{t^{n - 1}}^{t^{n +1}} (s - t^{n + 1})^8 ds \int\limits_{\mathclap{t^{n - 1}}}^{t^{n + 1}} \norm[\bigg]{\dfrac{\partial^5 u}{\partial t^5}(s)}^2_\alpha ds, \quad \text{(using $t^{n - 1} < t^{n + 1}$ and Cauchy-Schwarz)} \\
&= \dfrac{2}{81} \Delta t^7 \int\limits_{\mathclap{t^{n - 1}}}^{t^{n + 1}} \norm[\bigg]{\dfrac{\partial^5 u}{\partial t^5}(s)}^2_\alpha ds,
\end{align*}
and now summing both sides over $n = 0$ to $N-1$, we have that:
\begin{equation*}
  \sum\limits_{n = 0}^{N-1} \norm{R^n_u}^2_\alpha \le \dfrac{2}{81} \Delta t^7 \int\limits_0^T \norm[\bigg]{\dfrac{\partial^5 u}{\partial t^5}(s)}^2_\alpha ds = \dfrac{2}{81} \Delta t^7 \norm[\bigg]{\dfrac{\partial^5 u}{\partial t^5}}^2_{L^2(0, T; L^2_\alpha(\Omega))}.
\end{equation*}
Similarly, for any operator $\Phi$ acting on $u$ belonging to an appropriate function space on $\Omega$, using the Cauchy-Schwarz and Chebyshev inequalities, we have that:
\begin{equation*}
  \sum\limits_{n = 0}^{N-1} \norm{\Phi (r^n_u)}^2_\alpha \le \dfrac{\Delta t^3}{252} \int\limits_0^T \norm[\bigg]{\dfrac{\partial^4 (\Phi(u))}{\partial t^4}(s)}^2_\alpha ds = \dfrac{\Delta t^3}{252} \norm[\bigg]{\dfrac{\partial^4 (\Phi(u))}{\partial t^4}}^2_{L^2(0, T; L^2_\alpha(\Omega))},
\end{equation*}
for $u = \{p, E, H\}$ and $n = 0, 1, \dots, N-1$. Finally, using the regularity assumptions for $p$, $E$ and $H$, and $1$- and $2$-norm equivalences, we obtain the required result:
\[
  \norm{e_p^{N - 1}}_{\varepsilon^{-1}} + \norm{e_E^{N - 1}}_{\varepsilon} + \norm{e_H^{N - 1}}_{\mu} \le C \left[ \left(\Delta t\right)^4 + \norm{e_p^0}_{\varepsilon^{-1}} + \norm{e_E^0}_{\varepsilon} + \norm{e_H^0}_{\mu} + \norm{e_p^1}_{\varepsilon^{-1}} + \norm{e_E^1}_{\varepsilon} + \norm{e_H^1}_{\mu} \right]. \qedhere
\]
\end{proof}

\subsection{Error Estimate for Full Discretization}

The full error analysis of the three-field Maxwell's systems with arbitrary order de Rham finite elements in conjunction with TS$_4$ is nearly identical to that for the LF$_4$ case in Section~\labelcref{sec:implicit_lf4}. Like earlier, we again define the errors for $p$, $E$ and $H$ at time $n \Delta t$ under the full discretization as the following:
\begin{alignat}{2}
  e_{p_h}^n &\coloneq p(t^n) - p_h^n &&= \eta^n - \eta_h^n, \label{eqn:p_fullerror_ts4} \\
  e_{E_h}^n &\coloneq E(t^n) - E_h^n &&= \zeta^n - \zeta_h^n, \label{eqn:E_fullerror_ts4} \\
  e_{H_h}^n &\coloneq H(t^n) - H_h^n &&= \xi^n - \xi_h^n, \label{eqn:H_fullerror_ts4}
\end{alignat}
and in which we again have the following definitions for the above newly introduced terms:
\begin{alignat}{3}
  \eta^n &\coloneq p(t^n) - \Pi_h^0 p(t^n), &&\qquad \eta_h^n &&\coloneq p_h^n  - \Pi_h^0 p(t^n), \label{eqn:p_fullerror_sub_ts4} \\
  \zeta^n &\coloneq E(t^n) - \Pi_h^1 E(t^n), &&\qquad \zeta_h^n &&\coloneq E_h^n - \Pi_h^1 E(t^n), \label{eqn:E_fullerror_sub_ts4} \\
  \xi^n &\coloneq H(t^n) - \Pi_h^2 H(t^n), &&\qquad \xi_h^n &&\coloneq H_h^n - \Pi_h^2 H(t^n). \label{eqn:H_fullerror_sub_ts4}
\end{alignat}
For the TS$_4$ scheme as in Equations~\labelcref{eqn:maxwell_p_ts4,eqn:maxwell_E_ts4,eqn:maxwell_H_ts4}, using a de Rham sequence of finite dimensional subspaces of the corresponding function spaces for the spatial discretization of $(p^n, E^n, H^n)$, we obtain the following discrete problem: find $(p_h^n, E_h^n, H_h^n) \in U_h \times V_h \times W_h \subseteq \mathring{H}_{\varepsilon^{-1}}^1 \times \mathring{H}_{\varepsilon}(\curl; \Omega) \times \mathring{H}_{\mu}(\divgn; \Omega)$ such that:
    \begin{subequations}
    \begin{equation}
   \aInnerproduct{\varepsilon^{-1} \dfrac{p_h^{n + 1} - p_h^{n - 1}}{2 \Delta t}}{\widetilde{p}} - \dfrac{1}{6} \aInnerproduct{E_h^{n + 1} +4 E_h^n  + E_h^{n - 1}}{\nabla \widetilde{p}} = 0,  \label{eqn:maxwell_p_ts4_full}
   \end{equation}
\begin{equation}
   \dfrac{1}{6} \aInnerproduct{\nabla \left(p_h^{n + 1} + 4 p_h^n  + p_h^{n - 1} \right)}{\widetilde{E}}  + \aInnerproduct{\varepsilon \dfrac{E_h^{n + 1} - E_h^{n - 1}}{2 \Delta t}}{\widetilde{E}} - \dfrac{1}{6} \aInnerproduct{\nabla \times \left(H_h^{n + 1} + 4 H_h^n  + H_h^{n - 1} \right)}{\widetilde{E}} = 0,  \label{eqn:maxwell_E_ts4_full}
\end{equation}
\begin{equation}
  \aInnerproduct{\mu \dfrac{H_h^{n + 1} - H_h^{n - 1}}{2 \Delta t}}{\widetilde{H}} +  \dfrac{1}{6} \aInnerproduct{\nabla \times \left( E_h^{n + 1} + 4 E_h^n  + E_h^{n - 1} \right)}{\widetilde{H}} = 0.  \label{eqn:maxwell_H_ts4_full}
\end{equation}
for all $(\widetilde{p}, \widetilde{E}, \widetilde{H}) \in U_h \times V_h \times W_h$, $n = 1, \dots, N - 1$. As earlier, we again have the same $L^2$ smoothed projection operators, namely, $\Pi_h^0: \mathring{H}^1_{\varepsilon^{-1}}(\Omega) \longto U_h$, $\Pi_h^1: \mathring{H}_{\varepsilon}(\curl; \Omega) \longto V_h$ and $\Pi_h^2: \mathring{H}_{\mu}(\divgn; \Omega) \longto W_h$, and for the initial conditions, we simply again set:
  \begin{equation}
u_h^0 \coloneq \Pi_h^k u_0 \text{~and~} u_h^1 \coloneq \Pi_h^k u_1, \text{~for~} u = \{p, E, H\} \text{~and~} k = 0,1 \text{~or~} 2 \text{~respectively~}. \label{eqn:initial_ts4}
  \end{equation}
\end{subequations}

With this setup, we now state and prove our theorem for the convergence of the errors in the full discretization of our system of Maxwell's equations using TS$_4$ and arbitrary order de Rham finite elements.

\begin{theorem}[Full Error Estimate]\label{thm:full_error_estmt_ts4}
Let $p \in C^5(0, T; \mathring{H}^1_{\varepsilon^{-1}}(\Omega))$, $E \in C^5(0, T; \mathring{H}_{\varepsilon}(\curl; \Omega))$, and $H \in C^5(0, T; \mathring{H}_{\mu}(\divgn; \Omega))$ be the solution to the variational formulation of the Maxwell's equations as in Equations~\labelcref{eqn:maxwell_p_wf,eqn:maxwell_E_wf,eqn:maxwell_H_wf} with sufficient regularity, and let $(p_h^n, E_h^n, H_h^n)$ be the solution of the fully discretized Maxwell's equations using TS$_4$ scheme as in Equations~\labelcref{eqn:maxwell_p_ts4_full,eqn:maxwell_E_ts4_full,eqn:maxwell_H_ts4_full,eqn:initial_lf4}. With Assumption~\labelcref{assume:innerproducts}, if the time step $\Delta t > 0$ and the mesh parameter $h > 0$ are sufficiently small, then there exists a positive bounded constant $C$ independent of both $\Delta t$ and $h$ such that the following error estimate holds:
\[
  \norm{e_{p_h}^{N - 1}}_{\varepsilon^{-1}} + \norm{e_{E_h}^{N - 1}}_{\varepsilon} + \norm{e_{H_h}^{N - 1}}_{\mu} \le C \left[ \Delta t^4 + h^r + h^r \Delta t^4 \right],
\]
where the finite element subspaces $U_h$, $V_h$ and $W_h$ are each spanned by their respective Whitney form basis of polynomial order $r \ge 1$.
\end{theorem}

\begin{proof}
 First, we shall subtract the set of equations for the full discretization as in Equations~\labelcref{eqn:maxwell_p_ts4_full,eqn:maxwell_E_ts4_full,eqn:maxwell_H_ts4_full} from Equations~\labelcref{eqn:remainder_p_ts4,eqn:remainder_E_ts4,eqn:remainder_H_ts4},  and then use the error terms in Equations~\labelcref{eqn:p_fullerror_ts4,eqn:E_fullerror_ts4,eqn:H_fullerror_ts4} and thereby obtain:
    \begin{equation*}
   \aInnerproduct{\varepsilon^{-1} \dfrac{e_{p_h}^{n + 1} - e_{p_h}^{n - 1}}{2 \Delta t}}{\widetilde{p}} - \dfrac{1}{6} \aInnerproduct{e_{E_h}^{n + 1} +4 e_{E_h}^n  + e_{E_h}^{n - 1}}{\nabla \widetilde{p}} = \aInnerproduct{\varepsilon^{-1} R_p^n + \dfrac{\Delta t^2}{6} \nabla \cdot r_E^n}{\widetilde{p}}, 
   \end{equation*}
\begin{equation*}
\begin{split}
   \dfrac{1}{6} \aInnerproduct{\nabla \left(e_{p_h}^{n + 1} + 4 e_{p_h}^n  + e_{p_h}^{n - 1} \right)}{\widetilde{E}}  + \aInnerproduct{\varepsilon \dfrac{e_{E_h}^{n + 1} - e_{E_h}^{n - 1}}{2 \Delta t}}{\widetilde{E}} - \dfrac{1}{6} \aInnerproduct{\nabla \times \left(e_{H_h}^{n + 1} + 4 e_{H_h}^n  + e_{H_h}^{n - 1} \right)}{\widetilde{E}} \\ = \aInnerproduct{\varepsilon R_E^n + \dfrac{\Delta t^2}{6} \left(\nabla r_p^n - \nabla \times r_H^n \right)}{\widetilde{E}}, 
   \end{split}
\end{equation*}
\begin{equation*}
  \aInnerproduct{\mu \dfrac{e_{H_h}^{n + 1} - e_{H_h}^{n - 1}}{2 \Delta t}}{\widetilde{H}} +  \dfrac{1}{6} \aInnerproduct{\nabla \times \left( e_{E_h}^{n + 1} + 4 e_{E_h}^n  + e_{E_h}^{n - 1} \right)}{\widetilde{H}} = \aInnerproduct{\mu R_H^n + \dfrac{\Delta t^2}{6} \nabla \times r_E^n}{\widetilde{H}}. 
\end{equation*}
 Next, using the values of the error terms $e_{p_h}^n$, $e_{E_h}^n$ and $e_{H_h}^n$ as in Equations~\labelcref{eqn:p_fullerror_ts4,eqn:E_fullerror_ts4,eqn:H_fullerror_ts4} in the above equations, we get:
     \begin{multline*}
   \aInnerproduct{\varepsilon^{-1} \dfrac{\left(\eta^{n + 1} - \eta^{n - 1}\right) - \left( \eta^{n + 1}_h - \eta^{n - 1}_h \right)}{2 \Delta t}}{\widetilde{p}} - \dfrac{1}{6} \aInnerproduct{\left(\zeta^{n + 1} +4 \zeta^n  + \zeta^{n - 1}\right) - \left(\zeta_h^{n + 1} +4 \zeta_h^n  + \zeta_h^{n - 1}\right)}{\nabla \widetilde{p}} \\ = \aInnerproduct{\varepsilon^{-1} R_p^n + \dfrac{\Delta t^2}{6} \nabla \cdot r_E^n}{\widetilde{p}}, 
   \end{multline*}
\begin{multline*}
   \dfrac{1}{6} \aInnerproduct{\nabla \left(\eta^{n + 1} + 4 \eta^n  + \eta^{n - 1} \right) - \nabla \left(\eta_h^{n + 1} + 4 \eta_h^n  + \eta_h^{n - 1} \right)}{\widetilde{E}}  + \aInnerproduct{\varepsilon \dfrac{\left( \zeta^{n + 1} - \zeta^{n - 1} \right) - \left(\zeta^{n + 1}_h - \zeta^{n - 1}_h \right)}{2 \Delta t}}{\widetilde{E}} \\ - \dfrac{1}{6} \aInnerproduct{\left(\xi^{n + 1} + 4 \xi^n  + \xi^{n - 1} \right) - \left(\xi_h^{n + 1} + 4 \xi_h^n  + \xi_h^{n - 1} \right)}{\nabla \times \widetilde{E}} = \aInnerproduct{\varepsilon R_E^n + \dfrac{\Delta t^2}{6} \left(\nabla r_p^n - \nabla \times r_H^n \right)}{\widetilde{E}}, 
\end{multline*}
\begin{multline*}
  \aInnerproduct{\mu \dfrac{\left(\xi^{n + 1} - \xi^{n - 1} \right) - \left( \xi^{n + 1}_h - \xi^{n - 1}_h \right)}{2 \Delta t}}{\widetilde{H}} +  \dfrac{1}{6} \aInnerproduct{\nabla \times \left(\zeta^{n + 1} + 4 \zeta^n  + \zeta^{n - 1} \right) - \nabla \times \left(\zeta_h^{n + 1} + 4 \zeta_h^n  + \zeta_h^{n - 1} \right)}{\widetilde{H}} \\ = \aInnerproduct{\mu R_H^n + \dfrac{\Delta t^2}{6} \nabla \times r_E^n}{\widetilde{H}}. 
\end{multline*}
Since these equations are true for all $(\widetilde{p}, \widetilde{E}, \widetilde{H}) \in U_h \times V_h \times W_h$, we choose $\widetilde{p} = -2 \Delta t \left( \eta_h^{n + 1} + 4 \eta_h^n + \eta_h^{n - 1}\right)$, $\widetilde{E} = -2 \Delta t \left( \zeta_h^{n +1} +4 \zeta_h^n + \zeta_h^{n -1} \right)$ and $\widetilde{H} = -2 \Delta t \left( \xi_h^{n + 1} + 4 \xi_h^n + \xi_h^{n - 1} \right)$ and using the fact that $\nabla U_h \subseteq V_h$ and $\nabla \times V_h \subseteq W_h$, we obtain that:
\begin{multline}
 \ainnerproduct{\varepsilon^{-1} \left( \eta^{n + 1}_h - \eta^{n - 1}_h \right)}{\eta^{n + 1}_h + 4 \eta^n_h + \eta^{n -1}_h} + \ainnerproduct{\varepsilon \left( \zeta^{n + 1}_h - \zeta^{n - 1}_h \right)}{\zeta^{n + 1}_h + 4 \zeta^n_h + \zeta^{n - 1}_h} \\ + \ainnerproduct{\mu \left( \xi^{n + 1}_h - \xi^{n - 1}_h \right)}{ \xi^{n + 1}_h + 4 \xi^n_h + \xi^{n - 1}_h} =  \ainnerproduct{\varepsilon^{-1} \left( \eta^{n + 1} - \eta^{n - 1} \right)}{\eta^{n + 1}_h + 4 \eta^n_h + \eta^{n -1}_h} \\ + \ainnerproduct{\varepsilon \left( \zeta^{n + 1} - \zeta^{n - 1} \right)}{\zeta^{n + 1}_h + 4 \zeta^n_h + \zeta^{n - 1}_h} + \ainnerproduct{\mu \left( \xi^{n + 1} - \xi^{n - 1} \right)}{ \xi^{n + 1}_h + 4 \xi^n_h + \xi^{n - 1}_h}  \, \\ +
2 \Delta t \left[\ainnerproduct{-\varepsilon^{-1} R_p^n + \dfrac{\Delta t^2}{6} \nabla \cdot r_E^n}{\eta_h^{n + 1} + 4 \eta_h^n + \eta_h^{n - 1}} \right. \\ + \left. \ainnerproduct{-\varepsilon R_E^n + \dfrac{\Delta t^2}{6} \left(\nabla r_p^n - \nabla \times r_H^n \right)}{\zeta_h^{n +1} +4 \zeta_h^n + \zeta_h^{n -1}} + \ainnerproduct{-\mu R_H^n + \dfrac{\Delta t^2}{6} \nabla \times r_E^n}{\xi_h^{n + 1} + 4 \xi_h^n + \xi_h^{n - 1}} \right]. \label{eqn:suberror_p+E+H_ts4}
\end{multline}
Consider now that $\varepsilon^{-1} \left(\eta^{n + 1} - \eta^{n - 1}\right) = \varepsilon^{-1} \left(I - \Pi_h^0\right) \left( p(t^{n + 1}) - p(t^{n - 1})\right)$ by Equation~\eqref{eqn:p_fullerror_sub_ts4}. Using the Taylor theorem with remainder as in Theorem~\ref{thm:dscrt_error_estmt_ts4}, applying the Cauchy-Schwarz, AM-GM, and triangle inequalities, we have the following inequality:
\begin{align*}
\ainnerproduct{\varepsilon^{-1} \left( \eta^{n + 1} - \eta^{n - 1} \right)}{\eta^{n + 1}_h + 4 \eta_h^n + \eta^{n - 1}_h} & = 2 \Delta t \ainnerproduct{\varepsilon^{-1} \left( I - \Pi_h^0 \right) \left( \dfrac{\partial p}{\partial t} (t^n) +  \dfrac{\Delta t^2}{6} \dfrac{\partial^3 p}{\partial t^3 (t^n) }\right)}{\eta^{n + 1}_h + 4 \eta_h^n + \eta^{n - 1}_h} \\ & + 2 \Delta t \ainnerproduct{\varepsilon^{-1} \left( I - \Pi_h^0 \right) R_p^n}{\eta^{n + \frac{1}{2}}_h + \eta^{n - \frac{1}{2}}_h} \\ 
 & \le 18 \Delta t \bigg[ \norm[\bigg]{(I - \Pi_h^0) \dfrac{\partial p}{\partial t}(t^n)}^2_{\varepsilon^{-1}} \!\! + \dfrac{\Delta t^4}{36} \norm[\bigg]{(I - \Pi_h^0) \dfrac{\partial^3 p}{\partial t^3}(t^n)}^2_{\varepsilon^{-1}} \!\! \\ & +  \norm[\bigg]{(I - \Pi_h^0) R_p^n}^2_{\varepsilon^{-1}} \bigg] \! + 3 \Delta t \bigg[ \norm{\eta_h^{n + 1}}^2_{\varepsilon^{-1}} + \norm{\eta_h^n}^2_{\varepsilon^{-1}} + \norm{\eta_h^{n -1}}^2_{\varepsilon^{-1}} \bigg].
\end{align*}
Similarly, using Equations~\labelcref{eqn:E_fullerror_sub_ts4,eqn:H_fullerror_sub_ts4} for the full error terms for $E$ and $H$, we obtain the following:
\begin{gather*}
\begin{split}
\ainnerproduct{\varepsilon \left( \zeta^{n +1} - \zeta^{n -1} \right)}{\zeta^{n + 1}_h + 4 \zeta^n_h + \zeta^{n - 1}_h} \le 18 \Delta t \bigg[ \norm[\bigg]{(I - \Pi_h^1) \dfrac{\partial E}{\partial t}(t^n)}^2_{\varepsilon} \!\! + \dfrac{\Delta t^4}{36} \norm[\bigg]{(I - \Pi_h^1) \dfrac{\partial^3 E}{\partial t^3}(t^n)}^2_{\varepsilon} \!\! \\  +  \norm[\bigg]{(I - \Pi_h^1) R_E^n}^2_{\varepsilon} \bigg] \! + 3 \Delta t \bigg[ \norm{\zeta_h^{n + 1}}^2_{\varepsilon} + \norm{\zeta_h^n}^2_{\varepsilon} + \norm{\zeta_h^{n - 1}}^2_{\varepsilon} \bigg],
\end{split} \\ 
\begin{split}
\ainnerproduct{\mu \left( \xi^{n + 1} - \xi^{n - 1} \right)}{\xi^{n + 1}_h + 4 \xi^n_h + \xi^{n - 1}_h} \le 18 \Delta t \bigg[ \norm[\bigg]{(I - \Pi_h^2) \dfrac{\partial H}{\partial t}(t^{n})}^2_{\mu} \!\! + \dfrac{\Delta t^4}{36} \norm[\bigg]{(I - \Pi_h^2) \dfrac{\partial^3 H}{\partial t^3}(t^n)}^2_{\mu} \!\! \\  +  \norm[\bigg]{(I - \Pi_h^2) R_H^n}^2_{\mu} \bigg] \! + 3 \Delta t \bigg[ \norm{\xi_h^{n + 1}}^2_{\mu} + \norm{\xi_h^n}^2_{\mu} +  \norm{\xi_h^{n - 1}}^2_{\mu} \bigg].
\end{split}
\end{gather*}
Using these inequalities for $\eta$, $\zeta$ and $\xi$ in Equation~\eqref{eqn:suberror_p+E+H_ts4}, we thus obtain the following estimate:
\begin{multline*}
\norm{\eta_h^{n + 1}}^2_{\varepsilon^{-1}} + \norm{\eta_h^n}^2_{\varepsilon^{-1}} + 4 \ainnerproduct{\eta_h^{n + 1}}{\eta_h^n}_{\varepsilon^{-1}} - \norm{\eta_h^n}^2_{\varepsilon^{-1}} - \norm{\eta_h^{n -1}}^2_{\varepsilon^{-1}} - 4 \ainnerproduct{\eta_h^n}{\eta_h^{n - 1}}_{\varepsilon^{-1}} + \norm{\zeta_h^{n + 1}}^2_{\varepsilon} + \norm{\zeta_h^n}^2_{\varepsilon} \\ + 4 \ainnerproduct{\zeta_h^{n + 1}}{\zeta_h^n}_\varepsilon - \norm{\zeta_h^n}^2_{\varepsilon}  - \norm{\zeta_h^{n - 1}}^2_{\varepsilon} - 4 \ainnerproduct{\zeta_h^n}{\zeta_h^{n-1}}_\varepsilon + \norm{\xi_h^{n+1}}^2_{\mu} + \norm{\xi_h^n}^2_{\mu} + 4 \ainnerproduct{\xi_h^{n + 1}}{\xi_h^n}_\mu - \norm{\xi_h^n}^2_{\mu} - \norm{\xi_h^{n-1}}^2_{\mu} \\ - 4 \ainnerproduct{\xi_h^n}{\xi_h^{n-1}}_\mu \le
18 \Delta t \bigg[ \norm[\bigg]{(I - \Pi_h^0) \dfrac{\partial p}{\partial t}(t^n)}^2_{\varepsilon^{-1}} + \norm[\bigg]{(I - \Pi_h^1) \dfrac{\partial E}{\partial t}(t^n)}^2_{\varepsilon} + \norm[\bigg]{(I - \Pi_h^2) \dfrac{\partial H}{\partial t}(t^n)}^2_\mu \\ + \dfrac{\Delta t^4}{36} \left(\norm[\bigg]{(I - \Pi_h^0) \dfrac{\partial^3 p}{\partial t^3}(t^n)}^2_{\varepsilon^{-1}} + \norm[\bigg]{(I - \Pi_h^1) \dfrac{\partial^3 E}{\partial t^3}(t^n)}^2_{\varepsilon} + \norm[\bigg]{(I - \Pi_h^2) \dfrac{\partial^3 H}{\partial t^3}(t^n)}^2_{\mu} \right) \\ + \norm[\bigg]{(I - \Pi_h^0) R_p^n}^2_{\varepsilon^{-1}} + \norm[\bigg]{(I - \Pi_h^1) R_E^n}^2_{\varepsilon} + \norm[\bigg]{(I - \Pi_h^2) R_H^n}^2_\mu + \norm{R_p^n}^2_{\varepsilon^{-1}} + \norm{R_E^n}^2_{\varepsilon} + \norm{R_H^n}^2_{\mu} \\ + \dfrac{\Delta t^4}{36} \left( \norm{\nabla r_p^n} ^2_{\varepsilon^{-1}} + \norm{\nabla \cdot r_E^n}^2_{\varepsilon} + \varepsilon^{-1} \mu^{-1} \left(\norm{\nabla \times r_E^n}^2_{\varepsilon} + \norm{\nabla \times r_H^n}^2_{\mu} \right)\right) \bigg]\\ + 6 \Delta t \left[\norm{\eta_h^{n + 1}}^2_{\varepsilon^{-1}} + \norm{\eta_h^n}^2_{\varepsilon^{-1}} + \norm{\eta_h^{n - 1}}^2_{\varepsilon^{-1}} + \norm{\zeta_h^{n + 1}}^2_{\varepsilon} + \norm{\zeta_h^n}^2_{\varepsilon} + \norm{\zeta_h^{n - 1}}^2_{\varepsilon} + \norm{\xi_h^{n + 1}}^2_\mu + \norm{\xi_h^n}^2_\mu + \norm{\xi_h^{n - 1}}^2_\mu \right].
\end{multline*}
Summing from $n = 1$ to $N-2$ and using the Cauchy-Schwarz inequality along with non-negativity of inner products $\aInnerproduct{\cdot}{\cdot}_\alpha$ for $\alpha = \{\varepsilon^{-1}, \varepsilon, \mu\}$, we have that:
\begin{multline*}
\norm{\eta_h^{N - 1}}^2_{\varepsilon^{-1}} + \norm{\zeta_h^{N - 1}}^2_{\varepsilon} + \norm{\xi_h^{N - 1}}^2_{\mu} \le 3 \left(\norm{\eta_h^0}^2_{\varepsilon^{-1}} + \norm{\zeta_h^0}^2_\varepsilon + \norm{\xi_h^0}^2_\mu + \norm{\eta_h^1}^2_{\varepsilon^{-1}} + \norm{\zeta_h^1}^2_\varepsilon + \norm{\xi_h^1}^2_\mu \right) \\ + 18 \Delta t \sum\limits_{n=0}^{N-1} \bigg[ \norm[\bigg]{(I - \Pi_h^0) \dfrac{\partial p}{\partial t}(t^n)}^2_{\varepsilon^{-1}} + \norm[\bigg]{(I - \Pi_h^1) \dfrac{\partial E}{\partial t}(t^n)}^2_{\varepsilon} + \norm[\bigg]{(I - \Pi_h^2) \dfrac{\partial H}{\partial t}(t^n)}^2_\mu \\ + \dfrac{\Delta t^4}{36} \left(\norm[\bigg]{(I - \Pi_h^0) \dfrac{\partial^3 p}{\partial t^3}(t^n)}^2_{\varepsilon^{-1}} + \norm[\bigg]{(I - \Pi_h^1) \dfrac{\partial^3 E}{\partial t^3}(t^n)}^2_{\varepsilon} + \norm[\bigg]{(I - \Pi_h^2) \dfrac{\partial^3 H}{\partial t^3}(t^n)}^2_{\mu} \right) + \norm[\bigg]{(I - \Pi_h^0) R_p^n}^2_{\varepsilon^{-1}} \\ + \norm[\bigg]{(I - \Pi_h^1) R_E^n}^2_{\varepsilon} + \norm[\bigg]{(I - \Pi_h^2) R_H^n}^2_\mu + \norm{R_p^n}^2_{\varepsilon^{-1}} + \norm{R_E^n}^2_{\varepsilon} + \norm{R_H^n}^2_{\mu} + \dfrac{\Delta t^4}{36} \left( \norm{\nabla r_p^n} ^2_{\varepsilon^{-1}} + \norm{\nabla \cdot r_E^n}^2_{\varepsilon} \right. \\ \left. + \varepsilon^{-1} \mu^{-1} \left(\norm{\nabla \times r_E^n}^2_{\varepsilon} + \norm{\nabla \times r_H^n}^2_{\mu} \right)\right) \bigg] + 18 \Delta t  \sum\limits_{n=0}^{N-1} \left[\norm{\eta_h^n}^2_{\varepsilon^{-1}} + \norm{\zeta_h^n}^2_{\varepsilon} + \norm{\xi_h^n}^2_\mu \right].
\end{multline*}
Now, applying the discrete Gronwall inequality with $\Delta t < 1/24$, we get that:
\begin{multline*}
\norm{\eta_h^{N - 1}}^2_{\varepsilon^{-1}} + \norm{\zeta_h^{N - 1}}^2_{\varepsilon} + \norm{\xi_h^{N - 1}}^2_{\mu} \le \Bigg[3 \left(\norm{\eta_h^0}^2_{\varepsilon^{-1}} + \norm{\zeta_h^0}^2_\varepsilon + \norm{\xi_h^0}^2_\mu + \norm{\eta_h^1}^2_{\varepsilon^{-1}} + \norm{\zeta_h^1}^2_\varepsilon + \norm{\xi_h^1}^2_\mu \right) \\ + 18 \Delta t \sum\limits_{n=0}^{N-1} \bigg[ \norm[\bigg]{(I - \Pi_h^0) \dfrac{\partial p}{\partial t}(t^n)}^2_{\varepsilon^{-1}} + \norm[\bigg]{(I - \Pi_h^1) \dfrac{\partial E}{\partial t}(t^n)}^2_{\varepsilon} + \norm[\bigg]{(I - \Pi_h^2) \dfrac{\partial H}{\partial t}(t^n)}^2_\mu \\ + \dfrac{\Delta t^4}{36} \left(\norm[\bigg]{(I - \Pi_h^0) \dfrac{\partial^3 p}{\partial t^3}(t^n)}^2_{\varepsilon^{-1}} + \norm[\bigg]{(I - \Pi_h^1) \dfrac{\partial^3 E}{\partial t^3}(t^n)}^2_{\varepsilon} + \norm[\bigg]{(I - \Pi_h^2) \dfrac{\partial^3 H}{\partial t^3}(t^n)}^2_{\mu} \right) + \norm[\bigg]{(I - \Pi_h^0) R_p^n}^2_{\varepsilon^{-1}} \\ + \norm[\bigg]{(I - \Pi_h^1) R_E^n}^2_{\varepsilon} + \norm[\bigg]{(I - \Pi_h^2) R_H^n}^2_\mu + \norm{R_p^n}^2_{\varepsilon^{-1}} + \norm{R_E^n}^2_{\varepsilon} + \norm{R_H^n}^2_{\mu} + \dfrac{\Delta t^4}{36} \left( \norm{\nabla r_p^n} ^2_{\varepsilon^{-1}} + \norm{\nabla \cdot r_E^n}^2_{\varepsilon} \right. \\ \left. + \varepsilon^{-1} \mu^{-1} \left(\norm{\nabla \times r_E^n}^2_{\varepsilon} + \norm{\nabla \times r_H^n}^2_{\mu} \right)\right) \bigg] \Bigg] \exp \left(36 T \right).
\end{multline*}
Using our estimates for the Taylor remainders from Theorem~\labelcref{thm:dscrt_error_estmt_lf4} for the terms on the right-hand side of the above inequality, we further get that:
\begin{multline*}
 18  \Delta t \sum\limits_{n = 0}^{N-1} \Big[ \norm{R^n_p}^2_{\varepsilon^{-1}} + \norm{R^n_E}^2_{\varepsilon} + \norm{R^n_H}^2_\mu \Big] \le \\
 \dfrac{4}{9} \Delta t^8 \Bigg[ \norm[\bigg]{\dfrac{\partial^5 p}{\partial t^5}}^2_{L^2(0, T; L^2_{\varepsilon^{-1}}(\Omega))} \!\! + \, \norm[\bigg]{\dfrac{\partial^5 E}{\partial t^5}}^2_{L^2(0, T; L^2_\varepsilon(\Omega))} \!\! + \, \norm[\bigg]{\dfrac{\partial^5 H}{\partial t^5}}^2_{L^2(0, T; L^2_\mu(\Omega))} \Bigg],
\end{multline*}
and that:
\begin{multline*}
\dfrac{\Delta t^5}{2} \sum\limits_{n = 0}^N \left[ \norm{\nabla r_p^n} ^2_{\varepsilon^{-1}} + \norm{\nabla \cdot r_E^n}^2_{\varepsilon} + \varepsilon^{-1} \mu^{-1} \left(\norm{\nabla \times r_E^n}^2_{\varepsilon} + \norm{\nabla \times r_H^n}^2_{\mu} \right)\right] \\ \le \dfrac{\Delta t^8}{504} \left[ \norm[\bigg]{\dfrac{\partial^4 \left(\nabla p \right)}{\partial t^4}}^2_{L^2(0, T; L^2_{\varepsilon^{-1}}(\Omega))} \!\! + 
  \norm[\bigg]{\dfrac{\partial^4 (\nabla \cdot E)}{\partial t^4}}^2_{L^2(0, T; L^2_\varepsilon(\Omega))} \!\! \right. \\ \left.+ \varepsilon^{-1} \mu^{-1} \left( \norm[\bigg]{\dfrac{\partial^4 (\nabla \times E)}{\partial t^4}}^2_{L^2(0, T; L^2_{\varepsilon}(\Omega))} \!\! + \norm[\bigg]{\dfrac{\partial^4 (\nabla \times H)}{\partial t^4}}^2_{L^2(0, T; L^2_\mu(\Omega))}\right)\right].
\end{multline*}
Now, for $p \in \mathring{H}_{\varepsilon^{-1}}^1(\Omega)$, $E \in \mathring{H}_\varepsilon(\curl; \Omega)$ and $H \in \mathring{H}_\mu(\divgn; \Omega)$, there exists positive bounded constants $C_{1, u}$, $C_{2, u}$, $C_{3, u}$ for $u = \{p, E, H\}$ such that we have the following error bounds for the $L^2$ projections:
\[
  \norm[\bigg]{(I - \Pi_h^k) \dfrac{\partial u}{\partial t}(t^n)}_\alpha \le C_{1, u} h^r \norm[\bigg]{\dfrac{\partial u}{\partial t}(t^n)}_\alpha,  \norm[\bigg]{(I - \Pi_h^k) \dfrac{\partial^3 u}{\partial t^3}(t^n)}_\alpha \le C_{2, u} h^r \norm[\bigg]{\dfrac{\partial^3 u}{\partial t^3}(t^n)}_\alpha,
\]
\[
  \norm[\bigg]{(I - \Pi_h^0) R_p^n}_\alpha \le C_{3, p} h^r \norm[\bigg]{R_p^n}_\alpha,  
\]
where $k = 0, 1, 2$ and $\alpha = \{\varepsilon^{-1}, \varepsilon, \mu\}$, respectively.
  Set $C_0 \coloneq \max\limits_{u \in \{p, E, H\}} \{C_{1, u}, C_{2, u}, C_{3, u} \}$. We therefore have that:
 \begin{multline*}
  18  \Delta t \sum\limits_{n = 0}^{N-1} \left[ \norm[\bigg]{(I - \Pi_h^0) \dfrac{\partial p}{\partial t}(t^n)}^2_{\varepsilon^{-1}} \!\! + \norm[\bigg]{(I - \Pi_h^1) \dfrac{\partial E}{\partial t}(t^n)}^2_{\varepsilon} \!\! + \norm[\bigg]{(I - \Pi_h^2) \dfrac{\partial H}{\partial t}(t^n)}^2_\mu \right] \\
\le 18 C_0 h^{2 r} \sum\limits_{n = 0}^{N-1} \Delta t \left[ \norm[\bigg]{\dfrac{\partial p}{\partial t}(t^n)}^2_{\varepsilon^{-1}} \!\! + \norm[\bigg]{\dfrac{\partial E}{\partial t}(t^n)}^2_{\varepsilon} \!\! + \norm[\bigg]{\dfrac{\partial H}{\partial t}(t^n)}^2_{\mu} \right] \\
\le 18 C_0 h^{2 r} \int\limits_0^T \left[ \norm[\bigg]{\dfrac{\partial p}{\partial t}(t^n)}^2_{\varepsilon^{-1}} \!\! + \norm[\bigg]{\dfrac{\partial E}{\partial t}(t^n)}^2_{\varepsilon} \!\! + \norm[\bigg]{\dfrac{\partial H}{\partial t}(t^n)}^2_{\mu} \right] dt \\
= 18 C_0 h^{2 r} \left[ \norm[\bigg]{\dfrac{\partial p}{\partial t}}^2_{L^2(0, T; L^2_{\varepsilon^{-1}}(\Omega))} \!\! + \norm[\bigg]{\dfrac{\partial E}{\partial t}}^2_{L^2(0, T; L^2_{\varepsilon}(\Omega))} \!\! + \norm[\bigg]{\dfrac{\partial H}{\partial t}}^2_{L^2(0, T; L^2_{\mu}(\Omega))} \right],
\end{multline*}
 \begin{multline*}
   \dfrac{\Delta t^5}{2} \sum\limits_{n = 0}^{N-1} \left[ \norm[\bigg]{(I - \Pi_h^0) \dfrac{\partial^3 p}{\partial t^3}(t^n)}^2_{\varepsilon^{-1}} \!\! + \norm[\bigg]{(I - \Pi_h^1) \dfrac{\partial^3 E}{\partial t^3}(t^n)}^2_{\varepsilon} \!\! + \norm[\bigg]{(I - \Pi_h^2) \dfrac{\partial^3 H}{\partial t^3}(t^n)}^2_\mu \right] \\
\le C_0 h^{2 r} \sum\limits_{n = 0}^{N-1} \Delta t \left[ \norm[\bigg]{\dfrac{\partial^3 p}{\partial t^3}(t^n)}^2_{\varepsilon^{-1}} \!\! + \norm[\bigg]{\dfrac{\partial^3 E}{\partial t^3}(t^n)}^2_{\varepsilon} \!\! + \norm[\bigg]{\dfrac{\partial^3 H}{\partial t^3}(t^n)}^2_{\mu} \right] \quad (\because \Delta t < 1) \\
\le C_0 h^{2 r} \int\limits_0^T \left[ \norm[\bigg]{\dfrac{\partial^3 p}{\partial t^3}(t^n)}^2_{\varepsilon^{-1}} \!\! + \norm[\bigg]{\dfrac{\partial^3 E}{\partial t^3}(t^n)}^2_{\varepsilon} \!\! + \norm[\bigg]{\dfrac{\partial^3 H}{\partial t^3}(t^n)}^2_{\mu} \right] dt \\
= C_0 h^{2 r} \left[ \norm[\bigg]{\dfrac{\partial^3 p}{\partial t^3}}^2_{L^2(0, T; L^2_{\varepsilon^{-1}}(\Omega))} \!\! + \norm[\bigg]{\dfrac{\partial^3 E}{\partial t^3}}^2_{L^2(0, T; L^2_{\varepsilon}(\Omega))} \!\! + \norm[\bigg]{\dfrac{\partial^3 H}{\partial t^3}}^2_{L^2(0, T; L^2_{\mu}(\Omega))} \right],
\end{multline*}
and likewise for the $L^2$ projection of the remainder terms:
\begin{multline*}
 18 \Delta t \sum\limits_{n = 0}^{N-1} \left[ \norm[\bigg]{(I - \Pi_h^0) R_p^n}_{\varepsilon^{-1}} \!\! + \norm[\bigg]{(I - \Pi_h^1) R_E^n}_{\varepsilon} \!\! + \norm[\bigg]{(I - \Pi_h^2) R_H^n}_{\mu} \right], \\
  \le C_0 h^{2 r} \sum\limits_{n = 0}^{N-1}  \Delta t \left[ \norm{R_p^n}_{L^2_{\varepsilon^{-1}}(\Omega)} + \norm{R_E^n}_{L^2_{\varepsilon}(\Omega)} + \norm{R_H^n}_{L^2_{\mu}(\Omega)} \right], \\
  \le \dfrac{4}{9}C_0 h^{2 r} \Delta t^8 \left[ \norm[\bigg]{\dfrac{\partial^5 p}{\partial t^5}}^2_{L^2(0, T; L^2_{\varepsilon^{-1}}(\Omega))} \!\! + \norm[\bigg]{\dfrac{\partial^5 E}{\partial t^5}}^2_{L^2(0, T; L^2_{\varepsilon}(\Omega))} \!\! + \norm[\bigg]{\dfrac{\partial^5 H}{\partial t^5}}^2_{L^2(0, T; L^2_{\mu}(\Omega))} \right].
\end{multline*}
Now, we take $u_h^0 = \Pi_h^k u_0$, and $u_h^1 = \Pi_h^k u_1$,for $k = 0, 1, 2$ respectively, for $u = \{p, E, H\}$, and we note that there exists a positive bounded constant $M$ such that:
\begin{multline*}
  18 \left(\norm[\bigg]{\dfrac{\partial p}{\partial t}}^2_{L^2(0, T; L^2_{\varepsilon^{-1}}(\Omega))} \!\! + \norm[\bigg]{\dfrac{\partial E}{\partial t}}^2_{L^2(0, T; L^2_{\varepsilon}(\Omega))} \!\! + \norm[\bigg]{\dfrac{\partial H}{\partial t}}^2_{L^2(0, T; L^2_{\mu}(\Omega))} \right) + \norm[\bigg]{\dfrac{\partial^3 p}{\partial t^3}}^2_{L^2(0, T; L^2_{\varepsilon^{-1}}(\Omega))} \!\! \\ + \norm[\bigg]{\dfrac{\partial^3 E}{\partial t^3}}^2_{L^2(0, T; L^2_{\varepsilon}(\Omega))} \!\! + \norm[\bigg]{\dfrac{\partial^3 H}{\partial t^3}}^2_{L^2(0, T; L^2_{\mu}(\Omega))} + \dfrac{8}{9} \left(\norm[\bigg]{\dfrac{\partial^5 p}{\partial t^5}}^2_{L^2(0, T; L^2_{\varepsilon^{-1}}(\Omega))} \!\! + \norm[\bigg]{\dfrac{\partial^5 E}{\partial t^5}}^2_{L^2(0, T; L^2_{\varepsilon}(\Omega))} \!\! \right. \\ \left.+ \norm[\bigg]{\dfrac{\partial^5 H}{\partial t^5}}^2_{L^2(0, T; L^2_{\mu}(\Omega))} \right) + \dfrac{1}{504} \left( \norm[\bigg]{\dfrac{\partial^4 \left(\nabla p \right)}{\partial t^4}}^2_{L^2(0, T; L^2_{\varepsilon^{-1}}(\Omega))} \!\! + \norm[\bigg]{\dfrac{\partial^4 (\nabla \cdot E)}{\partial t^4}}^2_{L^2(0, T; L^2_\varepsilon(\Omega))} \!\! \right. \\ \left. + \varepsilon^{-1} \mu^{-1} \left(\norm[\bigg]{\dfrac{\partial^4 (\nabla \times E)}{\partial t^4}}^2_{L^2(0, T; L^2_{\varepsilon}(\Omega))} \!\! + \norm[\bigg]{\dfrac{\partial^4 (\nabla \times H)}{\partial t^4}}^2_{L^2(0, T; L^2_\mu(\Omega))} \right)\right) \leq M.
\end{multline*}
Consequently, we have the following estimate:
\[
  \norm{\eta_h^{N-1}}^2_{\varepsilon^{-1}} + \norm{\zeta_h^{N-1}}^2_{\varepsilon} + \norm{\xi_h^{N-1}}^2_{\mu} \le \widetilde{C} \left[h^{2 r} + h^{2 r} \Delta t^8 + \Delta t^8 \right],
\]
where $\widetilde{C} = M\exp(36 T)$ which then using the equivalence between $1$- and $2$-norms gives us that:
\[
  \norm{\eta_h^{N-1}}_{\varepsilon^{-1}} + \norm{\zeta_h^{N-1}}_{\varepsilon} + \norm{\xi_h^{N-1}}_{\mu} \le C_1 \left[h^{r} + h^r \Delta t^4 + \Delta t^4 \right],
\]
where $C_1 = \sqrt{3 \widetilde{C}}$. Also, there are positive bounded constants $\widetilde{C_2}$, $\widetilde{C_3}$ and $\widetilde{C_4}$ such that:
\begin{alignat*}{4}
  \norm{\eta^{N-1}}_{\varepsilon^{-1}} &= \norm{(I - \Pi_h^0) p(t^{N-1})}_{\varepsilon^{-1}} &&\le \widetilde{C_2} h^r \norm{p(t^{N-1})}_{L^2_{\varepsilon^{-1}}(\Omega)} &&\implies \norm{\eta^{N-1}}_{\varepsilon^{-1}} &&\le C_2 h^r, \\
  \norm{\zeta^{N-1}}_{\varepsilon} &= \norm{(I - \Pi_h^1) E(t^{N-1})}_{\varepsilon} &&\le \widetilde{C_3} h^r \norm{E(t^{N-1})}_{L^2_{\varepsilon}(\Omega)} &&\implies \norm{\zeta^{N-1}}_{\varepsilon} &&\le C_3 h^r, \\
  \norm{\xi^{N-1}}_\mu &= \norm{(I - \Pi_h^2) H(t^{N-1})}_\mu &&\le \widetilde{C_4} h^r \norm{H(t^{N-1})}_{L^2_{\mu}(\Omega)} &&\implies \norm{\xi^{N-1}}_\mu &&\le C_4 h^r,
\end{alignat*}
in which $C_2 = \widetilde{C_2} \norm{p(t^{N-1})}_{L^2_{\varepsilon^{-1}}(\Omega)}$, $C_3 = \widetilde{C_3} \norm{E(t^{N-1})}_{L^2_{\varepsilon}(\Omega)}$ and $C_4 = \widetilde{C_4} \norm{H(t^{N-1})}_{L^2_{\mu}(\Omega)}$ are all bounded positive constants due to Theorem~\ref{thm:dscrt_enrgy_cnsrvtn_ts4}. Finally, this provides us with our desired result by choosing $C = C_1 + C_2 + C_3 + C_4$:
\begin{align*}
  \norm{e_{p_h}^{N - 1}}_{\varepsilon^{-1}} + \norm{e_{E_h}^{N - 1}}_{\varepsilon} + \norm{e_{H_h}^{N - 1}}_{\mu} &= \norm{\eta^{N - 1} - \eta^{N -1}_h}_{\varepsilon^{-1}} + \norm{\zeta^{N - 1} - \zeta^{N - 1}_h}_{\varepsilon} + \norm{\xi^{N -1} - \xi^{N -1}_h}_\mu, \\
  &\le C \left[ (\Delta t )^4 + h^r + h^r \Delta t^4 \right]. \qedhere
\end{align*}
\end{proof}

\section{Numerical Results and Future Outlook}\label{sec:numerics}

We now present two validating computational examples in $\R^2$.

\medskip \noindent \textbf{Example 1}: This problem consists of the Maxwell's system being posed on a unit square in $\R^2$ with the analytical solution, material parameters, initial and final times as below:
\[
  p = 0, \quad E = %
  \begin{bNiceMatrix}
    \sin \pi y \cos \pi t \\
    \sin \pi x \cos \pi t
  \end{bNiceMatrix}, %
  \quad H = (\cos \pi y - \cos \pi x) \sin \pi t,
\]
\[\epsilon = 1, \quad \mu = 1, \quad T_{\min} = 0 , \quad T_{\max} = 1.\]
The computational results of solving this problem with the choices LF$_4$ and TS$_4$ in conjunction with linear or quadratic FEEC basis are all shown in Figures~\labelcref{fig:example1_lf4_linear,fig:example1_lf4_quadratic,fig:example1_ts4_linear,fig:example1_ts4_quadratic}.

\medskip \noindent \textbf{Example 2}: This problem is also on the unit square in $\R^2$ with $p \ne 0$, and with non homogeneous boundary conditions for both $p$ and $E$ and homogeneous boundary conditions for $H$.
\[
  p = \left(\cos \pi x + \cos \pi y\right) \sin \pi t,
\]
\[
 E = %
  \begin{bNiceMatrix}
    \sin \pi (\sqrt{2} t - x - y) - \sin \pi x \cos \pi t \\
    -\sin \pi (\sqrt{2} t - x - y) - \sin \pi y \cos \pi t
  \end{bNiceMatrix}, %
  \quad H = -\sqrt{2} \sin \pi (\sqrt{2} t - x - y),
\]
\[\epsilon = 1, \quad \mu = 1, \quad T_{\min} = 0 , \quad T_{\max} = 1.\]
The computational results of solving this problem with the choices LF$_4$ and TS$_4$ in conjunction with linear or quadratic FEEC basis are all shown in Figures~\labelcref{fig:example2_lf4_linear,fig:example2_lf4_quadratic,fig:example2_ts4_linear,fig:example2_ts4_quadratic}.

In the end, we wish to conclude by summarizing that we have outlined two different strategies for the fourth-order time discretization of the three-field formulation of Maxwell's equations that have resulted in two implicit time integration schemes which we have demonstrated to be energy conserving and convergent. We leave it to a future work to further generalize this to arbitrary-order schemes for the spatial strategy and a plethora of such higher-order schemes arising from the temporal strategy.

\printbibliography

\begin{figure}[htb]
  \centering
  \includegraphics[scale=1]{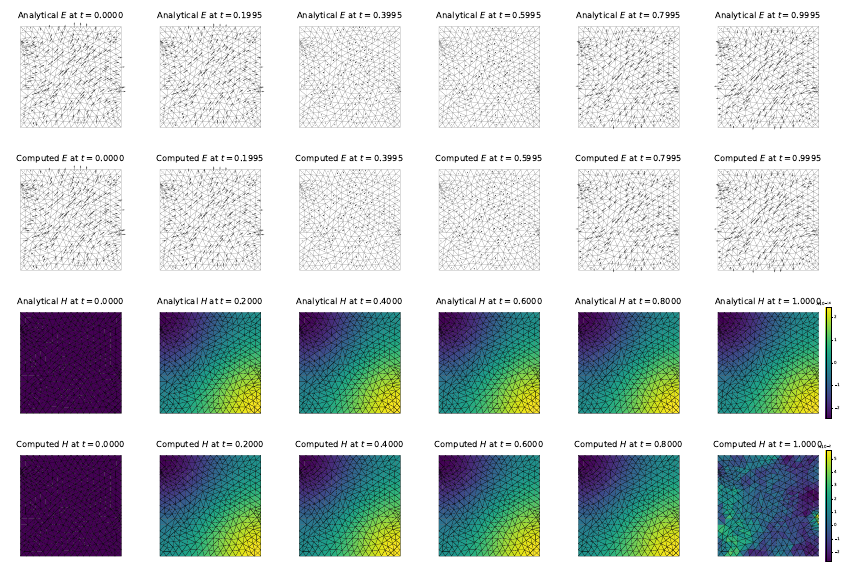}
  \caption{\textbf{Linear finite elements with LF$_4$}: Plots of solutions at different time steps for the problem described in \textbf{Example 1} of Section~\ref{sec:numerics} using the LF$_4$ and linear Whitney forms as basis for the FEEC spaces. The solutions for $p$ are not shown due to them being identically equal to $0$. The computed solutions for $E$ and $H$ visually match with the analytical solutions near identically.}
  \label{fig:example1_lf4_linear}
\end{figure}

\begin{figure}[htb]
  \centering
  \includegraphics[scale=1]{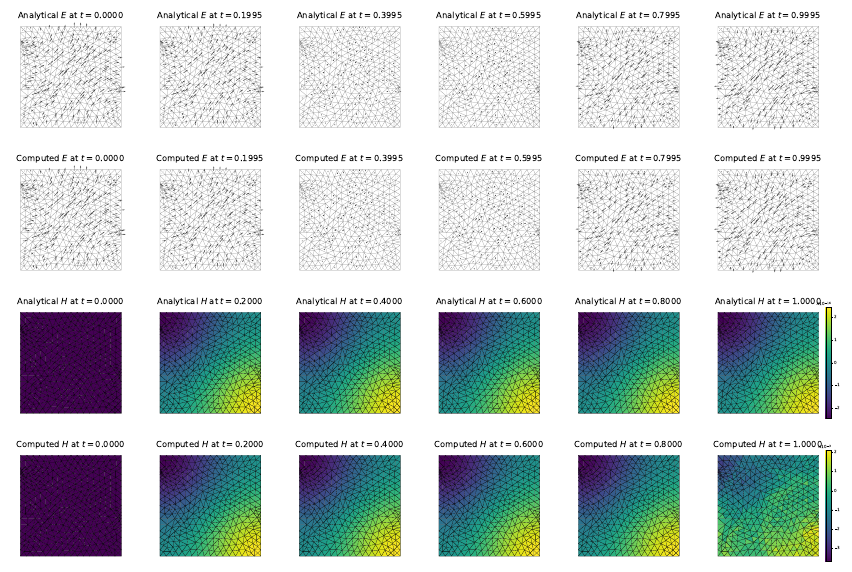}
  \caption{\textbf{Quadratic finite elements with LF$_4$}: Plots of solutions at different time steps for \textbf{Example 1} of Section~\ref{sec:numerics} using LF$_4$ and quadratic Whitney forms as basis for the FEEC spaces. The solutions for $p$ are not shown due to them being identically equal to $0$. The computed solutions for $E$ and $H$ visually match with the analytical solutions near identically.}
  \label{fig:example1_lf4_quadratic}
\end{figure}

\begin{figure}[htb]
  \centering
  \includegraphics[scale=1]{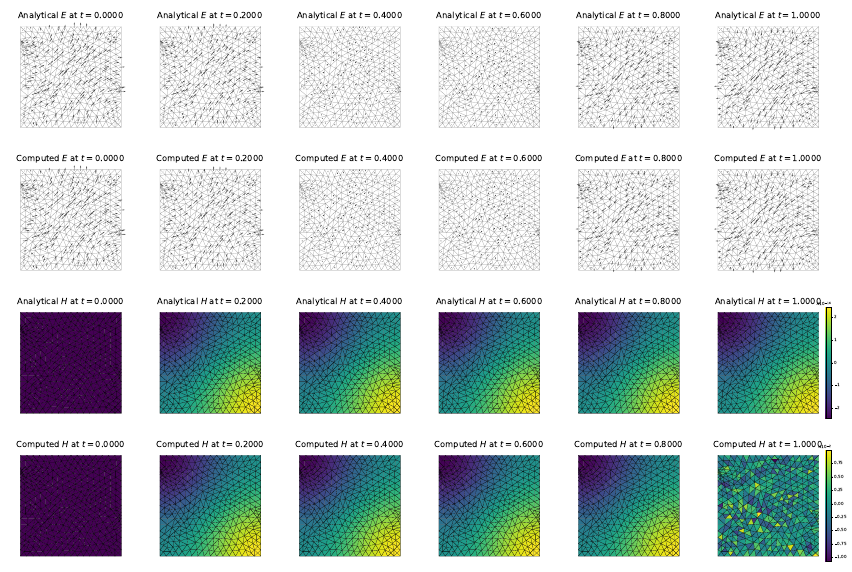}
  \caption{\textbf{Linear finite elements with TS$_4$}: Plots of solutions at different time steps for \textbf{Example 1} of Section~\ref{sec:numerics} using TS$_4$ and linear Whitney forms for the FEEC spaces. The solutions for $p$ are again not shown due to them being identically $0$.}
  \label{fig:example1_ts4_linear}
\end{figure}

\begin{figure}[htb]
  \centering
  \includegraphics[scale=1]{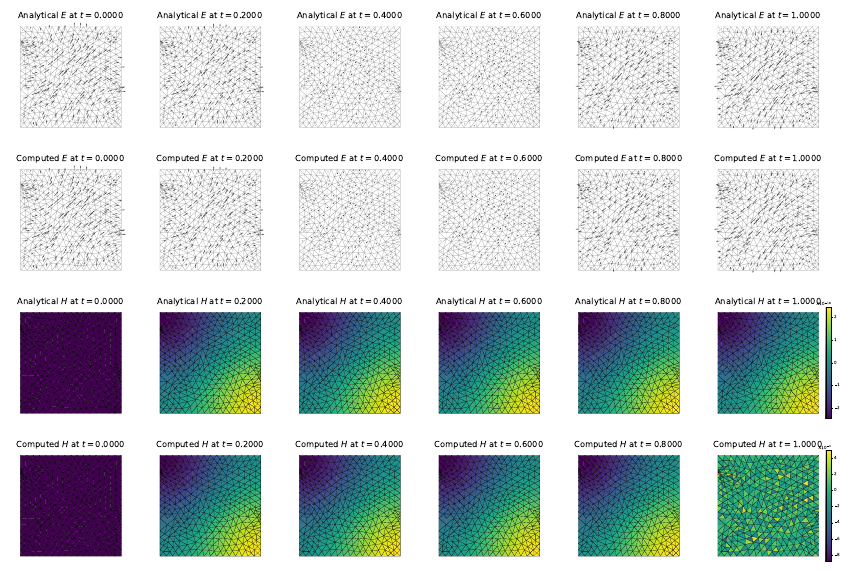}
  \caption{\textbf{Quadratic finite elements with TS$_4$}: Plots of solutions at different time steps for \textbf{Example 1} of Section~\ref{sec:numerics} using TS$_4$ and quadratic Whitney forms for the FEEC spaces.}
  \label{fig:example1_ts4_quadratic}
\end{figure}

\begin{figure}[htb]
  \centering
  \includegraphics[scale=1]{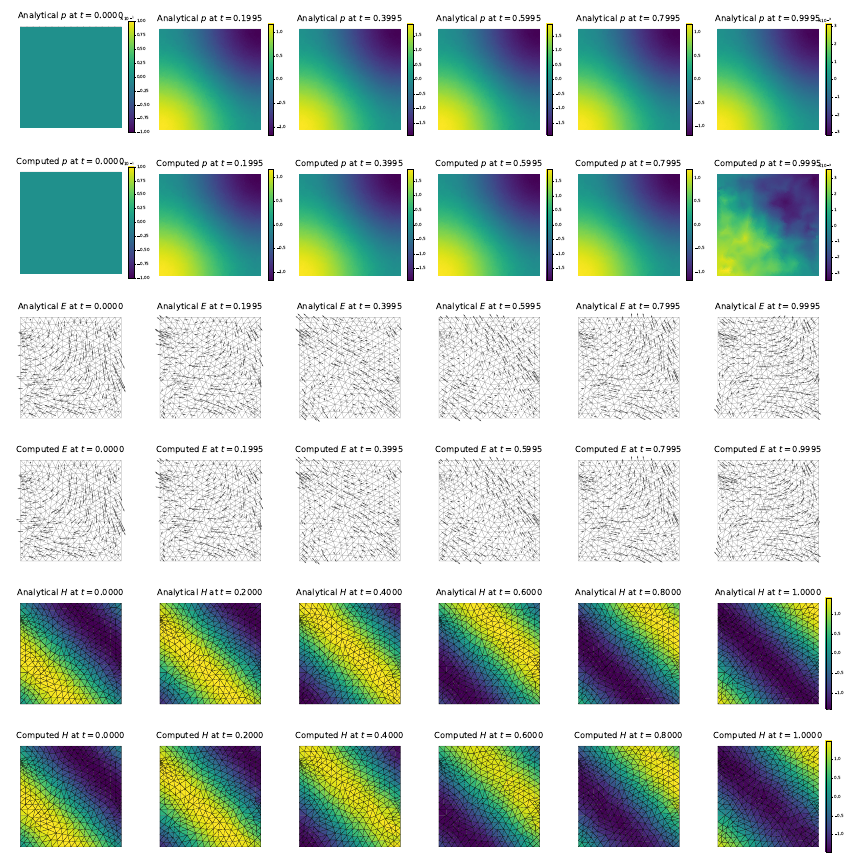}
  \caption{\textbf{Linear finite elements with LF$_4$}: Plots of solutions at different time steps for \textbf{Example 2} of Section~\ref{sec:numerics} using the LF$_4$ and linear Whitney forms in FEEC. The computed solutions for $p$, $E$ and $H$ visually match with the analytical solutions near identically.}
  \label{fig:example2_lf4_linear}
\end{figure}

\begin{figure}[htb]
  \centering
  \includegraphics[scale=1]{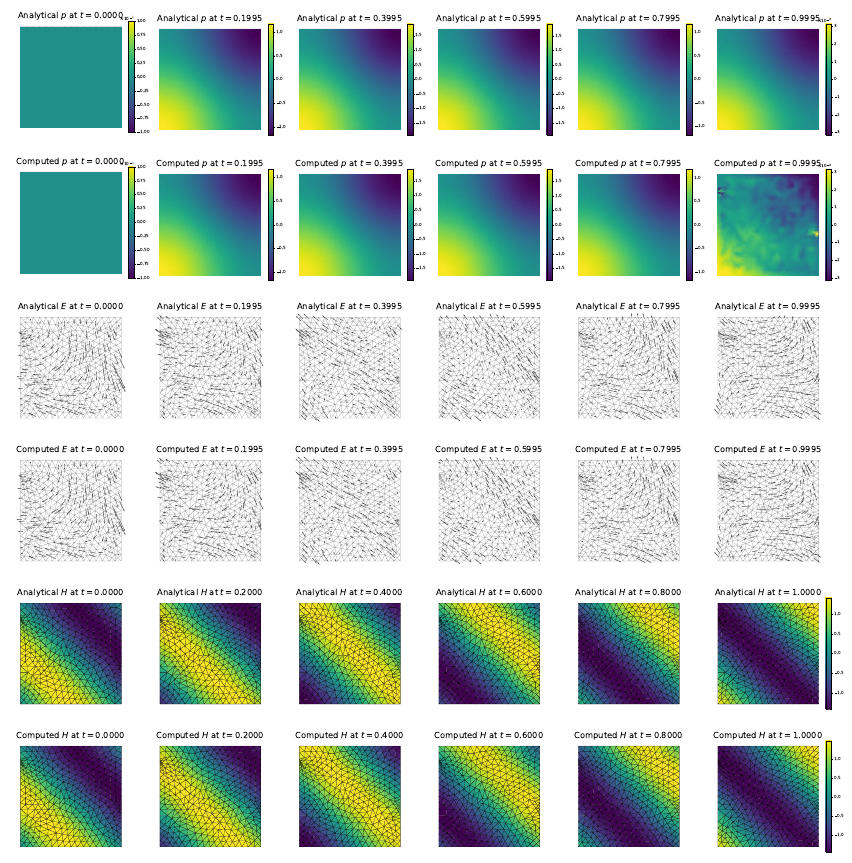}
  \caption{\textbf{Quadratic finite elements with LF$_4$}: Plots of solutions at different time steps for \textbf{Example 2} of Section~\ref{sec:numerics} using the LF$_4$ and quadratic Whitney forms in FEEC. The computed solutions for $p$, $E$ and $H$ visually match with the analytical solutions near identically.}
  \label{fig:example2_lf4_quadratic}
\end{figure}

\begin{figure}[htb]
  \centering
  \includegraphics[scale=1]{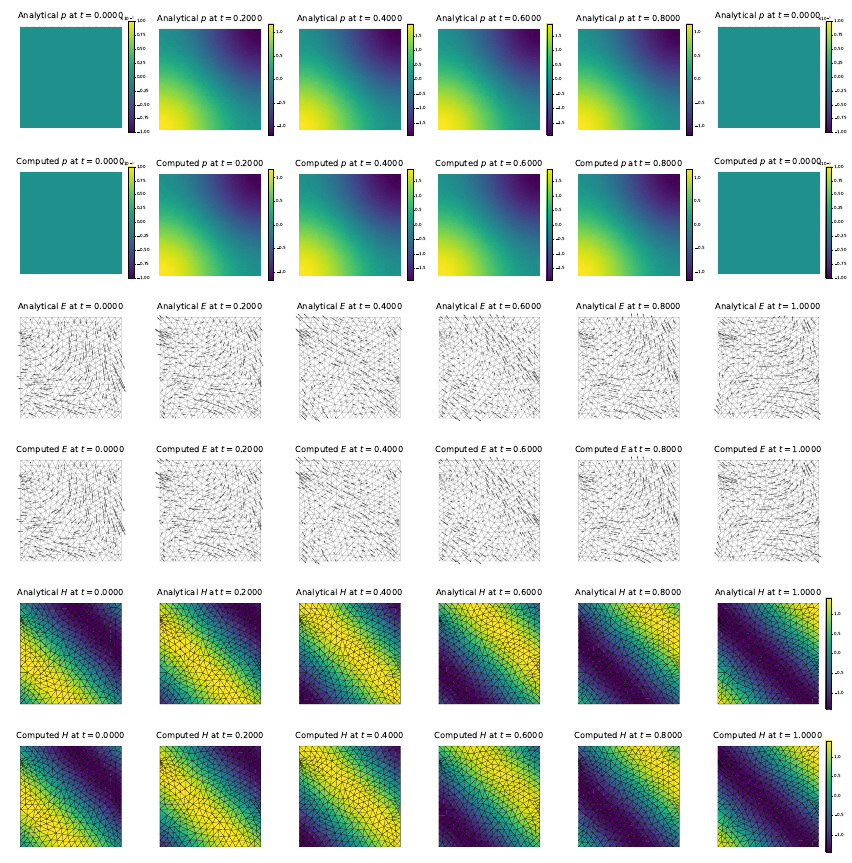}
  \caption{\textbf{Linear finite elements with TS$_4$}: Plots of solutions at different time steps for \textbf{Example 2} of Section~\ref{sec:numerics} using the TS$_4$ and linear Whitney forms in FEEC. The computed solutions for $p$, $E$ and $H$ again visually match with the analytical solutions near identically.}
  \label{fig:example2_ts4_linear}
\end{figure}

\begin{figure}[htb]
  \centering
  \includegraphics[scale=1]{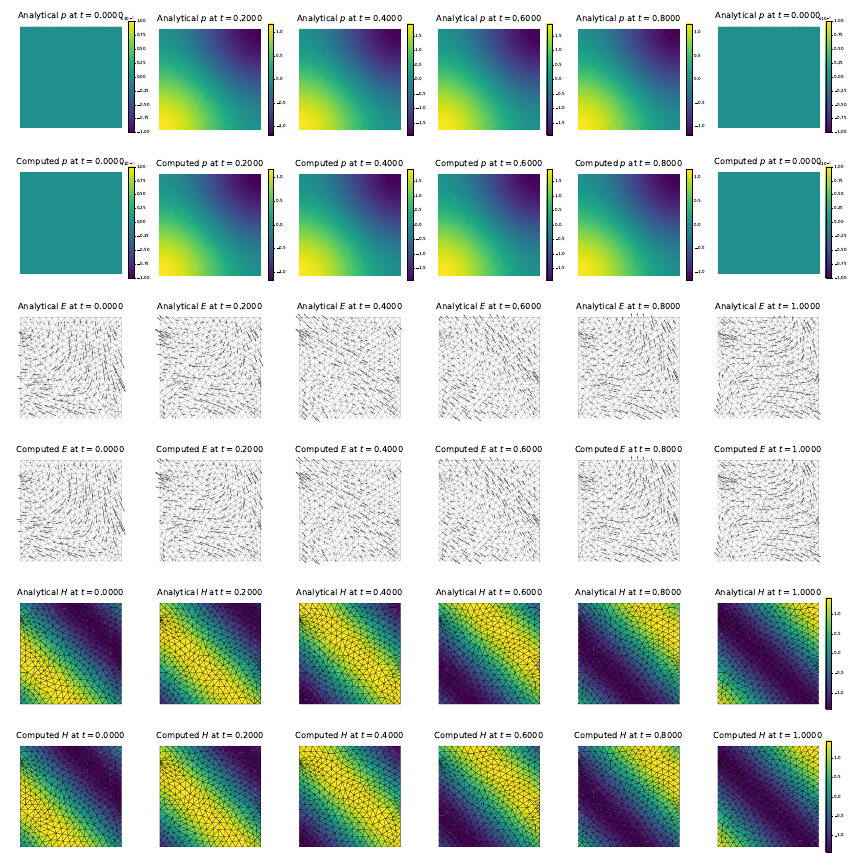}
  \caption{\textbf{Quadratic finite elements with TS$_4$}: Plots of solutions at different time steps for \textbf{Example 2} of Section~\ref{sec:numerics} using the TS$_4$ and quadratic Whitney forms in FEEC. The computed solutions for $p$, $E$ and $H$ yet again visually match with the analytical solutions near identically.}
  \label{fig:example2_ts4_quadratic}
\end{figure}

\end{document}